\documentclass[11pt]{amsart}

\usepackage[latin9]{luainputenc}
\usepackage{amsfonts}
\usepackage{amsmath}
\usepackage{amstext}
\usepackage{amsthm}
\usepackage{amssymb}
\usepackage{braket}
\usepackage{bm}
\usepackage{cancel}
\usepackage{caption}
\usepackage{color}
\usepackage{enumerate}
\usepackage{enumitem}
\usepackage{esint}
\usepackage{extarrows}
\usepackage[a4paper,margin=1in,footskip=0.25in]{geometry}
\usepackage{graphicx}
\usepackage{longtable}
\usepackage{mathrsfs}
\usepackage{mathtools}
\usepackage{mdframed}
\usepackage{soul}
\usepackage{subcaption}
\usepackage{tikz}
\usepackage{tikz-cd}
\usepackage{xcolor}
\usepackage{ulem}

\numberwithin{equation}{section}
\graphicspath{ {./images/} }
\usetikzlibrary{decorations.markings}
\tikzset{negated/.style={
        decoration={markings,
            mark= at position 0.5 with {
                \node[transform shape] (tempnode) {$\backslash$};
            }
        },
        postaction={decorate}
    }
}

\newtheorem{defi}{Definition}[section]
\newtheorem{lem}[defi]{Lemma}
\newtheorem{theo}[defi]{Theorem}
\newtheorem{cor}[defi]{Corollary}
\newtheorem{pro}[defi]{Proposition}
\newtheorem{eje}[defi]{Example}
\newtheorem{rem}[defi]{Remark}

\DeclareMathOperator{\divop}{div}

\DeclareMathOperator{\supp}{supp}
\DeclareMathOperator{\diam}{diam}
\DeclareMathOperator{\dist}{dist}

\DeclareMathOperator{\BL}{BL}
\DeclareMathOperator{\Lip}{Lip}

\DeclareMathOperator{\RR}{\mathbb{R}}

\DeclareMathOperator{\Ls}{Ls}

\newcommand{\framework}{\boldsymbol{\mathcal{F}}}

\allowdisplaybreaks 

\title[Fibered gradient flows and optimal transport]{Heterogeneous gradient flows in the topology of~fibered optimal transport}

\author{Jan Peszek}
\address{Institute of Applied Mathematics and Mechanics, University of Warsaw, ul. Banacha 2, 02-097 Warszawa, Poland}
\email{j.peszek@mimuw.edu.pl}

\author[David Poyato]{David Poyato}
\address{Departamento de Matem\'atica Aplicada and Research Unit ``Modeling Nature'' (MNat), Facultad de Ciencias, Universidad de Granada, 18071 Granada, Spain}
\email{davidpoyato@ugr.es}

\begin{document}

\date{\today}


\begin{abstract}
We introduce an optimal transport topology on the space of probability measures over a fiber bundle, which penalizes the transport cost from one fiber to another. For simplicity, we illustrate our construction in the Euclidean case $\mathbb{R}^d\times \mathbb{R}^d$, where we penalize the quadratic cost in the second component. Optimal transport becomes then  constrained to happen along fixed fibers. Despite the degeneracy of the infinitely-valued and discontinuous cost, we prove that the space of probability measures $(\mathcal{P}_{2,\nu}(\mathbb{R}^{2d}),W_{2,\nu})$ with fixed marginal $\nu\in \mathcal{P}(\mathbb{R}^d)$ in the second component becomes a  Polish space under the fibered transport distance, which enjoys a weak Riemannian structure reminiscent of the one proposed by F. Otto for the classical quadratic Wasserstein space. Three fundamental  issues are addressed: 1) We develop an abstract theory of gradient flows with respect to the new topology; 2) We show applications that identify a novel fibered gradient flow structure on a large class of evolution PDEs with heterogeneities; 3) We exploit our method to derive long-time behavior and global-in-time mean-field limits in a multidimensional Cucker-Smale-type alignment model with weakly singular coupling.
\end{abstract}

\maketitle

\section{Introduction}\label{sec:intro}
It is well known that gradient flows in the space of probability measures ${\mathcal P}(\RR^d)$ equipped with the quadratic Wasserstein metric $W_2$ are closely related to continuity equations. Indeed, as discovered by {\sc F. Otto} \cite{O-01}, the quadratic Wasserstein space $(\mathcal{P}_2(\mathbb{R}^d),W_2)$ is a formal infinite-dimensional Riemannian manifold. Therefore, under suitable assumptions, any gradient flow of a functional ${\mathcal E}:\mathcal{P}_2(\mathbb{R}^d)\longrightarrow (-\infty,+\infty]$ must solve the continuity equation
 \begin{align*}
 \begin{aligned}
 &\partial_t\mu + {\rm div}(\boldsymbol{u}\mu) = 0, & & t\geq 0,\,x\in \mathbb{R}^d,\\
 &-\boldsymbol{u}_t\in\partial_{W_2} {\mathcal E}[\mu_t], & & \mbox{a.e.}\ t\geq 0,
 \end{aligned}
 \end{align*}
in distributional sense, where $\mu_t=\mu_t(x)\in \mathcal{P}_2(\mathbb{R}^d)$, $\boldsymbol{u}_t=\boldsymbol{u}_t(x)\in L^2_{\mu_t}(\mathbb{R}^d,\mathbb{R}^d)$, and $\partial_{W_2}{\mathcal E}[\mu_t]$ is the Fr\' echet subdifferential of ${\mathcal E}$ with respect to such a Riemannian structure, see \cite{AGS-08,G-08-thesis}. 

If we further restrict to regular enough functionals, then the Fr\' echet subdifferential simply reduces to a single element, that is the Fr\'echet gradient. Indeed, for smooth functionals entirely defined over absolutely continuous measures $\mu\in \mathcal{P}_{2,ac}(\mathbb{R}^d)$, we have
$$\partial_{W_2}\mathcal{E}[\mu]=\left\{\nabla \frac{\delta \mathcal{E}}{\delta \rho}\right\},$$
where  $\mu=\rho\,dx$, $\rho\in L^1_+(\mathbb{R}^d)$, and $\frac{\delta \mathcal{E}}{\delta \rho}$ is the usual Euler first variation of $\mathcal{E}$ with respect to $\rho$ (when it exists), see \cite[Section 10.4]{AGS-08}. Then, we can formally  simplify the above abstract gradient flows to the following class of PDEs
 $$\partial_t\rho-\divop\left(\rho\,\nabla \frac{\delta \mathcal{E}}{\delta\rho}\right)=0,\quad t\geq 0,\,x\in \mathbb{R}^d.$$
 
As special cases, we find energy functionals with three distinguished components representing the internal energy $U=U(\rho)$, the external potential energy $V=V(x)$, and the pairwise interaction potential energy $W=W(x)$. Specifically, $\mathcal{E}$ is defined by
\begin{equation}\label{E-intro-energies}
\mathcal{E}[\mu]=\int_{\mathbb{R}^d} U(\rho(x))\,dx+\int_{\mathbb{R}^d} V(x)\,d\mu(x)+\frac{1}{2}\iint_{\mathbb{R}^d\times \mathbb{R}^d}W(x-x')\,d\mu(x)\,d\mu(x'),
\end{equation}
if $\mu=\rho\,dx\in \mathcal{P}_{2,ac}(\mathbb{R}^d)$, and $\mathcal{E}[\mu]=+\infty$ otherwise. These have been studied extensively over the last decades as they allowed setting the basis for the rigorous gradient flow formulation of a wide range of (nonlinear) integro-differential equations ranging from the porous medium equation, to the aggregation-diffusion equation, see \cite{CDFLS-11,JKO-98,O-99,O-01} and references therein.  
 
The underlying objective of this paper is to present a novel optimal transport topology in the space of probability measures ${\mathcal P}(\RR^d)$ that naturally introduces heterogeneity to gradient flows. Specifically, consider $d=d_1+d_2$ so that our space $\RR^d=\mathbb{R}^{d_1}\times \mathbb{R}^{d_2}$ can be regarded as a fiber bundle (the trivial one) with base space parametrized by $\omega\in \mathbb{R}^{d_2}$ and fiber space parametrized by $x\in \mathbb{R}^{d_1}$. Let us now set $\mu_1,\mu_2\in \mathcal {P}(\RR^{d_1+d_2})$, with a common marginal $\nu\in{\mathcal P}(\RR^{d_2})$ with respect to $\omega$. In a nutshell, our optimal transport topology shall be built as to penalize the transport from $\mu_1$ to $\mu_2$ whenever it happens across different fibers $\omega\in \mathbb{R}^{d_2}$. Namely, if $\omega_1,\omega_2\in \mathbb{R}^{d_2}$ represent two different fibers $\omega_1\neq \omega_2$, then mass of $\mu_1$ at the fiber $\omega_1$ requires an infinite cost to be moved to mass of $\mu_2$ at the fiber $\omega_2$. This penalization then constrains the optimal transport to take place in a fibered way along $\omega = \mbox{const}$, which in particular preserves the common marginal $\nu$. More specifically, we define a natural metric $W_{2,\nu}$ over the space $\mathcal{P}_\nu(\mathbb{R}^{d_1+d_2})$ of probability measures with fixed marginal $\nu$, which encodes the above penalization mechanism, and we shall refer to it as the {\it fibered Wasserstein distance}.
 
We show that $(\mathcal{P}_{2,\nu}(\mathbb{R}^{d_1}\times \mathbb{R}^{d_2}),W_{2,\nu})$ admits a Riemannian structure, and we develop a variant of the classical theory of gradient flows of functionals $\mathcal{E}:\mathcal{P}_{2,\nu}(\mathbb{R}^{d_1+d_2})\rightarrow (-\infty,+\infty]$ with respect to $W_{2,\nu}$. As we prove later, this construction leads to ``heterogeneous gradient flows'' with parameters controlled through $\nu$, which obey the continuity equation 
 \begin{align}\label{kurak}
 \begin{aligned}
 &\partial_t\mu + \divop_x(\boldsymbol{u}\mu) = 0, & & t\geq 0,\,(x,\omega)\in \mathbb{R}^{d_1}\times\mathbb{R}^{d_2},\\
 &-\boldsymbol{u}_t\in\partial_{W_{2,\nu}} {\mathcal E}[\mu_t], & & \mbox{a.e.}\ t\geq 0,
  \end{aligned}
 \end{align}
  where $\mu_t=\mu_t(x,\omega)\in \mathcal{P}_{2,\nu}(\mathbb{R}^{d_1+d_2})$, $\boldsymbol{u}_t=\boldsymbol{u}_t(x,\omega)\in L^2_{\mu_t}(\mathbb{R}^{d_1+d_2},\mathbb{R}^{d_1})$, and $\partial_{W_{2,\nu}} {\mathcal E}$ is a novel fibered variant of the Fr\'echet subdifferential that is introduced rigorously in the sequel. Similarly to the classical setting, for regular enough energy functionals entirely defined in the class of absolutely continuous measures $\mu\in \mathcal{P}_{2,\nu,ac}(\mathbb{R}^{d_1+d_2})$ ({\it i.e.}, measures admitting a disintegration $\mu(x,\omega)=\rho(x,\omega)\,dx\otimes \nu(\omega)$), equation \eqref{kurak} takes the equivalent form
\begin{equation}\label{kurak-variation}
\partial_t\rho-\divop_x\left(\rho\,\nabla_x\frac{\delta\mathcal{E}}{\delta \rho}\right)=0,\quad t\geq 0,\,(x,\omega)\in \mathbb{R}^{d_1+d_2}.
\end{equation}
We emphasize that, while \eqref{kurak} readily implies \eqref{kurak-variation} for regular enough functionals (so that, in particular, $\partial_{W_{2,\nu}}\mathcal{E}[\mu]=\{\nabla_x\frac{\delta\mathcal{E}}{\delta\rho}\}$), the converse is true only when restricted to the above class of absolutely continuous measures with fixed marginal $\nu$ (see examples in Remark \ref{R-examples-from-intro}). Note that the divergence in \eqref{kurak} has intentionally been taken only with respect to $x\in\RR^{d_1}$. Roughly speaking, it originates in the above idea that mass transportation is constrained to happen along the fibers of $\RR^{d_1+d_2}$ with $\omega=\mbox{const}$. However, the velocity $\boldsymbol{u}_t=\boldsymbol{u}_t(x,\omega)$ can actually depend on $\omega$ in a convoluted way, which is encoded in $\partial_{W_{2,\nu}} {\mathcal E}$. This means that the velocity could couple the effects of different (or all) fibers so that \eqref{kurak} cannot be treated simply as a family of independent continuity equations parametrized by $\omega$, see Remark \ref{ohhnaivety}.
 
 The goal of the paper is threefold.
First, we develop an abstract theory of gradient flows with respect to the fibered optimal transport topology, for which a Riemannian structure is identified in the spirit of Otto calculus. Second, we investigate and showcase possible applications of fibered gradient flows to various situations in the literature. Third, we provide new and meaningful insight into a special system (multidimensional Cucker-Smale-type alignment dynamics with weakly singular coupling) by using the language of fibered gradient flows.

Other variants of topologies based on fibered optimal transport have been investigated in the past. In the literature of optimal transport, {\sc N. Gigli} used one in \cite{G-08-thesis} as an auxiliary tool in the construction of the geometric tangent space to the quadratic Wasserstein space $({\mathcal P}_2(\RR^d),W_2)$, see also \cite{G-11} and \cite{AGS-08}. In the literature of collective dynamics, it was used in \cite{CCHKK-14,CZ-21,K-22} to prove contractivity respectively for the Kuramoto-Sakaguchi equation (with restricted initial data and large coupling strength), the 1D kinetic Cucker-Smale model with weakly singular couplings, and a variant of the multidimensional Cucker-Smale model with a Hessian communication weight. Recently, it has also been proposed in \cite{K-21-1,K-21-2} with applications to the existence and uniqueness of stationary measures of iterated function systems. The exact formulation of the metric we chose to use was established in \cite{P-19-arxiv} followed by \cite{MP-22} in the study of large-time behavior of the Kuramoto-Sakaguchi equation (with generic initial data and large coupling strength). The main contribution of the present paper is not in the introduction of the topology but in the robust study of its applicability in gradient flows and their applications.

 A rigorous presentation of the main goals of the paper requires an in-depth setup of preliminary information, which we provide later. In the subsequent sections we present a brief overview of the main results without getting into the technical details. The remainder of the paper is organised as follows. Section \ref{sec:prelim} is dedicated to a review of classical optimal transport theory. Section \ref{sec:fibered-wasserstein} is the main part of the paper wherein we introduce the fibered Wasserstein distance and meticulously construct fibered gradient flows. In Section \ref{sec:exist} we present  concrete examples of fibered gradient flows associated with three special energy functionals reminiscent of the classical ones in \eqref{E-intro-energies}. We also illustrate a variety of examples arising in the recent literature of collective dynamics ({\it cf.} Section \ref{maingoal2}), where this theory could provide a satisfying functional framework. Section \ref{apple} is dedicated to applications to the Cucker-Smale-type alignment model. Finally, Appendices \ref{Appendix-metric-valued-L2}, \ref{Appendix-slope-subdifferential} and \ref{Appendix-convexity-W} contain some auxiliary results and proofs to alleviate the reading of the main sections.
 
\begin{rem}\label{R-intro-choice-dimensions}
For simplicity of the presentation, we restrict to the case where the dimension of the base space $\RR^{d_2}$ and fiber space $\RR^{d_1}$ agree, {\it i.e.}, $d_1=d_2=d$. This results in a sufficient setting to tackle some PDEs arising in collective dynamics like the Kuramoto-type reformulation of the Cucker-Smale alignment dynamics that we address in this paper. However, our theory stays valid for $d_1\neq d_2$ or even more general base spaces and fiber spaces in the fiber bundle which are not necessarily Euclidean, see examples in Section \ref{maingoal2} below.
\end{rem}
 
 \subsection{Main goal 1: abstract theory of fibered gradient flows}\label{maingoal1}

The primary goal of the paper is to build an abstract theory of fibered gradient flows. To this end we follow the key steps of the classical theory of gradient flows associated with $\lambda$-convex functionals on metric spaces presented by {\sc L. Ambrosio}, {\sc N. Gigli} and {\sc G. Savar\' e} \cite{AGS-08}, see also the textbooks by {\sc C. Villani} \cite{V-09} and {\sc F. Santambrogio} \cite{S-15}. We begin by proving the basic properties of the fibered optimal transport such as stability of optimality, general differentiability of the fibered Wasserstein distance and the relation between different types of convergence. Then we establish a fibered analog of the weak Riemannian structure for the space $({\mathcal P}_2(\RR^d),W_2)$ due to {\sc F. Otto}, followed by fibered subdifferential calculus and three notions of fibered gradient flows, which become equivalent under suitable $\lambda$-convexity assumptions. Finally, we prove existence of fibered gradient flows, which follows directly from the classical theory of solutions to the Evolutionary Variational Inequality (EVI) valid on any Polish space. Our efforts towards this first goal culminates in the following formal meta-theorem.

\medskip

\begin{mdframed}
{\bf Theorem A.} Consider any functional $\mathcal{E}:\mathcal{P}_{2,\nu}(\mathbb{R}^d)\longrightarrow (-\infty,+\infty]$ for any $\nu\in \mathcal{P}(\mathbb{R}^d)$ satisfying suitable assumptions, which include the $\lambda$-convexity of ${\mathcal E}$ (see framework ${\mathcal F}$ in Definition \ref{D-framework-F} below). Then, there exists a unique, locally Lipschitz and exponentially stable gradient flow $t\mapsto\mu_t\in \mathcal{P}_{2,\nu}(\mathbb{R}^{2d})$ of ${\mathcal E}$ starting at any initial datum $\mu_0\in D(\mathcal{E})$, which in particular satisfies Equation \eqref{kurak} in the sense of distributions. A rigorous formulation of Theorem A can be found in Section \ref{sec:sub_calc_grad_flows}, Theorem \ref{t-A-rig}.
\end{mdframed}

\medskip

It is worthwhile to clarify that,  whilst widely assumed in the literature, the $\lambda$-convex framework that we employ is not necessary for the existence of gradient flows. However, it is essential for the uniqueness and the equivalence of the various notions of gradient flows, and thus we decided to use such a stronger framework for the sake of consistency, {\it cf.} Remark \ref{R-why-lambda-convexity}.

\begin{rem}\label{ohhnaivety}
 There could be a seemingly easy approach to this first goal. Namely, we may be tempted to reduce any problem on $(\mathcal{P}_{2,\nu}(\mathbb{R}^{2d}),W_{2,\nu})$ into a separate family of similar problems on the classical space $(\mathcal{P}_2(\mathbb{R}^d),W_2)$ for each $\omega$-fiber. Then, we may solve each of them separately via the classical theory and integrate with respect to $\omega$. This idea is further supported by the fact that any curve $t\mapsto \mu_t=\mu_t(x,\omega)$ solving \eqref{kurak} can be disintegrated with respect to the fixed marginal $\nu$, thus producing a $\omega$-indexed family of continuity equations
$$\partial_t \mu^\omega  + \divop_x(\boldsymbol{u}_t(\cdot,\omega) \mu^\omega) = 0,\quad t\geq 0,\ x\in \mathbb{R}^d,$$
for $\nu$-a.e. $\omega\in \mathbb{R}^d$, where $\mu_t(x,\omega)=\mu_t^\omega(x)\otimes \nu(\omega)$ (see Disintegration Theorem \ref{dis}). Whilst useful in many questions, this approach does not work in a few critical problems. One of the main factors contributing to the failure of such a strategy is that we cannot ensure that $\boldsymbol{u}_t(\cdot,\omega)$ depends only on $\mu^\omega$. This issue is by design, since otherwise all fibers amounting to $\mu$ in \eqref{kurak} would decouple. In all the examples presented in the forthcoming sections $\boldsymbol{u}$ actually depends on all fibers of the solution $\mu$.  On the technical level, another crucial factor is that narrow convergence is not stable under disintegration. Specifically, if $\mu_n$ have fixed marginal $\nu$ and converge narrowly to $\mu$, then we cannot ensure that $\mu_n^\omega$ also converge narrowly to $\mu^\omega$ (see Remark \ref{R-narrow-vs-fibers}). Similar problems appear throughout the paper and justify our efforts to carefully walk through most of the classical theory presented in \cite{AGS-08}.
 \end{rem}

 \subsection{Main goal 2: Applications to first-order models}\label{maingoal2}
 
The second goal of the paper is to explore possible applications of fibered gradient flows. Generally speaking, fibered gradient flows add heterogeneity to otherwise homogeneous classical gradient flows. Thus, they are useful as a tool to describe any phenomena that has a gradient flow structure, but is heterogeneous in nature, {\it e.g.} it involves multiple co-existing forces that influence separate parts of the system or it involves multiple species interacting with each other.

\subsubsection{Fibered gradient flows with additive heterogeneity}

$\,$
\medskip

\noindent $(i)$ {\bf The Kuramoto-Sakaguchi equation}. The Kuramoto model was originally introduced by {\sc Y. Kuramoto} in \cite{K-75} as a coupled system of ODEs describing a simple synchronization mechanism of a finite ensemble of $N\in \mathbb{N}$ coupled oscillators with phases $\theta_i\in \mathbb{R}$ and natural frequencies $\omega_i\in \mathbb{R}$, which obey the following coupled system of ODEs
$$\dot{\theta}_i=\omega_i+\frac{K}{N}\sum_{j=1}^N\sin(\theta_j-\theta_i),\qquad i=1,\ldots,N.$$
When $N\rightarrow \infty$, as mean-field limit we obtain the kinetic Kuramoto model (or Kuramoto-Sakaguchi equation), see \cite{L-05}. It takes the form of the following Vlasov-type PDE for the probability distribution $\mu=\mu_t(\theta,\omega)$ of oscillators
\begin{align}\label{E-Kuramoto-Sakaguchi}
\begin{aligned}
&\partial_t \mu+\partial_\theta\left(\boldsymbol{u}[\mu]\mu\right)=0,\qquad t\geq 0,\quad (\theta,\omega)\in \mathbb{R}^2,\\
&\boldsymbol{u}[\mu](t,\theta,\omega):=\omega+K\int_{\mathbb{R}^2}\sin(\theta'-\theta)\,d\mu_t(\theta',\omega').
\end{aligned}
\end{align}
It is clear that the dynamics in \eqref{E-Kuramoto-Sakaguchi} preserves the second marginal $\nu=\pi_{\omega\#}\mu_t$. In addition, $t\mapsto \mu_t$ is the fibered gradient flow of the energy functional
$$\mathcal{E}[\mu]:=-\int_{\mathbb{R}^2}\theta\,\omega\, d\mu(\theta,\omega)-\frac{K}{2}\iint_{\mathbb{R}^2\times \mathbb{R}^2}\cos(\theta-\theta')\,d\mu(\theta,\omega)\,d\mu(\theta',\omega'),$$
with $\mu\in \mathcal{P}_{2,\nu}(\mathbb{R}^2)$. Namely, writing $\mu(\theta,\omega)=\rho(\theta,\omega)\,d\theta\otimes \nu(\omega)$ we have $\boldsymbol{u}[\mu]=-\partial_\theta\frac{\delta\mathcal{E}}{\delta \rho}$, see Remark \ref{R-examples-from-intro}.

This justifies the approach in \cite{MP-22}, where an entropy method was implemented by the second author to derive the long-time asymptotics of \eqref{E-Kuramoto-Sakaguchi} in the large-coupling strength regime $K\gg 1$. Whilst the present fibered gradient-flow reformulation was not available at that moment, the authors found novel Talagrand and generalized log-Sobolev inequalities involving the fibered Wasserstein distance $W_{2,\nu}$, the above entropy functional $\mathcal{E}$ and the usual entropy-entropy dissipation relation for gradient flows, namely
$$\frac{d}{dt}\mathcal{E}[\mu_t]=-\mathcal{D}[\mu_t],\qquad \mathcal{D}[\mu_t]=\int_{\mathbb{R}^2}\vert \boldsymbol{u}[\mu_t](\theta,\omega)\vert^2\,d\mu_t(\theta,\omega).$$
We remark that $\mathcal{E}$ is however not convex, so that the entropy production argument in \cite{MP-22} is still essential under the new glasses to quantify the long-time behavior.

\medskip

\noindent $(ii)$ {\bf The kinetic Lohe matrix model}. In the above example, we have restricted to an Euclidean setting and, in addition, the dimension of the fiber and base space spaces agree. However, this is not totally necessary as anticipated in Remark \ref{R-intro-choice-dimensions}. To illustrate it, we consider the Lohe matrix model, which was proposed by {\sc M. A. Lohe} \cite{L-10} as a high-dimensional generalization of the Kuramoto model for the dynamics of a finite network of quantum coupled oscillators
$$
{\rm i}\,\dot{U}_i\,U_i^\dagger=H_i+\frac{{\rm i}\,K}{2N}\sum_{j=1}^N (U_j U_i^\dagger-U_i\,U_j^\dagger),\quad i=1,\ldots,N.
$$
States are unitary matrices $U_i\in \boldsymbol{U}(d)$ evolving under Hermitian Hamiltonians $H_i$, {\it i.e.}, $H_i^\dagger=H_i$ for the Hermitian conjugate $\dagger$. Arguing as in \cite{GH-19,HKR-18}, lifting the system to the universal covering $\mathbb{R}\times \boldsymbol{SU}(d)$ and assuming the simpler case where $H_i=-\omega_i I_d$ with $\omega_i\in \mathbb{R}$ we infer that the mean-field limit as $N\rightarrow\infty$ obeys the following Vlasov-type PDE for $\mu=\mu_t(\theta,V,\omega)$
\begin{align}\label{E-kinetic-Lohe}
\begin{aligned}
&\partial_t \mu+\partial_\theta(\boldsymbol{u}_\theta[\mu]\,\mu)+\divop_V\left(\boldsymbol{u}_V[\mu]\mu\right)=0,\qquad  t\geq 0,\quad (\theta,V,\omega)\in \mathbb{R}\times \boldsymbol{SU}(d)\times \mathbb{R},\\
&\boldsymbol{u}_\theta[\mu](t,\theta,V,\omega):=\omega+K\int_{\mathbb{R}\times \boldsymbol{SU}(d)\times \mathbb{R}} {\rm Im}\, {\rm tr}\left(V'\,V^\dagger e^{i(\theta'-\theta)}\right)\,d\mu_t(\theta',V',\omega'),\\
&\boldsymbol{u}_V[\mu](t,\theta,V,\omega):=\bigg[\frac{K}{2}\int_{\mathbb{R}\times \boldsymbol{SU}(d)\times \mathbb{R}} \bigg(V'\,V^\dagger e^{i(\theta'-\theta)}-V\,{V'}^\dagger e^{i(\theta-\theta')}\\
&\hspace{5.2cm}-\frac{1}{d}{\rm tr}(V'\,V^\dagger e^{i(\theta'-\theta)}-V\,{V'}^\dagger e^{i(\theta-\theta')})I_d\bigg)\,d\mu_t(\theta',V',\omega')\bigg]V.
\end{aligned}
\end{align}
The usual covering map is given by $(\theta,V)\in \mathbb{R}\times \boldsymbol{SU}(d)\mapsto U=e^{i\theta} V\in \boldsymbol{U}(d)$. Here $\boldsymbol{SU}(d)$ is the Lie group of the special unitary matrices, whose Lie algebra $\mathfrak{su}(d)$ are the traceless skew-Hermitian matrices. For convenience, the tangent space $\mathbb{R}\times \mathfrak{su}(d)$ at the unity $(1,I_d)$ is endowed with the usual inner product in $\mathbb{R}\times \mathbb{C}^{d\times d}$ with the $\RR$ component multiplied by $d$. Again, $\nu=\pi_{\omega\#}\mu_t$ is conserved, and $t\mapsto \mu_t$ is identified as the fibered gradient flow of
\begin{align*}
\mathcal{E}[\mu]:=&-d\int_{\mathbb{R}\times \boldsymbol{SU}(d)\times \mathbb{R}} \theta\,\omega\,d\mu(\theta',V',\omega')\\
&-\frac{K}{2}\iint_{(\mathbb{R}\times \boldsymbol{SU}(d)\times \mathbb{R})^2} {\rm tr}(e^{i(\theta-\theta')}V\,{V'}^\dagger)\,d\mu(\theta,V,\omega)\,d\mu(\theta',V',\omega'),
\end{align*}
with $\mu\in \mathcal{P}_{2,\nu}(\mathbb{R}\times \boldsymbol{SU}(d)\times \mathbb{R})$. Indeed, if $\mu(\theta,V,\omega)=\rho(\theta,V,\omega)\,d\theta\,dV\otimes \nu(\omega)$ (being $dV$ the Haar measure of $\boldsymbol{SU}(d)$), then $\boldsymbol{u}_\theta[\mu]=-\partial_\theta\frac{\delta\mathcal{E}}{\delta \rho}$ and $\boldsymbol{u}_V[\mu]=-\nabla_V\frac{\delta\mathcal{E}}{\delta \rho}$, see Remark \ref{R-examples-from-intro}.

\subsubsection{Fibered gradient flows with multiplicative heterogeneity}

$\,$

\medskip

\noindent $(i)$ {\bf Non-exchangeable systems with heterogeneous weights}. Non-exchangeable multi-agent systems are ubiquitous in nature. Whilst they can occur under various forms, a sufficiently simple setting can be described by a system of coupled SDEs
$$d X_i=-\sum_{j=1}^N \alpha_{ij}\,\nabla W(X_i-X_j)\,dt+\sqrt{2\sigma}\,dW_i,\qquad i=1,\ldots,N.$$
It describes the stochastic motion of an ensemble of $N$ particles with positions $X_i\in \mathbb{R}^d$ interacting according to a conservative force $F=-\nabla W\in (W^{1,1}\cap W^{1,\infty})(\mathbb{R}^d)$, heterogeneous couplings modulated by weights $\alpha_{ij}$, and under the influence of independent Brownian noise $dW_i$. Recently, the rigorous mean-field limit as $N\rightarrow \infty$ has been derived under \cite{CM-19,JPS-21-arxiv,KX-21-arxiv}, leading to the following Vlasov-McKean-type PDE for the distribution of agents $\mu=\mu_t(x,\omega)$
\begin{align}\label{E-kinetic-nonexchangeable}
\begin{aligned}
&\partial_t\mu +\divop_x(\boldsymbol{u}[\mu]\,\mu)=0,\qquad t\geq 0,\quad x\in \mathbb{R}^d,\quad \omega\in [0,1],\\
&\boldsymbol{u}[\mu](t,x,\omega):=-\int_{\mathbb{R}^d}\int_0^1\alpha(\omega,\omega')\,\nabla W(x-x')\,d\mu_t(x',\omega')-\sigma\,\nabla_x\log\rho_t(x,\omega),
\end{aligned}
\end{align}
where $\mu_t(x,\omega)=\rho_t(x,\omega)\,dx\otimes \nu(\omega)$. Here, $\omega\in [0,1]$ list the various types of agents and the function $\alpha=\alpha(\omega,\omega')$ corresponds to a continuous version of weights $\alpha_{ij}$. From the above literature, we have $\nu=d\omega_{\lfloor [0,1]}$ and $\alpha\in L^\infty([0,1]^2)$ for dense graphs \cite{CM-19}, whilst $\nu\in \mathcal{P}([0,\ 1])$ and $\alpha\in L^\infty_\xi([0,1],L^1_{\xi'}([0,1],\nu))\cap L^\infty_{\xi'}([0,1],L^1_{\xi}([0,1],\nu))$ for sparse ones \cite{JPS-21-arxiv}. Interestingly, when both $W$ and $\alpha$ are symmetric we can represent $t\mapsto \mu_t$ as the fibered gradient flow of
$$\mathcal{E}[\mu]:=\frac{1}{2}\iint_{(\mathbb{R}^d\times [0,1])^2}\alpha(\omega,\omega')\,W(x-x')\,d\mu(x,\omega)\,d\mu(x',\omega')+\sigma\int_{\mathbb{R}^d\times [0,1]}\log\rho(x,\omega)\,d\mu(x,\omega),$$
with $\mu=\rho(x,\omega)\,dx\otimes \nu(\omega) \in\mathcal{P}_{2,\nu}(\mathbb{R}^d\times [0,1])$ and $\mathcal{E}[\mu]=+\infty$ otherwise. Indeed, we have $\boldsymbol{u}[\mu]=-\nabla_x \frac{\delta\mathcal{E}}{\delta\rho}$, see Remark \ref{R-examples-from-intro}. In this case, the symmetry $\alpha(\omega,\omega')=\alpha(\omega',\omega)$ of weights  and the fact that $F=-\nabla W$ is conservative are crucial.

\medskip

\noindent $(ii)$ {\bf Multi-species Patlak-Keller-Segel model}. The practical role of the fibers $\omega\in [0,1]$ in \eqref{E-kinetic-nonexchangeable} is to list agents according to their type. Hence, the above PDE may be interpreted as a multi-species model, where $\omega$ parametrizes the various species of the system, and $\alpha=\alpha(\xi,\xi')$ describes the weights of interactions between distinct species. As special case we could set the Newtonian potential $W=W_d$ in dimension $d$, {\it i.e.}, 
$$
W_d(x)=\left\{
\begin{array}{ll}
-\frac{1}{2\pi}\log|x|, & d=2,\\
\frac{\Gamma(\frac{d}{2}+1)}{d(d-2)\pi^{d/2}} \frac{1}{|x|^{d-2}}, & d\geq 3.
\end{array}
\right.
$$
Then, we have for $\mu_t(x,\omega)=\rho_t(x,\omega)\,dx\otimes\nu(\omega)$
\begin{align}\label{E-multispecies-PKS}
\begin{aligned}
 &\partial_t \rho+\divop_x(\rho\,\nabla c)=\sigma\,\Delta_x\rho, \qquad t\geq 0,\quad x\in \mathbb{R}^d,\quad \omega\in [0,1],\\
&-\Delta_x c_t(x,\omega)=\int_0^1\alpha(\omega,\omega')\,\rho_t(x,\omega')\,d\nu(\omega').
\end{aligned}
\end{align}
This is an infinite multi-species Patlak-Keller-Segel model, where $\rho=\rho_t(x,\omega)$ is the density of bacteria and $c=c_t(x,\omega)$ correspond to the density of self-generated chemical of type $\omega$. Here, $\alpha=\alpha(\omega,\omega')$ represents the chemical generation coefficient ruling the relative impact of the bacteria of type $\omega'$ on the generation of chemical of type $\omega$, and $\nu\in \mathcal{P}([0,1])$ is a probability measure governing the distribution of the multiple type of species. If we restrict to atomic $\nu(\omega)=\frac{1}{M}\sum_{k=1}^M\delta_{k/M}(\omega)$, only finitely many species exist namely $\rho^k_t(x)=\rho_t(x,k/M)$ and $\alpha_{kl}=\alpha(k/M,l/M)$ with $k,l=1,\ldots,M$. This reduces the above infinite multi-species model \eqref{E-multispecies-PKS} to the finite multi-species model introduced in \cite{HT-21}, where the energy reads
$$\mathcal{E}(t):=\frac{1}{2M^2}\sum_{k,l=1}^M\alpha_ {kl}\iint_{\mathbb{R}^d\times \mathbb{R}^d}W_d(x-x')\,\rho^k_t(x)\rho^l_t(x')\,dx\,dx'+ \frac{\sigma}{M}\sum_{k=1}^M\int_{\mathbb{R}^d}\rho^k_t(x)\log\rho^k_t(x)\,dx.$$
This free energy was proposed in the above literature to derive the long-time dynamics of the system in dimension $d=2$. We also refer to \cite{BRYZ-21,CDEFS-20,CHS-18,HPS-21-arxiv} for other recent multi-species models in the literature where our methods could be applied.

\subsection{Main goal 3: Applications to singular second-order alignment models}\label{maingoal3}

The final goal of the paper is to focus on a special example of fibered gradient flows \eqref{kurak}, where the energy functional has the following Kuramoto-type shape
\begin{align}\label{W}
\begin{aligned}
&{\mathcal E}_W[\mu]:= - \int_{\RR^{2d}} \omega\cdot x\,d\mu(x,\omega)+K\iint_{\RR^{2d}\times \RR^{2d}} W(x-x')\,d\mu(x,\omega)\, d\mu(x',\omega'),\\
&W(x) = \frac{1}{2-\alpha}\frac{1}{1-\alpha}|x|^{2-\alpha},\quad \nabla W(x) = \frac{1}{1-\alpha}\phi(|x|)x,\quad \phi(|x|)=\frac{1}{|x|^\alpha},
\end{aligned}
\end{align}
with $\alpha \in (0,1)$. We remark that such an energy is reminiscent of the above one for the Kuramoto-Sakaguchi equation \eqref{E-Kuramoto-Sakaguchi}, but it has been set in any arbitrary dimension $d\in \mathbb{N}$, with a less regular (but convex) interaction potential $W\in C^{1,1-\alpha}(\RR^d)$. First, we prove that $\partial_{W_{2,\nu}}\mathcal{E}_W[\mu_t] =\{\boldsymbol{u}[\mu_t]\}$, which reduces \eqref{kurak} to the Kuramoto-type equation
\begin{align}\label{kurakk}
\begin{aligned}
&\partial_t\mu + \divop_x(\boldsymbol{u}[\mu]\mu) = 0,\qquad t\geq 0,\quad (x,\omega)\in \RR^{2d},\\
&\boldsymbol{u}[\mu_t](t,x,\omega):=\omega-K\int_{\RR^{2d}}\nabla W(x-x')\,d\mu_t(x',\omega').
\end{aligned}
\end{align}

Equation \eqref{kurakk} became relevant recently owing to its connection to the singular 1D Cucker-Smale equation. Roughly speaking, the 1D variant of \eqref{kurakk} can be recovered from the 1D Cucker-Smale equation by a simple change of variables. This observation was used in many previous works \cite{CZ-21,HKP-19,ZZ-20} providing a new insight into large-time behavior and well-posedness. It is noteworthy that in the above literature the authors usually treat the weak formulations for \eqref{kurakk} as a new relaxed formulation for the 1D Cucker-Smale equation. That is not the case in our recent work \cite{PP-22-2-arxiv}, wherein we deal precisely with the multidimensional version of \eqref{kurakk} and provide separate definitions of weak solutions of \eqref{kurakk} and of a certain multidimensional second-order alignment model. In \cite{PP-22-2-arxiv} we establish sufficient condition for equivalence between such weak formulations. In particular we are able to transfer the information on uniqueness of solutions to \eqref{kurak} to the uniqueness of solutions to \eqref{kurakk} and to weak solutions of our second-order alignment model. In particular we prove well-posedness for the weakly singular 1D Cucker-Smale kinetic equation with singularity in \eqref{W} of order $\alpha\in(0,\frac{2}{3})$ and conditionally of order $\alpha\in(0,1)$, expanding on the regime assumed in \cite{P-14,P-15,CCMP-17} ({\it i.e.} $\alpha\in(0,\frac{1}{2})$). For more information on the singular Cucker-Smale model we refer to the survey \cite{MMPZ-19}.

To summarize, our efforts in this area amount to the following meta-theorem.

\medskip

\begin{mdframed}
{\bf Theorem B.} Problem \eqref{kurakk} is well posed and Theorem A applies to the above energy functional $\mathcal{E}_W$ in \eqref{W}, so that the solutions of \eqref{kurakk} agree with the associated fibered gradient flows \eqref{kurak}. Furthermore, \eqref{kurakk} admits a global-in-time contractivity/stability estimate in appropriate fibered-type distances, which ensures that any solution converges to an equilibrium and is globally-in-time recovered by a mean-field limit. A rigorous formulation of Theorem B can be found in Section \ref{apple}, Theorem \ref{t-B-rig}.
\end{mdframed}

\medskip

In light of Theorem A, the only unclear part of Theorem B is related to the contractivity/stability estimates, which is the main issue resolved in Section \ref{apple}. It is worthwhile to compare our contribution to previous works in this direction, {\it cf.} \cite{CCHKK-14, MP-22, P-19-arxiv, CZ-21}. 

On the one hand, in \cite{CCHKK-14,MP-22} the authors provided a contractivity estimate for the Kuramoto--Sakaguchi equation \eqref{E-Kuramoto-Sakaguchi}. This form of contraction was crucial to quantify the long-time dynamics in the large coupling strength regime, departing either from compactly supported initial data on a sufficiently small interval \cite{CCHKK-14}, or from generic initial data \cite{MP-22}. On the other hand, in \cite{P-19-arxiv} the second author derived a Dobrushin-type stability estimate for the Kuramoto--Sakaguchi equation with singular, one-sided Lipschitz kernels. Finally the contribution in \cite{CZ-21} deals exactly with a global-in-time contractivity estimate for the 1D version of \eqref{kurakk} by exploiting the uniform-in-time stability of the associated particle system and a useful reformulation of the system in $1D$. We expand this result to the multidimensional case, while keeping the argumentation purely kinetic.

\section{Preliminaries }\label{sec:prelim}

\subsection{Classical optimal transport}\label{subsec:classical-optimal-transport}
In this section we review a selection of basics from optimal transport theory that can be found for example in \cite{AGS-08,S-15,V-09}. We start by recalling some tools from measure theory. 

\begin{defi}[Probability measures]
The space of probability measures on $\RR^d$ is denoted by $\mathcal{P}(\RR^d)$. Furthermore, we define the subspace of probability measures with finite second order moment as
$$\mathcal{P}_2(\RR^d):=\left\{\mu \in\mathcal{P}(\RR^d):\,\int_{\RR^d}\vert x\vert^2\,d\mu(x)<\infty\right\}.$$
\end{defi}

The natural topology for probability measures in $\RR^d$ is the so called \textit{narrow topology}, which is complete, separable and metrizable, and sequentially characterized as follows.

\begin{defi}[Narrow convergence]\label{D-narrow-convergence}
Consider any $\mu$ and any sequence $\{\mu_n\}_{n\in \mathbb{N}}$ in $\mathcal{P}(\mathbb{R}^d)$. We say that $\mu_n$ converges narrowly to $\mu$ in $\mathcal{P}(\mathbb{R}^d)$ if
$$\lim_{n\rightarrow \infty}\int_{\mathbb{R}^d}\varphi \,d\mu_n=\int_{\mathbb{R}^d}\varphi\,d\mu,$$
for every $\varphi \in C_b(\mathbb{R}^d)$.
\end{defi}
Except otherwise stated, $\mathcal{P}(\mathbb{R}^d)$ will be endowed with the narrow topology. It is a Polish space, {\it i.e.}, a complete, separable and metrizable topological space, {\it cf.} \cite[Theorem 11.3.3 and Corollary 11.5.5]{D-02} and \cite[Theorem 6.8]{B-99}. In fact, several different distances can be defined to metrize the narrow topology. We recall two of them that we shall use along the paper.

\begin{defi}[L\'evy-Prokhorov and bounded-Lipschitz distances]\label{D-Levy-Prokhorov-bounded-Lipschitz-distances}
Consider $\mu_1,\mu_2\in \mathcal{P}(\mathbb{R}^d)$. We respectively define their L\'evy-Prokhorov and bounded-Lipschitz distances by
\begin{align}
d_{LP}(\mu_1,\mu_2)&:=\inf\left\{\varepsilon>0:\,\mu_1(B)\leq \mu_2(B_\varepsilon)+\varepsilon,\,\mbox{for all Borel set }B\subseteq\mathbb{R}^d\right\},\label{E-Levy-Prohorov-distance}\\
d_{\BL}(\mu_1,\mu_2)&:=\sup\left\{\left\vert \int_{\mathbb{R}^d}\varphi\,d(\mu_1-\mu_2)\right\vert:\,\varphi\in \BL_1(\mathbb{R}^d)\right\}.\label{E-bounded-Lipschitz-distance}
\end{align}
Here, $B_\varepsilon:=\{x\in \mathbb{R}^d:\,\dist(x,B)\leq \varepsilon\}$ and $\BL_1(\mathbb{R}^d)$ consists of the bounded-Lipschitz functions $\varphi\in \BL(\mathbb{R}^d)$ such that $\Vert \varphi\Vert_{\BL(\mathbb{R}^d)}\leq 1$, for the bounded-Lipschitz norm
\begin{equation}\label{E-bounded-Lipschitz-norm}
\Vert \varphi\Vert_{\BL(\mathbb{R}^d)}=\max\{\Vert \varphi\Vert_{L^\infty(\mathbb{R}^d)},[\varphi]_{\Lip}\}.
\end{equation}
\end{defi}

A usual characterization of narrow convergence is the classical Portmanteau's theorem, see for instance \cite[Theorem 2.1]{B-99}.

\begin{theo}[Portmanteau]\label{T-Portmanteau}
Consider $\mu \in \mathcal{P}(\mathbb{R}^d)$ and any sequence $\{\mu_n\}_{n\in \mathbb{N}}\subseteq \mathcal{P}(\mathbb{R}^d)$. Then, the following statements are equivalent:
\begin{enumerate}[label=(\roman*)]
\item $\mu_n\rightarrow \mu$ narrowly in $\mathcal{P}(\mathbb{R}^d)$.
\item For any $\varphi\in \BL(\mathbb{R}^d)$, we have
$$\lim_{n\rightarrow\infty}\int_{\mathbb{R}^d}\varphi\,d\mu_n=\int_{\mathbb{R}^d}\varphi\,d\mu.$$
\item $\mu(U)\leq \liminf_{n\rightarrow \infty}\mu_n(U)$, for every open set $U\subseteq \mathbb{R}^d$.
\item $\limsup_{n\rightarrow \infty}\mu_n(C)\leq \mu(C)$, for every closed set $C\subseteq \mathbb{R}^d$.
\end{enumerate}
\end{theo}

Moreover, narrow compactness is fully characterized by the so-called Prokhorov's compactness theorem, see for instance \cite[Theorems 5.1 and 5.2]{B-99}.

\begin{theo}[Prokhorov]\label{T-Prokhorov}
Consider any family of measures $\Gamma\subseteq \mathcal{P}(\mathbb{R}^d)$. Then, $\Gamma$ is relatively compact in the narrow topology of $\mathcal{P}(\mathbb{R}^d)$ if, and only if, $\Gamma$ is uniformly tight, {\it i.e.}, for every $\varepsilon>0$ there exists a compact set $K\subseteq \mathbb{R}^d$ such that $\mu(K)\geq 1-\varepsilon$ for every $\mu\in \Gamma$.
\end{theo}

Now, we recall some tools from classical optimal transport.

\begin{defi}[Push forward measure]
Set $d_1,d_2\in \mathbb{N}$, a Borel-measurable map $T:\RR^{d_1}\to\RR^{d_2}$ and a non-negative finite Radon measure $\mu\in \mathcal{P}(\RR^{d_1})$. Then, the pushforward of $\mu$ along $T$ is defined as the measure $T_\# \mu\in \mathcal{P}(\mathbb{R}^{d_2})$ such that
\begin{align*}
(T_\#\mu)(B) = \mu(T^{-1}(B)),
\end{align*}
for every Borel set $B\subseteq \RR^{d_2}$. Furthermore, a measurable function $g:\RR^{d_2}\longrightarrow \RR$ is $T_\#\mu$-integrable if, and only if, $g\circ T:\RR^{d_1}\longrightarrow \RR$ is $\mu$-integrable and we have
\begin{align*}
\int_{\RR^{d_2}}g\,d(T_\#\mu) = \int_{\RR^{d_1}} (g\circ T)\,d\mu.
\end{align*}
\end{defi}

One of the main contributions of classical transport theory is that it provides an appropriate metrization of the subspace $\mathcal{P}_2(\mathbb{R}^d)$ compatible with the narrow topology, using the so called \textit{quadratic Wasserstein distance}. We shall recall the main concepts in the sequel.

\begin{defi}[Transference plans]
Set $d_1,\ldots,d_n\in \mathbb{N}$ and $\mu_i\in \mathcal{P}(\mathbb{R}^{d_i})$ for $i=1,\ldots,n$. We say that a probability measure $\gamma\in \mathcal{P}(\RR^{d_1+\cdots +d_n})$ is a transference plan between $\mu_1,\ldots,\mu_n$ when 
$$(\pi_i)_\#\gamma = \mu_i,$$
for all $i=1,\ldots,n$. Here, $\pi_i:\RR^{d_1+\cdots + d_n}\longrightarrow \RR^{d_i}$ is the projection onto the $i$-th component. The set of all transport plans between $\mu_1,\ldots,\mu_n$ will be denoted by $\Gamma(\mu_1,...,\mu_n)$.
\end{defi}

\begin{defi}[Quadratic Wasserstein distance]\label{D-Wasserstein}
Consider $\mu_1,\mu_2\in {\mathcal P}_2(\RR^d)$. We shall define $W_2$ by the solution of the following optimal transportation problem
\begin{align}\label{wasdef}
W_2(\mu_1, \mu_2) = \left(\inf_{\gamma\in \Gamma(\mu_1,\mu_2)}\int_{\RR^{2d}}\vert x-x'\vert^2\,d\gamma(x,x')\right)^{1/2}.
\end{align}
\end{defi}

The existence of minimizers $\gamma$ of the above problem \eqref{wasdef} is guaranteed by the narrow compactness of the set of transference plan $\Gamma(\mu_1,\mu_2)$ and the lower semicontinuity of the functional, see \cite[Theorem 4.1]{V-09}.

\begin{defi}[Optimal transference plans]
We say that $\gamma\in \Gamma(\mu_1,\mu_2)$ is an optimal transference plan between $\mu_1$ and $\mu_2$ if it is a minimizer of \eqref{wasdef}. The set of all such optimal transference plans will be denoted by $\Gamma_o(\mu_1,\mu_2)$.
\end{defi}

We recall the class of optimal transference plans $\gamma\in \Gamma_o(\mu_1,\mu_2)$ of the form $\gamma:=(I,T)_\#\mu_1$, where $T:\mathbb{R}^d\longrightarrow \mathbb{R}^d$ is a Borel-measurable map, such that $T_\#\mu_1=\mu_2$. These maps $T$ are usually called {\it optimal transport maps} and, when they exist, they allow restating $W_2$ in the following equivalent way
$$W_2(\mu_1,\mu_2)=\left(\int_{\RR^d}\vert x-T(x)\vert^2\,d\mu_1(x)\right)^{1/2}.$$

We end this section by recalling the following classical results, which characterize the metric structure and convergence of the quadratic Wasserstein space. See \cite[Proposition 7.1.5]{AGS-08} and \cite[Theorem 6.9]{V-09} for further details and proofs.

\begin{pro}[Quadratic Wasserstein space]
The space $(\mathcal{P}_2(\RR^d),W_2)$ is a Polish space, {\it i.e.}, a complete, separable metric space.
\end{pro}

\begin{pro}[Convergence in $(\mathcal{P}_2(\mathbb{R}^d),W_2)$]\label{P-characterization-convergence-W2}
Set any $\mu\in \mathcal{P}_2(\mathbb{R}^d)$ and  $\{\mu_n\}_{n\in \mathbb{N}}\subseteq \mathcal{P}_2(\mathbb{R}^d)$. Then, $\mu_n\rightarrow \mu$ in $W_2$ if, and only if, the following conditions are fulfilled:
\begin{enumerate}
\item (Narrow convergence) $\mu_n\rightarrow \mu$ narrowly in $\mathcal{P}(\mathbb{R}^d)$.
\item (Convergence of $2$-moments) The following conditions is verified
\begin{equation}\label{E-convergent-2-moments}
\lim_{n\rightarrow \infty}\int_{\mathbb{R}^d}\vert x\vert^2\,d\mu_n(x)=\int_{\mathbb{R}^d}\vert x\vert^2\,d\mu(x).
\end{equation}
\end{enumerate}
\end{pro}

\subsection{Random probability measures}\label{subsec:random-probability-measures}

As proposed in the Remark \ref{ohhnaivety} the measure-valued solutions to \eqref{kurak} cannot take arbitrary values in the full space $\mathcal{P}(\mathbb{R}^{2d})$. Instead, they are confined to a specific subspace of probability measures with prescribed marginal distributions with respect to $\omega$. This suggests the following definition.

\begin{defi}[Fibered probability measures]\label{D-fibered-measures}
Let $\nu\in \mathcal{P}(\mathbb{R}^d)$ be any probability measure. We define the following subset of $\mathcal{P}(\mathbb{R}^{2d})$
\begin{equation}\label{E-fibered-measures}\mathcal{P}_\nu(\mathbb{R}^{2d}):=\{\mu \in\mathcal{P}(\mathbb{R}^{2d}):\,\pi_{\omega\#}\mu=\nu\}.
\end{equation}
\end{defi}

In order to handle measures $\mu\in \mathcal{P}_\nu(\mathbb{R}^{2d})$ we just need to understand the $x$-dependent distribution at each value of fiber $\omega\in \mathbb{R}^d$. To such an end we shall systematically use conditional probabilities, which can be computed by virtue of the classical disintegration theorem.

\begin{theo}[Disintegration]\label{dis}
Set $d_1,d_2\in \mathbb{N}$, define the projection onto the second component $\pi_2:\mathbb{R}^{d_1}\times \mathbb{R}^{d_2}\longrightarrow \mathbb{R}^{d_2}$ and consider $\mu\in \mathcal{P}(\mathbb{R}^{d_1}\times \mathbb{R}^{d_2})$ and $\nu:=(\pi_2)_\# \mu\in \mathcal{P}(\mathbb{R}^{d_2})$. Then, there exists a family of probability measures $\{\mu^{x_2}\}_{x_2\in \mathbb{R}^{d_2}}\subseteq \mathcal{P}(\mathbb{R}^{d_1})$, which is uniquely defined $\nu$-a.e and verifies the following properties:
\begin{enumerate}[label=(\roman*)]
\item {\it (Borel family)} The following map is Borel-measurable
$$
x_2\in \mathbb{R}^{d_2}  \longmapsto  \mu^{x_2}(B),
$$
for every Borel set $B\subseteq \mathbb{R}^{d_1}$.
\item {\it (Disintegration formula)} The following formula holds true
\begin{align}\label{disb}
\iint_{\mathbb{R}^{d_1}\times \mathbb{R}^{d_2}} \varphi(x_1,x_2)\,d\mu(x_1,x_2)=\int_{\mathbb{R}^{d_2}}\left(\int_{\mathbb{R}^{d_1}}\varphi(x_1,x_2)\,d\mu^{x_2}(x_1)\right)\,d\nu(x_2),
\end{align}
for every Borel-measurable map $\varphi:\mathbb{R}^{d_1}\times \mathbb{R}^{d_2}\longmapsto [0,+\infty)$.
\end{enumerate}
\end{theo}

The above family $\{\mu^{x_2}\}_{x_2\in \mathbb{R}^{d_2}}$ of probability measures is usually called a {\it disintegration} of $\mu$ with respect to the marginal distribution $\nu$. For simplicity, we will often refer to \eqref{disb} as:
$$\mu(x_1,x_2)=\mu^{x_2}(x_1)\otimes \nu(x_2).$$
In a probabilistic setting, $\mu^{x_2}$ represents the {\it conditional probability} of $\mu$ given the value $x_2\in \mathbb{R}^{d_2}$. In the following result, we provide a useful characterization of the measurability condition in the first item of Theorem \ref{dis}, which will be key for our treatment of space $\mathcal{P}_\nu(\mathbb{R}^{2d})$. 

\begin{pro}[Borel family vs Borel-measurability]\label{P-Borel-family-vs-measurability}
Set $\nu\in \mathcal{P}(\mathbb{R}^d)$ and a family of probability measures $\{\mu^\omega\}_{\omega\in \mathbb{R}^d}\subseteq \mathcal{P}(\mathbb{R}^d)$. Then, the following conditions are equivalent:
\begin{enumerate}[label=(\roman*)]
\item {\it (Borel family I)} The following scalar function
$$\begin{array}{cccl}
\mu_B:&\mathbb{R}^d & \longrightarrow & \mathbb{R}\\
& \omega & \longmapsto & \mu^\omega(B),
\end{array}$$
is Borel-measurable for every Borel subset $B\subseteq \mathbb{R}^d$.
\item {\it (Borel family II)} The following scalar function
$$\begin{array}{cccl}
\mu_\phi:&\mathbb{R}^d & \longrightarrow & \mathbb{R}\\
& \omega & \longmapsto &  \int_{\mathbb{R}^d}\phi\,d\mu^\omega,
\end{array}$$
is Borel-measurable for every bounded and Borel-measurable $\phi:\mathbb{R}^d\longrightarrow \mathbb{R}$.
\item {\it (Borel-measurability)} The following measure-valued map
$$\begin{array}{cccl}
\mathfrak{X}_\mu:&\mathbb{R}^d & \longrightarrow & \mathcal{P}(\mathbb{R}^d)\\
& \omega & \longmapsto & \mu^\omega,
\end{array}
$$
is Borel-measurable when $\mathcal{P}(\mathbb{R}^d)$ is endowed with its narrow topology.
\end{enumerate}
Additionally, assume that $\mu^\omega\in \mathcal{P}_2(\mathbb{R}^{d})$ for each $\omega\in \mathbb{R}^d$. Then all the above three conditions are also equivalent with the fact that the following measure-valued map
$$\begin{array}{cccl}
\mathfrak{X}_\mu:&\mathbb{R}^d & \longrightarrow & \mathcal{P}_2(\mathbb{R}^d)\\
& \omega & \longmapsto & \mu^\omega,
\end{array}
$$
is Borel-measurable when $\mathcal{P}_2(\mathbb{R}^d)$ is endowed with the Wasserstein distance $W_2$.
\end{pro}

\begin{proof}
$\diamond$ {\sc Step 1}: Equivalence of the first and second conditions.\\
On the one hand, assume that the first condition is verified and take any bounded Borel-measurable function $\phi:\mathbb{R}^d\longrightarrow\mathbb{R}$ that can be assumed nonnegative without loss of generality. Using the layer cake representation \cite[Theorem 1.13]{LL-01}, we obtain
\begin{align*}
\int_{\mathbb{R}^d}\phi\,d\mu^\omega&=\int_0^{\Vert \phi\Vert_{L^\infty(\mathbb{R}^d)}} \mu^\omega(\{x\in \mathbb{R}^d:\,\phi(x)>\lambda\})\,d\lambda\\
&=\lim_{n\rightarrow\infty} \frac{\Vert\phi\Vert_{L^\infty(\mathbb{R}^d)}}{n}\sum_{k=1}^n\mu^\omega\left(\left\{x\in \mathbb{R}^d:\,\phi(x)> \Vert \phi\Vert_{L^\infty(\mathbb{R}^d)}\frac{k}{n} \right\}\right),
\end{align*}
where we have used an approximation with Riemann sums in the second line. Since $\phi$ is Borel-measurable, then each of the level sets is a Borel set. Thus, the second condition follows. On the other hand, the converse follows by taking the bounded Borel-measurable functions $\phi:=\chi_{B}$, where $B\subseteq \mathbb{R}^d$ is any arbitrary Borel set.

\medskip

$\diamond$ {\sc Step 2}: Equivalence with the third condition.\\
A proof of the ``if'' condition is given in \cite[Lemma 12.4.7]{AGS-08}. Here, we use an easier alternative derivation that shows the full equivalence. The cornerstone is the following observation from \cite{K-95}. For each bounded Borel-measurable function $\phi:\mathbb{R}^d\longrightarrow\mathbb{R}$ consider the associated functional
$$\begin{array}{cccc}
T_\phi:&\mathcal{P}(\mathbb{R}^d) & \longrightarrow & \mathbb{R}\\
& \lambda & \longmapsto & \int_{\mathbb{R}^d} \phi\,d\lambda.
\end{array}$$
We recall that the Borel $\sigma$-algebra of $\mathcal{P}(\mathbb{R}^d)$ (endowed with the narrow topology) is precisely the smallest $\sigma$-algebra on $\mathcal{P}(\mathbb{R}^d)$ for which $T_\phi$ is Borel-measurable for every bounded and Borel-measurable $\phi$, see \cite[Theorem 17.24]{K-95}. In symbols
\begin{equation}\label{E-sigma-algebra-P}
\mathcal{B}(\mathcal{P}(\mathbb{R}^d))=\sigma\left(\left\{T_\phi^{-1}((a,b]):\,\phi\mbox{ is bounded Borel-measurable and }a<b\right\}\right).
\end{equation}
Also, notice that by definition $\mu_\phi=T_\phi\circ \mathfrak{X}_\mu$ for every bounded Borel-measurable function $\phi$. Then, by the identity \eqref{E-sigma-algebra-P} it is clear that $\mathfrak{X}_\mu$ is Borel-measurable if, and only if, $\mu_\phi$ is Borel-measurable for every bounded and Borel-measurable $\phi$. This ends the proof of this part.

\medskip

$\diamond$ {\sc Step 3}: Equivalence with the last condition.\\
Assume that $\mu^\omega\in\mathcal{P}_2(\mathbb{R}^d)$ for each $\omega\in \mathbb{R}^d$ and let us show that the third condition is equivalent with the last one. First, assume that the third condition holds. Recall that the narrow topology on $\mathcal{P}(\mathbb{R}^d)$ is metrizable by the L\'evy--Prokhorov metric $d_{LP}$ given in \eqref{E-Levy-Prohorov-distance} of Definition \ref{D-Levy-Prokhorov-bounded-Lipschitz-distances}. Then we define the following distance over $\mathcal{P}_2(\mathbb{R}^d)$
\begin{equation}\label{E-quadratic-Wasserstein-distance-equivalent}
d_{LP,2}(\lambda_1,\lambda_2):=d_{LP}(\lambda_1,\lambda_2)+\left\vert\int_{\mathbb{R}^d} \vert x\vert^2\,d\lambda_1-\int_{\mathbb{R}^d}\vert x\vert^2\,d\lambda_2\right\vert,
\end{equation}
for $\lambda_1,\lambda_2\in \mathcal{P}_2(\mathbb{R}^d)$. Notice that the third condition is exactly equivalent to saying that
$$
\omega\in \mathbb{R}^d\longmapsto d_{LP}(\mu^\omega,\sigma),
$$
is Borel-measurable for every $\sigma\in \mathcal{P}(\mathbb{R}^d)$. Also, by \eqref{E-quadratic-Wasserstein-distance-equivalent} we infer that
$$
\omega\in \mathbb{R}^d\longmapsto d_{LP,2}(\mu^\omega,\sigma),
$$
is Borel-measurable for every $\sigma\in \mathcal{P}_2(\mathbb{R}^d)$. We conclude by noticing that this amounts to the last condition thanks to the fact that $d_{LP,2}$ is topologically equivalent to the quadratic Wasserstein distance $W_2$ in $\mathcal{P}_2(\mathbb{R}^d)$ by Proposition \ref{P-characterization-convergence-W2}. The ``only if'' part is clear by continuity (thus Borel measurability) of the embedding $\mathcal{P}_2(\mathbb{R}^d)\hookrightarrow \mathcal{P}(\mathbb{R}^d)$.
\end{proof}

\begin{rem}[Narrow vs stable topology]\label{R-narrow-vs-stable}
The above Theorem \ref{dis} and Proposition \ref{P-Borel-family-vs-measurability} allow identifying objects $\mu\in \mathcal{P}_\nu(\mathbb{R}^{2d})$ alternatively as:
\begin{enumerate}[label=(\roman*)]
\item A probability measure $\mu\in \mathcal{P}(\mathbb{R}^{2d})$ with a fixed marginal $\nu$.
\item A Borel family $\{\mu^\omega\}_{\omega\in \mathbb{R}^d}\subseteq \mathcal{P}(\mathbb{R}^d)$, {\it i.e.}, a {\it Markov transition kernel}.
\item A Borel-measurable map $\omega\in \mathbb{R}^d\longmapsto \mu^\omega\in \mathcal{P}(\mathbb{R}^d)$, {\it i.e.}, a {\it random probability measure}, or also called {\it Young measure}.
\end{enumerate} 
Consequently, various different topologies could be given to $\mathcal{P}_\nu (\mathbb{R}^{2d})$. For instance, the representation {\it (i)} suggests endowing $\mathcal{P}_\nu(\mathbb{R}^{2d})$ with the induced narrow topology from $\mathcal{P}(\mathbb{R}^{2d})$. Alternatively, representation {\it (iii)} identifies elements of $\mathcal{P}_\nu(\mathbb{R}^{2d})$ with Young measures. The canonical choice in this community is the stable topology, that is, the coarsest topology on $\mathcal{P}_\nu(\mathbb{R}^{2d})$ such that the maps $\mu\in \mathcal{P}_\nu(\mathbb{R}^{2d})\longmapsto \int_{\mathbb{R}^{2d}}\varphi\,d\mu$ are continuous for all bounded test functions $\varphi:\mathbb{R}^{2d}\longrightarrow \mathbb{R}$ verifying that $\varphi(x,\cdot)$ is Borel-measurable for all $x\in \mathbb{R}^d$ and $\varphi(\cdot,\omega)$ is continuous for $\nu$-a.e. $\omega\in \mathbb{R}^d$ (see \cite{Ba-99,CRV-04,HL-15,V-89}). Whilst apparently coarser than the narrow topology, both turn out to agree on $\mathcal{P}_\nu(\mathbb{R}^{2d})$, see \cite[Corollary 2.9]{JM-81}, \cite[Theorem 2.1.1(D)]{CRV-04} or \cite[Lemma 2.1]{BL-18-arxiv}.
\end{rem}

\begin{rem}[Narrow topology and fibers]\label{R-narrow-vs-fibers}
When a sequence $\{\mu_n\}_{n\in \mathbb{N}}\subseteq \mathcal{P}_\nu(\mathbb{R}^{2d})$ converges narrowly to some $\mu\in \mathcal{P}_\nu(\mathbb{R}^{2d})$, one might be tempted to claim that $\mu_n^\omega\rightarrow \mu^\omega$ narrowly for $\nu$-a.e. $\omega\in \mathbb{R}^d$. This naive idea would simplify many technical aspects of this paper to their classical counterparts on each fiber, but unfortunately such a property does not hold. A possible argument relies on the well-known density of Dirac Young measure, see \cite[Theorem 2.2.3]{CRV-04} and \cite[Proposition 2.2]{BL-18-arxiv}. Specifically, let us define
$$\mathfrak{X}_\nu(\mathbb{R}^{2d}):=\left\{\delta_{u(\omega)}(x)\otimes \nu(\omega)\,:u:\mathbb{R}^d\longrightarrow\mathbb{R}^d\mbox{ is Borel-measureable}\right\}.$$
Then, $\mathfrak{X}_\nu(\mathbb{R}^{2d})$ is narrowly dense in $\mathcal{P}_\nu(\mathbb{R}^{2d})$ for non-atomic $\nu$. However, for any sequence $\mu_n(x,\omega):=\delta_{u_n(\omega)}(x)\otimes \nu(\omega)$, the corresponding disintegrations $\mu_n^\omega(x)=\delta_{u_n(\omega)}(x)$ consist in Dirac masses, and then they cannot converge narrowly to anything else than a Dirac mass.
\end{rem}

\section{Fibered Wasserstein space}\label{sec:fibered-wasserstein}

In this section we introduce a novel transport distance, well adapted to the fibered structure of measures $\mathcal{P}_\nu(\mathbb{R}^{2d})$ in Definition \ref{D-fibered-measures}. We also present some results extending the main properties of the classical setting to our new fibered case. In particular, we obtain a fibered variant of the Riemannian structure of ${\mathcal P}_2(\RR^d)$ found by {\sc F. Otto} \cite{O-01}, which will set the basis to the study of fibered gradient flows. In doing so, we recall that there are fundamental obstructions when one tries to apply classical methods at each fiber value ({\it cf.} Remark \ref{ohhnaivety}). These will become apparent in the sequel.

\subsection{Fibered optimal transport}

\begin{defi}[Fibered quadratic Wasserstein space]\label{D-fibered-Wasserstein}
Let $\nu\in \mathcal{P}(\mathbb{R}^d)$ be any probability measure. We define the fibered quadratic Wasserstein space $(\mathcal{P}_{2,\nu}(\mathbb{R}^{2d}),W_{2,\nu})$ by
\begin{align}
\mathcal{P}_{2,\nu}(\mathbb{R}^{2d})&:=\left\{\mu\in \mathcal{P}_\nu(\mathbb{R}^{2d}):\,\int_{\mathbb{R}^{2d}}\vert x\vert^2\,d\mu(x,\omega)<\infty\right\},\label{p2nu}\\
W_{2,\nu}(\mu_1,\mu_2)&:=\left(\int_{\mathbb{R}^d}W_2^2(\mu_1^\omega,\mu_2^\omega)\,d\nu(\omega)\right)^{1/2}\label{dnu},
\end{align}
for any $\mu_1,\mu_2\in \mathcal{P}_{2,\nu}(\mathbb{R}^{2d})$, where $\{\mu_1^\omega\}_{\omega\in \mathbb{R}^d}$ and $\{\mu_2^\omega\}_{\omega\in \mathbb{R}^d}$ in $\mathcal{P}_2(\RR^d)$ are their families of disintegrations with respect to the variable $\omega$, and $W_2$ is the classical quadratic Wasserstein distance in $\RR^d$ ({\it cf.} Definition \ref{D-Wasserstein}).
\end{defi}

\begin{rem}[Good definition of $W_{2,\nu}$]\label{R-W2nu-good-definition}
Notice that $\mathcal{P}_{2,\nu}(\RR^d)$ in \eqref{p2nu} has been carefully chosen in order for $W_{2,\nu}(\mu_1,\mu_2)$ above to be well defined and finite. On the one hand, the function $\omega\in \mathbb{R}^d \longmapsto W_2(\mu_1^\omega,\mu_2^\omega)$ is Borel-measurable because $\{\mu^\omega_1\}_{\omega\in \RR^d}$ and $\{\mu^\omega_2\}_{\omega\in \RR^d}$ are Borel families in the sense of item $(i)$ in Theorem \ref{dis} (see \cite[Lemma 12.4.7]{AGS-08} and Proposition \ref{P-Borel-family-vs-measurability} above). On the other hand, the above function belongs indeed to $L^2_\nu(\RR^d)$ by the finiteness of the second order moment with respect to $x$ in \eqref{p2nu}, namely,
$$
\int_{\mathbb{R}^d}W_2^2(\mu_1^\omega,\mu_2^\omega)\,d\nu(\omega)\leq 2\int_{\mathbb{R}^{2d}}\vert x\vert^2\,d\mu_1(x,\omega)+2\int_{\mathbb{R}^{2d}}\vert x\vert^2\,d\mu_2(x,\omega)<\infty.
$$
Finally, since $\mu_1$ and $\mu_2$ have the same marginal $\nu$, then $W_{2,\nu}$ is non-degenerate. Specifically, $W_{2,\nu}(\mu_1,\mu_2)=0$ if, and only if, $\mu_1^\omega= \mu_2^\omega$ for $\nu$-a.e. $\omega\in \mathbb{R}^d$. Note that the latter amounts to $\mu_1=\mu_2$ thanks to the $\nu$-a.e. uniqueness of disintegrations in Theorem \ref{dis}. We remark that when different marginals $\nu_1$ and $\nu_2$ are involved, then the above definition (with either of the two $\nu_1$ or $\nu_2$) would clearly break down.
\end{rem}

One might be tempted to think that since measures in $\mathcal{P}_{2,\nu}(\mathbb{R}^{2d})$ have a fixed marginal $\nu$ in the second variable, then horizontal transport is always cheaper than any other type of transport and therefore $W_{2,\nu}$ should agree with $W_2$. However, this intuition is false as explained below.

\begin{eje}

Consider $x_1,x_2\in \mathbb{R}^d$ and $\omega_1,\omega_2\in \mathbb{R}^d$, and define the probability measures $\mu_1,\mu_2\in \mathcal{P}_{2,\nu}(\mathbb{R}^{2d})$ with $\nu(\omega):=\frac{1}{2}\delta_{\omega_1}(\omega)+\frac{1}{2}\delta_{\omega_2}(\omega)$ given by
\begin{align*}
\mu_1(x,\omega)&:=\frac{1}{2}\delta_{x_1}(x)\otimes \delta_{\omega_1}(\omega)+\frac{1}{2}\delta_{x_2}(x)\otimes \delta_{\omega_2}(\omega),\\
\mu_2(x,\omega)&:=\frac{1}{2}\delta_{x_2}(x)\otimes \delta_{\omega_1}(\omega)+\frac{1}{2}\delta_{x_1}(x)\otimes \delta_{\omega_2}(\omega).
\end{align*}
Then, by explicit calculation one has
$$W_{2,\nu}(\mu_1,\mu_2)=|x_1-x_2|,\quad W_2(\mu_1,\mu_2)=\min\{|x_1-x_2|,|\omega_1-\omega_2|\}.$$
Therefore, in general we have $W_2(\mu_1,\mu_2)\leq W_{2,\nu}(\mu_1,\mu_2)$ ({\it cf.} Proposition \ref{P-scaled-Wasserstein-distances} below). However, when the vertical distance is smaller than the horizontal distance, {\it i.e.}, $|\omega_1-\omega_2|<|x_1-x_2|$, then we actually have the strict inequality $W_2(\mu_1,\mu_2)<W_{2,\nu}(\mu_1,\mu_2)$.
\end{eje}

In what follows, we investigate the metric structure of $(\mathcal{P}_{2,\nu}(\mathbb{R}^{2d}),W_{2,\nu})$. This relies on an abstract construction that we recall in Appendix \ref{Appendix-metric-valued-L2}. Namely, given the metric measure space $(\RR^d,\vert \cdot\vert,\nu)$ and the metric space $(\mathcal{P}_2(\RR^d),W_2)$, we can define the metric-valued Lebesgue space $L^2_\nu(\RR^d,(\mathcal{P}_2(\RR^d),W_2))$ in Definition \ref{D-metric-valued-L2}. By construction and Proposition \ref{P-Borel-family-vs-measurability} we obtain
\begin{align}\label{E-W2nu-metric-valued-L2}
\begin{aligned}
\mathcal{P}_{2,\nu}(\mathbb{R}^{2d})&\equiv L^2_\nu(\RR^d,(\mathcal{P}_2(\RR^d),W_2)),\\
W_{2,\nu}&\equiv d_{L^2_\nu(\RR^d,(\mathcal{P}_2(\RR^d),W_2))}.
\end{aligned}
\end{align}
Consequently, by Proposition \ref{P-metric-valued-L2-properties} the following result holds true.

\begin{pro}\label{P-dnu-polish}
For any $\nu \in \mathcal{P}(\RR^d)$, the space $(\mathcal{P}_{2,\nu}(\RR^d),W_{2,\nu})$ is a Polish space.
\end{pro}

Although the fibered quadratic Wasserstein distance $W_{2,\nu}$ in Definition \ref{D-fibered-Wasserstein} has been introduced through a gluing process of fibered information, in the following we restate it as a pure constrained optimization problem over an appropriate class of transference plans.

\begin{defi}[Admissible transference plans]\label{D-admissible-plans-fibered}
Set any $\nu\in \mathcal{P}(\mathbb{R}^d)$ and $\mu_i\in \mathcal{P}_\nu(\mathbb{R}^{2nd})$ for $i=1,\ldots,n$. We say that $\gamma\in \mathcal{P}(\mathbb{R}^{2nd})$ is a $\nu$-admissible transference plan between $\mu_1,\ldots,\mu_n$ when $\gamma\in \Gamma (\mu_1,\ldots,\mu_n)$ and, in addition, the support of the marginal $\pi_{(\omega_1,\ldots,\omega_n)\#}\gamma$ lies in the set of diagonal points $\omega_1=\ldots=\omega_n$. By disintegration, $\nu$-admissible plans can be written as
\begin{equation}\label{plan-multi}
\gamma(x_1,\ldots,x_n,\omega_1,\ldots,\omega_n)=\gamma^\omega(x_1,\ldots,x_n)\otimes \nu(\omega_1)\otimes \delta_{\omega_1}(\omega_2)\otimes \delta_{\omega_1}(\omega_3)\otimes ... \otimes \delta_{\omega_1}(\omega_n),
\end{equation}
for some Borel family of probability measures $\{\gamma^\omega\}_{\omega\in \mathbb{R}^d}\subseteq \mathcal{P}(\mathbb{R}^{nd})$ such that $\gamma^\omega\in \Gamma(\mu_1^\omega,\ldots,\mu_n^\omega)$ for $\nu$-a.e. $\omega\in \mathbb{R}^d$. The set of $\nu$-admissible transference plans is denoted by $\Gamma_\nu(\mu_1,\ldots,\mu_n)$.
\end{defi}

For $n=2$ note that the formula \eqref{plan-multi} for a $\nu$-admissible plan $\gamma\in \Gamma_\nu(\mu_1,\mu_2)$ reduces to
\begin{equation}\label{plan}
\gamma(x,x',\omega,\omega')=\gamma^\omega(x,x')\otimes \nu(\omega)\otimes\delta_{\omega}(\omega'),
\end{equation}
where $\{\gamma^\omega\}_{\omega\in \mathbb{R}^d}\subseteq \mathcal{P}(\mathbb{R}^{2d})$ is a Borel family of probability measures such that $\gamma^\omega\in \Gamma(\mu_1^\omega,\mu_2^\omega)$ for $\nu$-a.e. $\omega\in \mathbb{R}^d$. Using classical measurable selection theorems \cite{CV-77}, the following result becomes clear (see also \cite[Lemma 12.4.7]{AGS-08} and \cite[Corollary 5.22]{V-09}).

\begin{lem}[Measurable selections]\label{L-measurable-selection}
Set any probability measure $\nu \in\mathcal{P}(\RR^d)$ and consider $\mu_1,\mu_2\in \mathcal{P}_{2,\nu}(\RR^{2d})$. Then, there exists a Borel family of probability measures $\{\gamma^\omega_o\}_{\omega\in \RR^d}\subseteq \mathcal{P}(\RR^d)$ such that $\gamma^\omega_o\in \Gamma_o(\mu_1^\omega,\mu_2^\omega)$ for $\nu$-a.e. $\omega\in \RR^d$.
\end{lem}

The existence of measurable selections fiberwise determined by optimal plans allows restating $W_{2,\nu}$ in Definition \ref{D-fibered-Wasserstein} as a constrained optimization problem as follows.

\begin{pro}[Optimal transport formulation I]\label{P-dnu-optimal-problem}
Set any $\nu\in \mathcal{P}(\RR^d)$ and consider $\mu_1,\mu_2\in \mathcal{P}_{2,\nu}(\RR^{2d})$. Then, the following identity holds true
\begin{equation}\label{dnu-optimal-problem}
W_{2,\nu}^2(\mu_1,\mu_2)=\inf_{\gamma\in \Gamma_\nu(\mu_1,\mu_2)}\int_{\RR^{4d}}\vert x-x'\vert^2\,d\gamma(x,x',\omega,\omega').
\end{equation}
In addition, minimizers to the above problem \eqref{dnu-optimal-problem} exist and they all take the form
\begin{equation}\label{plan-optimal}
\gamma_o(x,x',\omega,\omega')=\gamma^\omega_o(x,x')\otimes \nu(\omega)\otimes \delta_{\omega}(\omega'),
\end{equation}
for a Borel family $\{\gamma^\omega_o\}_{\omega\in \RR^d}\subseteq \mathcal{P}(\RR^{2d})$ with $\gamma^\omega_o\in \Gamma_o(\mu_1^\omega,\mu_2^\omega)$ for $\nu$-a.e. $\omega\in \RR^d$ (Lemma \ref{L-measurable-selection}).
\end{pro}

\begin{proof}

$\diamond$ {\sc Step 1}: Existence of minimizers.\\
First, consider any $\gamma_o$ as in \eqref{plan-optimal} and notice that, for any other $\nu$-admissible plan $\gamma$ as in \eqref{plan}, the following inequality holds
\begin{equation}\label{E-minimizers-relation}
\int_{\RR^{2d}}\vert x-x'\vert^2\,d\gamma_o^\omega(x,x')\leq \int_{\RR^{2d}}\vert x-x'\vert^2\,d\gamma^\omega(x,x'),
\end{equation}
for $\nu$-a.e. $\omega\in \RR^d$ by the optimality of the fibers $\gamma_o^\omega$. Integrating \eqref{E-minimizers-relation} against $\nu$ and recalling definition \eqref{dnu} of $W_{2,\nu}$ along with definitions \eqref{plan} and \eqref{plan-optimal} of $\gamma$ and $\gamma_o$ respectively, we obtain
$$W_{2,\nu}^2(\mu_1,\mu_2)=\int_{\mathbb{R}^{4d}}\vert x-x'\vert^2\,d\gamma_o(x,x',\omega,\omega')\leq \int_{\mathbb{R}^{4d}}\vert x-x'\vert^2\,d\gamma(x,x',\omega,\omega').$$
Hence, $\gamma_o$ is a minimizer. 

\medskip

$\diamond$ {\sc Step 2}: Minimizers take the form \eqref{plan-optimal}.\\
Now, consider any other minimizer $\widetilde{\gamma}_o\in \Gamma_\nu(\mu_1,\mu_2)$. Then, we obtain
$$\int_{\RR^d}\left(\int_{\RR^{2d}}\vert x-x'\vert^2\,d\widetilde{\gamma}_o^\omega(x,x')-\int_{\RR^{2d}} \vert x-x'\vert^2\,d\gamma_o^\omega(x,x')\right)\,d\nu(\omega)=0.$$
By the relation \eqref{E-minimizers-relation} with $\gamma=\widetilde{\gamma}_o$, we infer that
$$\int_{\RR^{2d}}\vert x-x'\vert^2\,d\widetilde{\gamma}_o^\omega(x,x')=\int_{\RR^{2d}} \vert x-x'\vert^2\,d\gamma_o^\omega(x,x')=W_2^2(\mu_1^\omega,\mu_2^\omega),$$
for $\nu$-a.e. $\omega\in \RR^d$, and this ends the proof.
\end{proof}

\begin{defi}[Optimal admissible transference plans]
Set any $\nu \in \mathcal{P}(\RR^d)$. We say that $\gamma\in \Gamma_\nu(\mu_1,\mu_2)$ is an optimal $\nu$-admissible transference plan between $\mu_1$ and $\mu_2$ if it is a minimizer of the problem \eqref{dnu-optimal-problem}. The set of all such optimal transference plans will be denoted by $\Gamma_{o,\nu}(\mu_1,\mu_2)$.
\end{defi}

We note that transference plans in the optimization problem \eqref{dnu-optimal-problem} in Proposition \ref{P-dnu-optimal-problem} are constrained to the class of $\nu$-admissible plans $\Gamma_\nu(\mu_1,\mu_2)$ in Definition \ref{D-admissible-plans-fibered}. In the following result we reformulate it as a genuine (constraint-free) optimal transport problem associated to a specific cost function; we omit the proof.

\begin{pro}[Optimal transport formulation II]\label{P-dnu-optimal-problem-2}
Set any $\nu\in \mathcal{P}(\RR^d)$ and consider $\mu_1,\mu_2\in \mathcal{P}_{2,\nu}(\RR^{2d})$. Then, the following identity holds true
\begin{equation}\label{dnu-optimal-problem-2}
W_{2,\nu}^2(\mu_1,\mu_2)=\inf_{\gamma\in \Gamma(\mu_1,\mu_2)}\int_{\mathbb{R}^{4d}}c_\infty((x,\omega),(x',\omega'))\,d\gamma(x,x',\omega,\omega'),
\end{equation}
where $c_\infty:\mathbb{R}^{2d}\times \mathbb{R}^{2d}\longrightarrow [0,+\infty]$ is the infinitely-valued cost function given by
\begin{equation}\label{E-scaled-costs-infinity}
c_\infty((x,\omega),(x',\omega')):=\left\{
\begin{array}{ll}
\vert x-x'\vert^2, & \mbox{if }\omega=\omega',\\
+\infty, & \mbox{if }\omega\neq \omega'.
\end{array}
\right.
\end{equation}
\end{pro}

In other words, note that since $c_\infty$ is an infinitely-valued cost function, then $W_{2,\nu}^2$ can be regarded as a usual optimal transport problem where that transport between different fibers is penalized with an infinite cost. Indeed, we obtain the following relation, provided that $\nu\in {\mathcal P}_2(\RR^d)$.

\begin{pro}[Hierarchy of distances]\label{P-scaled-Wasserstein-distances}
For any $\varepsilon\in \mathbb{R}_+^*$, we define the cost function  $c_\varepsilon:\mathbb{R}^{2d}\times \mathbb{R}^{2d}\longrightarrow \mathbb{R}_+$ and its transport distance $W_{c_\varepsilon}:\mathcal{P}_2(\mathbb{R}^{2d})\times \mathcal{P}_2(\mathbb{R}^{2d})\longrightarrow\mathbb{R}_0^+$ as follows
\begin{align}
&c_\varepsilon((x,\omega),(x',\omega')):=\vert x-x'\vert^2+\varepsilon^2\, \vert \omega-\omega'\vert^2,\label{E-scaled-costs}\\
&W_{c_\varepsilon}^2(\mu_1,\mu_2):=\inf_{\gamma\in \Gamma(\mu_1,\mu_2)}\int_{\mathbb{R}^{4d}}c_\varepsilon((x,\omega),(x',\omega'))\,d\gamma(x,x',\omega,\omega'),\label{E-scaled-distances}
\end{align}
Then, the following embedding is found
\begin{equation}\label{E-scaled-distance-relations}
W_2(\pi_{x\#}\mu_1,\pi_{x\#}\mu_2)\leq W_{c_\varepsilon}(\mu_1,\mu_2)\leq W_{2,\nu}(\mu_1,\mu_2),
\end{equation}
for any $\varepsilon\in \mathbb{R}_+^*$ and $\mu_1,\mu_2\in \mathcal{P}_2(\mathbb{R}^{2d})$. Moreover, the asymptotic regimes $\varepsilon\rightarrow 0$ and $\varepsilon\rightarrow \infty$ are identified, namely, we obtain
\begin{align}\label{E-scaled-distance-limits}
\begin{aligned}
\lim_{\varepsilon\rightarrow 0}W_{c_\varepsilon}(\mu_1,\mu_2)&=W_2(\pi_{x\#}\mu_1,\pi_{x\#}\mu_2),\\
\lim_{\varepsilon\rightarrow \infty}W_{c_\varepsilon}(\mu_1,\mu_2)&=W_{2,\nu}(\mu_1,\mu_2),
\end{aligned}
\end{align}
for any $\mu_1,\mu_2\in \mathcal{P}_2(\mathbb{R}^{2d})$.
\end{pro}

\begin{proof}
First, note that for any $\varepsilon\in \mathbb{R}^+$ we have $c_0\leq c_\varepsilon\leq c_\infty$, and thus $W_{c_0}(\mu_1,\mu_2)\leq W_{c_\varepsilon}(\mu_1,\mu_2)\leq W_{c_\infty}(\mu_1,\mu_2)$ for any $\mu_1,\mu_2\in \mathcal{P}_2(\mathbb{R}^{2d})$. Then, \eqref{E-scaled-distance-relations} follows by noting that
$$W_{c_0}(\mu_1,\mu_2)=W_2(\pi_{x\#}\mu_1,\pi_{x\#}\mu_2)\quad \mbox{and}\quad W_{c_\infty}(\mu_1,\mu_2)=W_{2,\nu}(\mu_1,\mu_2).$$
Second, note that $W_{c_\varepsilon}(\mu_1,\mu_2)$ is non-decreasing with $\varepsilon$. Hence, in the sequel we just focus on proving \eqref{E-scaled-distance-limits} along a subsequence.

\medskip

$\diamond$ {\sc Step 1}: Regime $\varepsilon\rightarrow 0$.\\
Let
$T_\varepsilon(x,\omega):=(x,\varepsilon\omega)$ for  $(x,\omega)\in\mathbb{R}^{2d}$ and $\varepsilon\geq 0$.
Then, by definition \eqref{E-scaled-distances},
\begin{equation}\label{E-scaled-distance-pushforward}
W_{c_\varepsilon}(\mu_1,\mu_2)=W_2(T_{\varepsilon\#}\mu_1,T_{\varepsilon\#}\mu_2).
\end{equation}
By definition $T_{\varepsilon\#}\mu_1\rightarrow T_{0\#}\mu_1$ and $T_{\varepsilon\#}\mu_2\rightarrow T_{0\#}\mu_2$ narrowly as $\varepsilon\rightarrow 0$. In addition, 
$$\int_{\mathbb{R}^{2d}}(\vert x\vert^2+\vert \omega\vert^2)\,d(T_{\varepsilon\#}\mu_i)(x,\omega)=\int_{\mathbb{R}^d}\vert x\vert^2\,d(\pi_{x\#}\mu_i)(x)+\varepsilon^2\int_{\mathbb{R}^d}\vert \omega\vert^2\,d(\pi_{\omega\#}\mu_i)(\omega),$$
for any $i=1,2$. Thus
$$\lim_{\varepsilon\rightarrow 0}\int_{\mathbb{R}^{2d}}(\vert x\vert^2+\vert \omega\vert^2)\,d(T_{\varepsilon\#}\mu_i)=\int_{\mathbb{R}^{2d}}(\vert x\vert^2+\vert \omega\vert^2)\,d(T_{0\#}\mu_i),$$
for any $i=1,2$. Consequently, by Proposition \ref{P-characterization-convergence-W2}, $T_{\varepsilon\#}\mu_1\rightarrow T_{0\#}\mu_1$ and $ T_{\varepsilon\#}\mu_2\rightarrow T_{0\#}\mu_2$ in $W_2$. Hence, we can pass to the limit in \eqref{E-scaled-distance-pushforward} and find
$$\lim_{\varepsilon\rightarrow 0}W_{c_\varepsilon}(\mu_1,\mu_2)=W_2(T_{0\#}\mu_1,T_{0\#}\mu_2).$$
We conclude the proof of $\eqref{E-scaled-distance-limits}_1$ by noting that $T_{0\#}\mu_i=(\pi_{x\#}\mu_i)\otimes \delta_0(\omega)$.

\medskip

$\diamond$ {\sc Step 2}: Regime $\varepsilon\to\infty$.\\
For any $\varepsilon>0$ consider $\gamma_\varepsilon\in \Gamma(\mu_1,\mu_2)$ optimal with respect to the cost function $c_\varepsilon$. Since $\{\gamma_\varepsilon\}_{\varepsilon\in \mathbb{R}^+}$ is uniformly tight by \cite[Lemma 4.4]{V-09}, then there exists a sequence $\varepsilon_n\rightarrow \infty$ and $\gamma_\infty\in \Gamma(\mu_1,\mu_2)$ such that $\gamma_{\varepsilon_n}\rightarrow \gamma_\infty$ narrowly by Prokhorov's Theorem \ref{T-Prokhorov}. By \eqref{E-scaled-distance-relations} we have
\begin{equation}\label{E-scaled-distance-concentration}
\int_{\mathbb{R}^{4d}}\vert \omega-\omega'\vert^2\,d\gamma_\varepsilon(x,x',\omega,\omega')\leq \frac{W_{2,\nu}^2(\mu_1,\mu_2)}{\varepsilon^2}\xrightarrow{\varepsilon\to\infty} 0,
\end{equation}
which implies that $\gamma_\infty$ is concentrated on $\{\omega=\omega'\}$ and thus $\gamma_\infty\in \Gamma_\nu(\mu_1,\mu_2)$.
 Then,
$$W_{2,\nu}^2(\mu_1,\mu_2)\leq \int_{\mathbb{R}^{4d}}c_\infty\,d\gamma_\infty=\int_{\mathbb{R}^{4d}}\vert x-x'\vert^2\,d\gamma_\infty\leq \liminf_{n\rightarrow \infty} \int_{\mathbb{R}^{4d}}c_0\,d\gamma_{\varepsilon_n}\leq \liminf_{n\rightarrow\infty} W_{c_{\varepsilon_n}}^2(\mu_1,\mu_2).$$
where in the first inequality we have used Proposition \ref{P-dnu-optimal-problem-2}, in the first identity we have used that $\gamma_\infty$ is $\nu$-admissible and in the second inequality we have applied the lower semicontinuity property \cite[Lemma 5.1.7]{AGS-08}. Hence, by the reverse inequalities \eqref{E-scaled-distance-relations} we conclude that
$$\lim_{n\rightarrow\infty}W_{c_{\varepsilon_n}}(\mu_1,\mu_2)=W_{2,\nu}(\mu_1,\mu_2).$$
\end{proof}

Note that $W_{c_1}$ reduces to the usual quadratic Wassertein distance, and then \eqref{E-scaled-distance-relations} yields a continuous and non-expansive embedding $(\mathcal{P}_{2,\nu}(\mathbb{R}^{2d}),W_{2,\nu})\hookrightarrow (\mathcal{P}_2(\mathbb{R}^{2d}),W_2)$. In addition, \eqref{E-scaled-distance-limits} identifies the asymptotic regimes of cheapest ($\varepsilon\rightarrow 0$) and most expensive ($\varepsilon\rightarrow \infty$) transportation cost in the variable $\omega$ as $W_2(\pi_{x\#}\mu_1,\pi_{x\#}\mu_1)$ and $W_{2,\nu}(\mu_1,\mu_2)$ respectively. Whilst the latter has not been much treated in the literature (it is the object of study of this paper), the former was already studied in \cite{CGS-10} with $d=1$. Indeed, the authors obtained sufficient conditions on $\mu_1$ and $\mu_2$ ({\it e.g.}, absolutely continuous and compactly supported) so the optimal plans $\gamma_\varepsilon$ associated with $W_{c_\varepsilon}(\mu_1,\mu_2)$ converge to $(I,T_{K})_{\#}\mu_1$ as $\varepsilon\rightarrow 0$, being $T_{K}$ the increasing (in lexicographical order) Knothe--Rosenblatt rearrangement between $\mu_1$ and $\mu_2$. This in particular recovers the first half $\eqref{E-scaled-distance-limits}_1$.

Notice that, although infinitely-valued, the cost function $c_\infty$ in \eqref{E-scaled-costs-infinity} of Proposition \ref{P-dnu-optimal-problem-2} is lower semicontinuous. This guarantees that many results of classical optimal transport are still available. In particular, optimal transference plans always exist thanks to \cite[Theorem 4.1]{V-09}. Indeed, they all belong to the class of $\nu$-admissible plans as proved in Proposition \ref{P-dnu-optimal-problem}. Similarly, Kantorovich duality still holds by virtue of \cite[Theorem 6.1.1]{AGS-08} or \cite[Theorem 5.10]{V-09}. Unfortunately, in our case $c_\infty$-cyclical monotonicity does not fully characterize optimality of transference plans.

\begin{rem}[Optimality of plans vs $c_\infty$-cyclical monotonocity]\label{optvsmon}
~

\indent $\bullet$ {\it (Necessary condition)} Since $c_\infty$ is lower-semicontinuous, then the $c_\infty$-cyclical monotonicity is still a necessary condition for the optimality of a transference plan (see \cite[Theorem 6.1.4]{AGS-08}). Indeed, consider any $\mu_1,\mu_2\in \mathcal{P}_{2,\nu}(\mathbb{R}^{2d})$ and assume that $\gamma\in \Gamma_{\nu}(\mu_1,\mu_2)$ is any optimal transference plan for $W_{2,\nu}(\mu_1,\mu_2)$. Then, $\gamma$ is concentrated on a subset $\Gamma\subseteq \mathbb{R}^{4d}$ which verifies
\begin{equation}\label{E-cinfty-cyclical-monotonicity}
\sum_{i=1}^n c_\infty((x,\omega),(x',\omega'))\leq \sum_{i=1}^n c_\infty((x_i,\omega_i),(x_{\sigma(i)}',\omega_{\sigma(i)}')),
\end{equation}
for any $(x_1,\omega_1,x_1',\omega_1'),\ldots,(x_n,\omega_n,\omega_n',x_k')\in \Gamma$ and each permutation $\sigma\in \mathfrak{S}_n$. Sets $\Gamma$ verifying \eqref{E-cinfty-cyclical-monotonicity} are often called $c_\infty$-cyclically monotone.

$\bullet$ {\it (Sufficient condition)} Condition \eqref{E-cinfty-cyclical-monotonicity} ensures optimality for many cost functions, particularly for finitely-valued costs (see \cite{AP-03}) and continuous infinitely-valued costs (see \cite{P-08}). However, there are classical counter-examples for generic lower semicontinuous costs ({\it e.g.}, \cite{AP-03}). Indeed, a more appropriate notion characterizing optimality in those degenerate cases appears to be the strong cyclical monotonicity proposed in \cite{ST-08}, which unfortunately loses the pointwise character in definition \eqref{E-cinfty-cyclical-monotonicity}. Our cost function $c_\infty$ is infinitely-valued and discontinuous, and it is indeed easy to infer that \eqref{E-cinfty-cyclical-monotonicity} does not characterize optimal plans.
\end{rem}

 In the following we study the stability of optimality of plans $\gamma_n\in \Gamma_{o,\nu}(\mu_n,\widetilde{\mu}_n)$ under narrow convergence of $\mu_n$ and $\widetilde{\mu}_n$. Note that the classical approach to stability of optimality with respect to the narrow convergence exploits the cyclical monotonocity as a characterization of optimality of $\gamma_n$, and the Kuratowski convergence of the supports of the transference plans under narrow convergence (see \cite[Proposition 7.1.3]{AGS-08} and \cite[Theorem 5.20]{V-09}). However, as explained in Remark \ref{optvsmon}, our cost $c_\infty$ is neither finitely-valued nor continuous so that cyclical monotonicity does not characterize optimality, and then this approach breaks down. 
 
 In addition, note that a simple fiberwise argument where one applies the classical result on each fiber $\omega\in \mathbb{R}^d$ would require the narrow convergence of each $\gamma_n^\omega$, and this is certainly something that one cannot ensure under solely joint narrow convergence of the plans $\gamma_n$ ({\it cf}. Remark \ref{R-narrow-vs-fibers}). Instead, we propose a suitable adjustment of the classical proof, which does not use the fiberwise narrow convergence of the plans, but it rather exploits the (weaker) fiberwise Kuratowski convergence of the supports of the transference plans under joint narrow convergence (see \eqref{E-fibered-Kuratowski} below), along with the fiberwise characterization of optimality through cyclical monotonocity. Whilst possibly confusing at first glance in view of Remarks \ref{ohhnaivety} and \ref{R-narrow-vs-fibers}, we anticipate that there is no real conflict as justified below in Remark \ref{R-narrow-fiberwise-Kuratowski-convergence}.

\begin{pro}[Narrow lsc and stability of optimality]\label{P-stability-optimal-plans}
Consider $\nu\in \mathcal{P}(\mathbb{R}^d)$, $\mu,\,\widetilde{\mu}\in \mathcal{P}_{2,\nu}(\mathbb{R}^{2d})$ and sequences $\{\mu_n\}_{n\in \mathbb{N}},\,\{\widetilde{\mu}_n\}_{n\in \mathbb{N}}\subseteq\mathcal{P}_{2,\nu}(\mathbb{R}^{2d})$ such that $\mu_n\rightarrow\mu$ and $\widetilde{\mu}_n\rightarrow \widetilde{\mu}$ narrowly in $\mathcal{P}_\nu(\mathbb{R}^{2d})$.
Then, the following properties hold true:
\begin{enumerate}[label=(\roman*)]
\item (Narrow lower semicontinuity)
\begin{equation}\label{E-lsc-W2nu}
W_{2,\nu}(\mu,\widetilde{\mu})\leq \liminf_{n\rightarrow\infty}W_{2,\nu}(\mu_n,\widetilde{\mu}_n).
\end{equation}
\item (Stability of optimality) 
Set $\gamma_n\in \Gamma_{o,\nu}(\mu_n,\widetilde{\mu}_n)$ for every $n\in \mathbb{N}$. Then, $\{\gamma_n\}_{n\in \mathbb{N}}$ is narrowly relatively compact and any limit point belongs to $\Gamma_{o,\nu}(\mu,\widetilde{\mu})$.
\end{enumerate}
\end{pro}

\begin{proof}

$\diamond$ {\sc Step 1}: Compactness of plans in the narrow topology.\\
Since the sequences $\{\mu_n\}_{n\in \mathbb{N}}$ and $\{\widetilde{\mu}_n\}_{n\in \mathbb{N}}$ are  narrowly convergent, then they are narrowly relatively compact. A classical argument based on Prokhorov's theorem ensures that any sequence $\gamma_n\in \Gamma_{\nu}(\mu_n,\widetilde{\mu}_n)$ is narrowly relatively compact  (see \cite[Lemma 5.2.2]{AGS-08}). Thus, up to a subsequence, $\gamma_n\to\gamma$ narrowly in ${\mathcal P}(\RR^{4d})$. Moreover, it is easy to show that the narrow convergence preserves the structure \eqref{plan}, and thus, the limit $\gamma$ belongs to $\Gamma_\nu(\mu,\widetilde{\mu})$.

\medskip

$\diamond$ {\sc Step 2}: Lower semicontinuity property \eqref{E-lsc-W2nu}.\\
By {\sc Step 1}, we can subtract a subsequence (denoted by $\{\gamma_n\}_{n\in \mathbb{N}}$ for simplicty) so that $\gamma_n\rightarrow \gamma$ narrowly in $\mathcal{P}(\mathbb{R}^{4d})$. Since the function $(x,x',\omega,\omega')\in \mathbb{R}^{4d}\longmapsto \vert x-x'\vert^2$ is lower semicontinuous and lower bounded then a standard lower semicontinuity argument \cite[Lemma 5.1.7]{AGS-08} implies
\begin{equation}\label{E-lsc-W2nu-finer}
\int_{\mathbb{R}^{4d}}\vert x-x'\vert^2\,d\gamma\leq \liminf_{n\rightarrow \infty}\int_{\mathbb{R}^{4d}}\vert x-x'\vert^2\,d\gamma_n=\liminf_{n\rightarrow\infty}W_{2,\nu}^2(\mu_n,\widetilde{\mu}_n),
\end{equation}
where in the last identity we have used that $\gamma_n\in\Gamma_{o,\nu}(\mu_n,\widetilde{\mu}_n)$. In particular, since $\gamma$ is a competitor in $W_{2,\nu}(\mu,\widetilde{\mu})$, then the above inequality implies
$$W_{2,\nu}(\mu,\widetilde{\mu})\leq \liminf_{n\rightarrow \infty}W_{2,\nu}(\mu_n,\widetilde{\mu}_n),$$
over the above convergence subsequence $\gamma_n\rightarrow \gamma$. Notice that the above argument can be repeated for any subsequence of $\{\mu_n\}_{n\in \mathbb{N}}$ and $\{\widetilde{\mu}_n\}_{n\in \mathbb{N}}$. Then, it implies that the above property holds for the full sequence, thus yielding \eqref{E-lsc-W2nu}.

\medskip

$\diamond$ {\sc Step 3}. Optimality of limiting plans.\\
Narrow relative compactness of $\{\gamma_n\}_{n\in{\mathbb N}}$ follows from {\sc Step 1}. Let $\gamma\in \Gamma_\nu(\mu,\widetilde{\mu})$ be any narrow limit of an appropriate subsequence again denoted by $\{\gamma_n\}_{n\in{\mathbb N}}$ for simplicity. Our goal is to show that $\gamma\in \Gamma_{o,\nu}(\mu,\widetilde{\mu})$, that is, $\gamma^\omega\in \Gamma_o(\mu^\omega,\widetilde{\mu}^\omega)$ for $\nu$-a.e. $\omega\in \mathbb{R}^d$. 

As discussed in Remark \ref{R-narrow-vs-stable} the narrow and stable topologies agree in $\mathcal{P}_\nu(\mathbb{R}^{2d})$ and therefore we can apply \cite[Theorem 4.12]{Ba-99} (see also \cite[Theorem 4.3.12]{CRV-04}), which yields the Kuratowski convergence of the supports of almost all of fibers $\gamma^\omega_n$. Specifically, we have
\begin{equation}\label{E-fibered-Kuratowski}
\supp\gamma^\omega\subseteq \Ls_{n\rightarrow\infty}\supp\gamma^\omega_n,
\end{equation}
for $\nu$-a.e. $\omega\in \mathbb{R}^d$ (say on a $\nu$-full measure set $F\subseteq \mathbb{R}^d$). Here, $\Ls$ denotes the Kuratowski superior limit, which is defined for any sequence of sets $A_n$ as $\Ls_{n\rightarrow\infty}A_n:=\cap_{n\in \mathbb{N}}\overline{\cup_{m\geq n}A_m}$. 

Fix any $\omega\in F$ and take any finite collection of points $(x_1,x_1'),\ldots,(x_k,x_k')\in \supp\gamma^\omega$. Our goal is to show the cyclical monotonicity of such a collection of points. The above Kuratowski convergence \eqref{E-fibered-Kuratowski} allows extracting a subsequence $\{\gamma^\omega_{\kappa(n)}\}_{n\in \mathbb{N}}$ (depending on the fixed $\omega$) and also $(x_{i,n},x_{i,n}')\in \supp\gamma_{\kappa(n)}^\omega$ for all $i=1,\ldots,k$ such that $(x_{i,n},x_{i,n}')\rightarrow (x_i,x_i')$ as $n\rightarrow \infty$. The subsequence may depend on the specific $(x_i,x_i')$, but since there are finitely many, we can always take a common subsequence. Since $\gamma_{\kappa(n)}^\omega\in \Gamma_o(\mu_{\kappa(n)}^\omega,\widetilde{\mu}_{\kappa(n)}^\omega)$, then
$$\sum_{i=1}^k |x_{i,n}-x_{i,n}'|^2\leq \sum_{i=1}^k |x_{i,n}-x_{\sigma(i),n}'|^2,$$
for all permutation $\sigma$ of $\{1,\ldots,k\}$ and all $n\in \mathbb{N}$. Passing to the limit in the above inequality shows that $\supp\gamma^\omega$ is cyclically monotone for the quadratic cost, and this ends the proof.
\end{proof}

\begin{rem}[Narrow convergence and fiberwise Kuratowski convergence]\label{R-narrow-fiberwise-Kuratowski-convergence}
The use of the fiberwise Kuratowski convergence \eqref{E-fibered-Kuratowski} in the above proof could appear to contradict Remark \ref{R-narrow-vs-fibers} at the first glance, since the latter provides examples of narrowly convergent sequences of transference plans, which are not fiberwise narrowly convergent $\gamma_n^\omega\rightarrow \gamma^\omega$. However, there is no such contradiction in view of the exact relations below:

\medskip

    \begin{center}
    \begin{tikzcd}[arrows=Rightarrow, column sep=1.3cm, row sep=1.3cm, every arrow/.append style={shift left=0.8ex}]
    \mu_n\rightarrow\mu\mbox{ narrowly} \quad
    \arrow[negated]{d}
    \arrow{r}
    \arrow[shorten <= 15pt, shorten >= 25pt]{dr}
    & \quad \mbox{supp}\,\gamma\subseteq {\rm Ls}_{n\rightarrow \infty}\mbox{supp}\,\gamma_n 
    \arrow[negated]{l}
    \arrow[negated]{d}
    \\
    \mu_n^\omega\rightarrow\mu^\omega\mbox{ narrowly} \quad
    \arrow{r}
    \arrow{u}
    & \quad \mbox{supp}\,\gamma^\omega\subseteq {\rm Ls}_{n\rightarrow \infty}\mbox{supp}\,\gamma_n^\omega
    \arrow[negated]{l}
    \arrow{u}
    \arrow[negated,shorten <= 25pt, shorten >= 15pt]{ul}
    \end{tikzcd}
    \end{center}

\medskip
    
In particular, note that, whilst both the narrow convergence and the fiberwise narrow convergence imply the fiberwise Kuratowski convergence, the latter does not imply any of the former. This can be further explained by noting that, whilst narrow convergence does not imply fiberwise narrow convergence in general, it does imply a certain fiberwise approximation of the supports due to the inherent oscillatory behavior of the stable convergence.
\end{rem}

\begin{rem}[Narrow vs $W_{2,\nu}$ convergence]\label{R-characterization-convergence-W2nu}
It is unclear if the classical characterisation of the $W_2$ convergence \cite[Theorem 6.9]{V-09} as narrow convergence supplemented with convergence of second order moments can be extended to the fibered case. Here we make the following observation. It is clear that if $\mu_n\rightarrow\mu$ in $W_{2,\nu}$ then $\mu_n\rightarrow\mu$ narrowly and 
\begin{equation*}
\lim_{n\rightarrow\infty}\int_{\mathbb{R}^{2d}}\vert x\vert^2\,d\mu_n(x,\omega)=\int_{\mathbb{R}^{2d}}\vert x\vert^2\,d\mu(x,\omega).
\end{equation*}
Whether it is possible to reverse the above implication seems unlikely. We refer to Example \ref{E-Rademacher} for a counterexample. In fact, by Vitali's convergence theorem $\mu_n\rightarrow\mu$ in $W_{2,\nu}$ if and only if the sequence of functions $\omega\mapsto W_2^2(\mu^\omega_n,\mu^\omega)$ converges to $0$ in measure $\nu$ and is uniformly integrable with respect to $\nu$, which implies $\mu_n^\omega\rightarrow\mu^\omega$ narrowly for $\nu$-a.e. $\omega\in \mathbb{R}^d$ up to subsequence.
\end{rem}

\begin{eje}[A Rademacher sequence]\label{E-Rademacher}
Set $\nu(\omega):=\chi_{(0,1)}(\omega)\,d\omega$ and define $u_n:(0,1)\longrightarrow \mathbb{R}$ by $u_n(\omega):=
{\rm sign}(\sin(2^n\pi\omega))$. Define the Young measures $\mu_n,\mu\in \mathcal{P}_{2,\nu}(\mathbb{R}^2)$ given by 
$$\mu_n(x,\omega):=\delta_{u_n(\omega)}(x)\otimes \nu(\omega),\quad \mu(x,\omega):=\frac{1}{2}(\delta_{-1}(x)+\delta_{1}(x))\otimes \nu(\omega).$$
By \cite[Example 3.2, Remark 4.2]{Ba-99} (see also \cite[Example 3.1.4]{CRV-04}) we have that $\mu_n\rightarrow \mu$ narrowly as $n\rightarrow\infty$, and also by inspection we obtain
$$\int_{\mathbb{R}^{2d}}|x|^2\,d\mu_n=1\rightarrow 1=\int_{\mathbb{R}^{2d}}|x|^2\,d\mu.$$
However, it is not true that $\mu_n^\omega\rightarrow \mu^\omega$ narrowly for {\it a.e.} $\omega\in (0,1)$.
\end{eje}

\subsection{Absolutely continuous curves and tangent bundle}

In this section, we recall the concept of absolutely continuous curves over a Polish space and provide a dynamic representation of absolutely continuous curves over the fibered quadratic Wasserstein space $(\mathcal{P}_{2,\nu}(\RR^{2d}),W_{2,\nu})$. Our result is reminiscent of the analogous one for absolutely continuous curves over the classical quadratic Wasserstein space $(\mathcal{P}_2(\RR^d),W_2)$ in terms of continuity equations, see \cite[Theorem 8.3.1]{AGS-08}. Interestingly, absolutely continuous curves along the fibered space solve a continuity equation where the transportation with respect to the variable $\omega$ is deprived and it only occurs with respect to the variable $x$ in a fiberwise way.

\begin{defi}[Absolutely continuous curves]\label{D-ac}
Let $({\mathbb X},d_\mathbb{X})$ be a Polish space and $p\in [1,\infty]$. We say that a curve $\bm{\mu}$ over $\mathbb{X}$ is $p$-absolutely continuous in $(a,b)$ with respect to $d_\mathbb{X}$ if there exists $\delta\in L^p(a,b)$ such that
\begin{align}\label{contm}
d_\mathbb{X}(\mu_t,\mu_s)\leq \int_s^t \delta(\tau)d\tau,
\end{align}
for all $a<s\leq t<b$. The space of all $p$-absolutely continuous curves with respect to $d_\mathbb{X}$ will be denoted by $AC^p(a,b;(\mathbb X,d_\mathbb{X}))$. For $p=1$, curves will be simply called {\rm absolutely continuous curves} and denoted by $AC(a,b;(\mathbb{X},d_\mathbb{X}))\equiv AC^1(a,b;(\mathbb{X},d_\mathbb{X}))$, whilst for $p=\infty$, curves will be called Lipschitz continuous and we denote $AC^\infty(a,b;(\mathbb{X},d_\mathbb{X}))\equiv {\rm Lip}(a,b;(\mathbb{X},d_\mathbb{X}))$.
\end{defi}

\begin{pro}[Metric derivative]\label{md}
Let $({\mathbb X},d_\mathbb{X})$ be a Polish space, $p\in [1,\infty]$ and let $\bm{\mu}\in AC^p(a,b;(\mathbb{X},d_{\mathbb{X}}))$. Then the following limit
\begin{align*}
\vert\bm{\mu}'\vert_{d_\mathbb{X}}(t):=\lim_{s\to t}\frac{d_\mathbb{X}(\mu_t,\mu_s)}{|t-s|},
\end{align*}
referred to as the metric derivative of $\bm{\mu}$ with respect to $d_\mathbb{X}$, exists for a.e. $t\in[a,b]$. Moreover $t\in (a,b)\mapsto \vert \bm{\mu}'\vert_{d_\mathbb{X}}(t)$ belongs to $L^p(a,b)$ and is the minimal admissible function in the right-hand side of \eqref{contm}.
\end{pro}

We refer to \cite[Theorem 1.1.2]{AGS-08} for the proof of the above proposition. When the distance $d_\mathbb{X}$ is clear from the context, we will simplify notation in Definition \ref{D-ac} and Proposition \ref{md} and will write $AC^p(a,b;\mathbb{X})$ and $\vert \bm{\mu}'\vert(t)$ for simplicity. Next, we proceed with the characterization of curves in $AC^p(a,b,\mathcal{P}_{2,\nu}(\mathbb{R}^{2d}))$. To this end, we first introduce the candidate for the tangent space at any measure in $\mathcal{P}_{2,\nu}(\mathbb{R}^{2d})$.

\begin{defi}[Tangent space]\label{D-tan}
Consider $\nu\in \mathcal{P}(\mathbb{R}^d)$ and $\mu\in \mathcal{P}_{2,\nu}(\mathbb{R}^{2d})$. We define
\begin{equation}\label{tan}
{\rm Tan}_\mu(\mathcal{P}_{2,\nu}(\mathbb{R}^{2d})):=\overline{\{\nabla_x\varphi:\,\varphi\in C^\infty_c(\mathbb{R}^{2d})\}}^{L^2_\mu(\mathbb{R}^{2d},\mathbb{R}^d)}.
\end{equation}
\end{defi}

Note that in Definition \ref{D-tan} gradients are computed only with respect to $x$ whilst $\varphi\in C^\infty_c(\mathbb{R}^{2d})$ depends on both $x$ and $\omega$. This induces a significant difference with the tangent bundle of the classical Wasserstein space $(\mathcal{P}_2(\mathbb{R}^{2d}),W_2)$, where gradients are computed with respect to both variables. To better understand the structure in \eqref{tan} we note that, similar to the classical case, there exists a canonical projection of generic $L^2_\mu(\mathbb{R}^{2d},\mathbb{R}^d)$ vector fields onto the tangent space ${\rm Tan}_\mu(\mathcal{P}_{2,\nu}(\mathbb{R}^{2d}))$. More precisely, the latter can be regarded as the orthogonal of a certain subspace of $L^2_\mu(\mathbb{R}^{2d},\mathbb{R}^d)$. Using the same arguments as in \cite[Lemma 8.4.2]{AGS-08} supported by the Hilbert projection theorem we obtain the following.

\begin{lem}[Projection into ${\rm Tan}_\mu(\mathcal{P}_{2,\nu}(\mathbb{R}^{2d}))$]\label{L-projection-tangent-space}
Consider $\nu\in \mathcal{P}(\mathbb{R}^d)$, $\mu\in \mathcal{P}_{2,\nu}(\mathbb{R}^{2d})$ and define the subspace of $L^2_\mu(\mathbb{R}^{2d},\mathbb{R}^d)$ divergence-free vector fields with respect to $x$, {\it i.e.},
$$\mathcal{X}_\mu:=\{\boldsymbol{w}\in L^2_\mu(\mathbb{R}^{2d},\mathbb{R}^d):\,\divop_x(\boldsymbol{w}\mu)=0\},$$
where the divergence is considered in distributional sense. Then, 
\begin{equation}\label{tan-3}
{\rm Tan}_\mu(\mathcal{P}_{2,\nu}(\mathbb{R}^{2d}))=\mathcal{X}_\mu^\perp.
\end{equation}
In particular, the orthogonal projection $\Pi_\mu:L^2_\mu(\mathbb{R}^{2d},\mathbb{R}^d)\rightarrow {\rm Tan}_\mu(\mathcal{P}_{2,\nu}(\mathbb{R}^{2d})),$ verifies the following variational selection principle:
$$\Vert \Pi_\mu[\boldsymbol{u}]\Vert_{L^2_\mu(\mathbb{R}^{2d},\mathbb{R}^d)}\leq \Vert \boldsymbol{u}+\boldsymbol{w}\Vert_{L^2_\mu(\mathbb{R}^{2d},\mathbb{R}^d)},$$
for any $\boldsymbol{u}\in L^2_\mu(\mathbb{R}^{2d},\mathbb{R}^d)$ and any $\boldsymbol{w}\in \mathcal{X}_\mu$.
\end{lem}

Using the above lemma, we arrive at the following alternative representation of the tangent space ${\rm Tan}_\mu(\mathcal{P}_{2,\nu}(\mathbb{R}^{2d}))$, where the fibered role of $\omega$ becomes apparent once again.

\begin{pro}\label{P-tan}
Consider $\nu\in \mathcal{P}(\mathbb{R}^d)$ and $\mu\in \mathcal{P}_{2,\nu}(\mathbb{R}^{2d})$. Then,
\begin{equation}\label{tan-2}
{\rm Tan}_\mu(\mathcal{P}_{2,\nu}(\mathbb{R}^{2d}))=\{\boldsymbol{u}\in L^2_\mu(\mathbb{R}^{2d},\mathbb{R}^d):\,\boldsymbol{u}(\cdot,\omega)\in {\rm Tan}_{\mu^\omega}(\mathcal{P}_2(\mathbb{R}^d)) \mbox{ for }\nu\mbox{-a.e. }\omega\in \mathbb{R}^d\}.
\end{equation}
\end{pro}

\begin{proof}
Let us assume that $\boldsymbol{u}\in {\rm Tan}_\mu(\mathcal{P}_{2,\nu}(\mathbb{R}^{2d}))$. By definition there exists a sequence of test functions $\{\varphi_n\}_{n\in \mathbb{N}}\subseteq C^\infty_c(\mathbb{R}^{2d})$ such that $\nabla_x\varphi_n\rightarrow \boldsymbol{u}$ in $L^2_\mu(\mathbb{R}^{2d},\mathbb{R}^d)$. Then, for an appropriate subsequence $\{\varphi_{\sigma(n)}\}_{n\in \mathbb{N}}$ we obtain $\nabla_x\varphi_{\sigma(n)}(\cdot,\omega)\rightarrow\boldsymbol{u}(\cdot,\omega)$ in $L^2_{\mu^\omega}(\mathbb{R}^d,\mathbb{R}^d)$ for $\nu$-a.e. $\omega\in \mathbb{R}^d$. By definition we conclude that $\boldsymbol{u}(\cdot,\omega)\in {\rm Tan}_{\mu^\omega}(\mathcal{P}_2(\mathbb{R}^d))$ for $\nu$-a.e. $\omega\in \mathbb{R}^d$. Conversely, assume that the latter holds and take any $\boldsymbol{w}\in \mathcal{X}_\mu$. Since $C^\infty_c(\mathbb{R}^d)$ is separable under the topology of $C^1_c(\mathbb{R}^d)$, then we build special test functions of the form $\varphi_n(x,\omega)=\phi_n(x)\psi(\omega)$ for a generic $\psi\in C^\infty_c(\mathbb{R}^d)$ and a dense subset $\{\phi_n\}_{n\in \mathbb{N}}\subseteq C^\infty_c(\mathbb{R}^d)$. Testing the distributional equation for $\boldsymbol{w}$ against such a family and using the disintegration Theorem \ref{dis} imply that $\divop_x(\boldsymbol{w}(\cdot,\omega)\mu^\omega)=0$ in the sense of distributions for $\nu$-a.e. $\omega\in \mathbb{R}^d$. By \cite[Lemma 8.4.2]{AGS-08} we then obtain
$$\int_{\mathbb{R}^d}\boldsymbol{u}(x,\omega)\cdot \boldsymbol{w}(x,\omega)\,d\mu^\omega(x)=0,$$
for $\nu$-a.e. $\omega\in \mathbb{R}^d$. Hence, integrating against $\nu$ yields $\boldsymbol{u}\in \mathcal{X}_\mu^\perp$ by the arbitrariness of $\boldsymbol{w}\in\mathcal{X}_\mu$. We then conclude that $\boldsymbol{u}\in {\rm Tan}_\mu(\mathcal{P}_{2,\nu}(\mathbb{R}^{2d}))$ by Lemma \ref{L-projection-tangent-space}.
\end{proof}

The above choice for the tangent space will be clarified in the following result.

\begin{pro}\label{equiv}
Set $\nu\in \mathcal{P}(\mathbb{R}^d)$, let $\boldsymbol{\mu}$ be a curve in $\mathcal{P}_{2,\nu}(\mathbb{R}^{2d})$ and consider $p\in [1,\infty]$. 

\begin{enumerate}[label=(\roman*)]
\item If $\boldsymbol{\mu}$ belongs to the space $AC^p(0,T;\mathcal{P}_{2,\nu}(\mathbb{R}^{2d}))$, then there exists a Borel family of vector fields $\boldsymbol{u}_t\in L^2_{\mu_t}(\mathbb{R}^{2d},\mathbb{R}^d)$ with $\int_0^T\Vert \boldsymbol{u}_t\Vert_{L^2_{\mu_t}(\mathbb{R}^{2d},\mathbb{R}^d)}^p\,dt<\infty$ such that $(\mu_t,\boldsymbol{u}_t)$ verifies
\begin{equation}\label{E-continuity-equation}
\partial_t\mu+\divop_x(\boldsymbol{u}\mu)=0,\quad t\geq 0,\,(x,\omega)\in \mathbb{R}^{2d},
\end{equation}
in the sense of distributions and $\Vert \boldsymbol{u}_t\Vert_{L^2_{\mu_t}(\mathbb{R}^{2d},\mathbb{R}^d)}\leq \vert \boldsymbol{\mu}'\vert_{W_{2,\nu}}(t)$ for a.e. $t\in [0,T]$. In addition, the vector field can be taken so that $\boldsymbol{u}_t\in {\rm Tan}_{\mu_t}(\mathcal{P}_{2,\nu}(\mathbb{R}^{2d}))$ for a.e. $t\in (0,T)$.

\item Conversely, if $(\mu_t,\boldsymbol{u}_t)$ verifies \eqref{E-continuity-equation} in the sense of distribution for some Borel family of vector fields $\boldsymbol{u}_t   \in L^2_{\mu_t}(\mathbb{R}^{2d},\mathbb{R}^d)$ such that $\int_0^T\Vert \boldsymbol{u}_t\Vert_{L^2_{\mu_t}(\mathbb{R}^{2d},\mathbb{R}^d)}^p\,dt<\infty$, then $\boldsymbol{\mu}$ belongs to $AC^p(0,T;\mathcal{P}_{2,\nu}(\mathbb{R}^{2d}))$ and $\vert \boldsymbol{\mu}'\vert_{W_{2,\nu}}(t)\leq \Vert \boldsymbol{u}_t\Vert_{L^2_{\mu_t}(\mathbb{R}^{2d},\mathbb{R}^d)}$ for a.e. $t\in (0,T)$.
\end{enumerate}
Moreover, $\boldsymbol{\mu}\in AC^p(0,T;\mathcal{P}_{2,\nu}(\mathbb{R}^{2d}))$ if, and only if, there exists $\widehat{\delta}\in L^p(0,T;L^2_\nu(\mathbb{R}^d))$ such that
\begin{equation}\label{E-fibered-ac}
W_2(\mu_t^\omega,\mu_s^\omega)\leq \int_s^t \widehat{\delta}(\tau,\omega)\,d\tau,
\end{equation}
for $0\leq s\leq t\leq T$ and $\nu$-a.e. $\omega\in \mathbb{R}^d$. In particular, $\boldsymbol{\mu}^\omega\in AC^p(0,T;\mathcal{P}_2(\mathbb{R}^d))$ for $\nu$-a.e. $\omega\in \mathbb{R}^d$.
\end{pro}

Note that once \eqref{E-continuity-equation} holds for some $\boldsymbol{u}_t\in L^2_{\mu_t}(\mathbb{R}^{2d},\mathbb{R}^d)$, then Lemma \ref{L-projection-tangent-space} guarantees that it also holds with $\boldsymbol{u}_t$ replaced by the tangent vector $\Pi_{\mu_t}[\boldsymbol{u}_t]\in{\rm Tan}_{\mu_t}(\mathcal{P}_{2,\nu}(\mathbb{R}^{2d}))$ because $\boldsymbol{u}_t-\Pi_{\mu_t}[\boldsymbol{u}_t]\in \mathcal{X}_{\mu_t}$. 
If $\nu\in \mathcal{P}_2(\mathbb{R}^d)$, then $AC^2(0,T;\mathcal{P}_{2,\nu}(\mathbb{R}^{2d}))\subseteq AC^2(0,T;\mathcal{P}_2(\mathbb{R}^{2d}))$, which was evident by the embedding in Proposition \ref{P-scaled-Wasserstein-distances}. The proof of Proposition \ref{equiv} requires a subtle adaptation of the classical result \cite[Theorem 8.3.1]{AGS-08} for the quadratic Wasserstein space that bears in mind the fibered nature of the new space $(\mathcal{P}_{2,\nu}(\mathbb{R}^{2d}),W_2)$. We sketch it here for clarity.

\medskip

\begin{proof}[Proof of Proposition \ref{equiv}]

~
\medskip

$\diamond$ {\sc Step 1}: $\boldsymbol{\mu}\in AC^p(0,T;\mathcal{P}_{2,\nu}(\mathbb{R}^{2d}))$ implies the continuity equation \eqref{E-continuity-equation}.\\
Let us assume that $\boldsymbol{\mu}$ belongs to $AC^p(0,T;\mathcal{P}_{2,\nu}(\mathbb{R}^{2d})) $ and set any test function $\varphi\in C^\infty_c(\mathbb{R}^{2d})$. Now define the bounded and upper-semicontinuous function
$$H_\varphi(x,x',\omega):=\left\{
\begin{array}{ll}
\displaystyle\frac{\vert \varphi(x,\omega)-\varphi(x',\omega)\vert}{\vert x-x'\vert}, & \mbox{if }x\neq x',\\
\displaystyle\vert \nabla_x\varphi(x,\omega)\vert, & \mbox{if }x=x'.
\end{array}
\right.$$
Then, using the above function $H_\varphi$ and the Cauchy--Schwarz inequality, we obtain
\begin{align*}
\left\vert\int_{\mathbb{R}^{2d}}\varphi\,d(\mu_{t+h}-\mu_t)\right\vert&=\left\vert\int_{\mathbb{R}^d}\int_{\mathbb{R}^{2d}} (\varphi(x,\omega)-\varphi(x',\omega))\,d\gamma_{t+h,t}^\omega(x,x')\,d\nu(\omega)\right\vert\\
&\leq W_{2,\nu}(\mu_{t+h},\mu_t) \left(\int_{\mathbb{R}^{4d}}H_\varphi(x,x',\omega)^2\,d\gamma_{t+h,t}(x,x',\omega,\omega')\right)^{1/2}.
\end{align*}
for every $t\in (0,T)$ and $h\in (-t,T-t)$.
Here, $\gamma_{t+h,t}\in \Gamma_{o,\nu}(\mu_{t+h},\mu_t)$ that is,
$$\gamma_{t+h,t}(x,x',\omega,\omega'):=\gamma_{t+h,t}^\omega(x,x')\otimes \nu(\omega)\otimes \delta_{\omega}(\omega'),$$
where $\{\gamma_{t+h,t}^\omega\}_{\omega\in \mathbb{R}^d}$ is any Borel family of probability measures with $\gamma_{t+h,t}^\omega\in \Gamma_o(\mu_{t+h}^\omega,\mu_t^\omega)$ for $\nu$-a.e. $\omega\in \mathbb{R}^d$ (see Lemma \ref{L-measurable-selection}). By the assumptions on $\boldsymbol{\mu}$ and the boundedness of $H_\varphi$ the function $t\in(0,T)\longmapsto \int_{\mathbb{R}^{2d}}\varphi\,d\mu_t(x,\omega)$ is absolutely continuous and
\begin{equation}\label{insert}
\left\vert\frac{d}{dt}\int_{\mathbb{R}^{2d}}\varphi\,d\mu_t\right\vert\leq \vert \boldsymbol{\mu}'\vert_{W_{2,\nu}}(t)\,\limsup_{h\rightarrow 0}\left(\int_{\mathbb{R}^{4d}}H_\varphi(x,x',\omega)^2\,d\gamma_{t+h,t}(x,x',\omega,\omega')\right)^{1/2},
\end{equation}
for a.e. $t\in (0,T)$. Moreover by Taylor's formula, we have
\begin{align*}
    H_\varphi(x,x',\omega) = \frac{|\nabla_x\varphi(x',\omega)\cdot (x-x') + \mathcal{R}(x,x',\omega)|}{|x-x'|},
\end{align*}
whenever $x\neq x'$, for a remainder $\mathcal{R}$ verifying $|\mathcal{R}(x,x',\omega)|\leq \|\varphi\|_{C^2(\mathbb{R}^{2d})} |x-x'|^2$. Thus,
\begin{align*}
    H_\varphi(x,x',\omega)^2 \leq |\nabla_x\varphi(x',\omega)|^2 + 2\|\varphi\|_{C^2(\mathbb{R}^{2d})}\Vert\varphi\Vert_{C^1(\mathbb{R}^{2d})}|x-x'| + \|\varphi\|_{C^2(\mathbb{R}^{2d})}^2|x-x'|^2,
\end{align*}
for every $x,x',\omega\in \mathbb{R}^d$. Therefore, by integrating against $\gamma_{t+h,t}$ we obtain
\begin{multline*}
\int_{\mathbb{R}^{4d}}H_\varphi(x,x',\omega)^2\, d\gamma_{t+h,t}(x,x',\omega,\omega')
\leq \int_{\RR^{2d}}|\nabla_x\varphi(x,\omega)|^2d\mu_t(x,\omega)\\
+2\Vert \varphi\Vert_{C^2(\mathbb{R}^{2d})} \Vert\varphi\Vert_{C^1(\mathbb{R}^{2d})}W_{2,\nu}(\mu_{t+h},\mu_t)+\Vert \varphi\Vert_{C^2(\mathbb{R}^{2d})}^2 W_{2,\nu}^2(\mu_{t+h},\mu_t).
\end{multline*}
Taking $\limsup$ as $h\rightarrow 0$ and noting that $\lim_{h\rightarrow 0} W_{2,\nu}(\mu_{t+h},\mu_t)=0$ by hypothesis, we obtain
$$\limsup_{h\rightarrow 0}\left(\int_{\mathbb{R}^{4d}}H_\varphi(x,x',\omega)^2\,d\gamma_{t+h,t}(x,x',\omega,\omega')\right)^{1/2}\leq \Vert \nabla_x\varphi\Vert_{L^2_{\mu_t}(\mathbb{R}^{2d},\mathbb{R}^d)},$$
which, together with \eqref{insert}, yields
\begin{equation}\label{E-equiv-ac-weak}
\left\vert \frac{d}{dt}\int_{\mathbb{R}^{2d}}\varphi\,d\mu_t\right\vert\leq \vert \boldsymbol{\mu}'\vert_{W_{2,\nu}}(t)\,\Vert \nabla_x\varphi\Vert_{L^2_{\mu_t}(\mathbb{R}^{2d},\mathbb{R}^d)},
\end{equation}
for a.e. $t\in (0,T)$. Let us now define the vector space $\mathcal{X}_0:=\{\nabla_x\varphi:\,\varphi\in C^\infty_c((0,T)\times\mathbb{R}^{2d})\}$ and the linear functional $\mathcal{L}:\mathcal{X}_0\longrightarrow \mathbb{R}$ given by
$$\mathcal{L}(\nabla_x\varphi):=-\int_0^T\int_{\mathbb{R}^{2d}}\partial_t\varphi(t,x,\omega)\,d\mu_t(x,\omega)\,dt,\quad \nabla_x\varphi\in \mathcal{X}_0.$$
Moreover, by \eqref{E-equiv-ac-weak} and H\"{o}lder's inequality, we obtain
$$\vert \mathcal{L}(\nabla_x\varphi)\vert\leq \Vert \vert \mu'\vert_{W_{2,\nu}}\Vert_{L^p(0,T)}\, \Vert \nabla_x\varphi\Vert_{L^{p'}(0,T;L^2_\mu(\mathbb{R}^{2d},\mathbb{R}^d))},$$
for every $\nabla_x\varphi\in \mathcal{X}_0$, where $p'=\frac{p}{p-1}$ is the Lebesgue conjugate exponent. Hence $\mathcal{L}$ is a bounded linear operator with respect to the norm of the space $L^{p'}(0,T;L^2_\mu(\mathbb{R}^{2d},\mathbb{R}^d))$ and it can be extended by continuity to $\mathcal{X}:=\overline{\mathcal{X}_0}^{L^{p'}(0,T;L^2_\mu(\mathbb{R}^{2d},\mathbb{R}^d))}$, {\it cf.} \cite[Remark 8.1.1]{AGS-08}. By the Riesz representation theorem for Lebesgue--Bochner spaces (see \cite[Theorem IV.1.1]{DU-77}), we then claim that there exists $\boldsymbol{u}\in L^p(0,T;L^2_\mu(\mathbb{R}^{2d},\mathbb{R}^d))$ such that
$$-\int_0^T\int_{\mathbb{R}^{2d}}\partial_t \varphi(x,\omega)\,d\mu_t(x,\omega)\,dt=\int_0^T\int_{\mathbb{R}^{2d}} \boldsymbol{u}_t(x,\omega)\cdot \nabla_x\varphi(t,x,\omega)\,d\mu_t(x,\omega)\,dt,$$
for each $\varphi\in C^\infty_c((0,T)\times \mathbb{R}^{2d})$. Then, $(\mu_t,\boldsymbol{u}_t)$ verifies the continuity equation \eqref{E-continuity-equation} in distributional sense. Moreover, as stated below the statement of Proposition \ref{equiv}, by Lemma \ref{L-projection-tangent-space}, we may assume without loss of generality that $\boldsymbol{u}_t$ belongs to ${\rm Tan}_{\mu_t}(\mathcal{P}_{2,\nu}(\mathbb{R}^{2d}))$ for a.e. $t\in (0,T)$. Finally, an easy cut-off argument shows that $\Vert\boldsymbol{u}_t\Vert_{L^2_{\mu_t}(\mathbb{R}^{2d},\mathbb{R}^d)}\leq \vert \boldsymbol{\mu}'\vert_{W_{2,\nu}}(t)$ for a.e. $t\in (0,T)$.

\medskip

$\diamond$ {\sc Step 2}: The continuity equation \eqref{E-continuity-equation} implies $\boldsymbol{\mu}\in AC^p(0,T;\mathcal{P}_{2,\nu}(\mathbb{R}^{2d}))$.\\
Assume that \eqref{E-continuity-equation} holds for $(\mu_t,\boldsymbol{u}_t)$, where $\boldsymbol{u}_t\in L^2_{\mu_t}(\mathbb{R}^{2d},\mathbb{R}^d)$ is a Borel family of vector fields with $\int_0^T\Vert \boldsymbol{u}_t\Vert_{L^2_{\mu_t}(\mathbb{R}^{2d},\mathbb{R}^d)}^p\,dt<\infty$.

\medskip

$\circ$ {\sc Step 2.1}: Parameter $p\in [1,2]$.\\ 
Note that the space $C^\infty_c((0,T)\times \mathbb{R}^d)$ is separable under the topology of $C^1_c((0,T)\times \mathbb{R}^d)$. Then, we can set a family of special test functions of the form $\varphi(t,x,\omega)=\phi_n(t,x)\psi(\omega)$ for generic $\psi\in C^\infty_c(\mathbb{R}^d)$ and a dense subset $\{\phi_n\}_{n\in \mathbb{N}}\subseteq C^\infty_c((0,T)\times \mathbb{R}^d)$. Writing \eqref{E-continuity-equation} in weak form against such a family of test functions and using the disintegration Theorem \ref{dis} we achieve
\begin{equation}\label{cont-eq-dis}
\partial_t\mu_t^\omega+\divop_x(\boldsymbol{u}_t(\cdot,\omega)\mu_t^\omega)=0,
\end{equation}
in the sense of distribution for $\nu$-a.e. $\omega\in \mathbb{R}^d$. In addition, since $p\leq 2$ ({\it i.e.} $\frac{2}{p}\geq 1$), then Minkowski's integral inequality yields
$$\left[\int_{\mathbb{R}^d}\left(\int_0^T\Vert \boldsymbol{u}_t(\cdot,\omega)\Vert_{L^2_{\mu_t^\omega}(\mathbb{R}^d,\mathbb{R}^d)}^p\,dt\right)^{2/p}\,d\nu(\omega)\right]^{p/2}\leq \int_0^T \Vert \boldsymbol{u}_t\Vert_{L^2_{\mu_t}(\mathbb{R}^{2d},\mathbb{R}^d)}^p\,dt<\infty.$$
In particular, we obtain
\begin{equation}\label{cont-eq-dis-norm}
\int_0^T \Vert \boldsymbol{u}_t(\cdot,\omega)\Vert_{L^2_{\mu_t^\omega}(\mathbb{R}^d,\mathbb{R}^d)}^p\,dt<\infty,
\end{equation}
for $\nu$-a.e. $\omega\in \mathbb{R}^d$. By virtue of \eqref{cont-eq-dis} and \eqref{cont-eq-dis-norm}, the classical result \cite[Theorem 8.3.1]{AGS-08} can be applied at any fiber $\omega\in \mathbb{R}^d$ implying that $\boldsymbol{\mu}^\omega$ belongs to $AC^p(0,T;\mathcal{P}_2(\mathbb{R}^d))$ for $\nu$-a.e. $\omega\in \mathbb{R}^d$ and $\vert ({\boldsymbol{\mu}^\omega})'\vert_{W_2}(t)\leq \Vert \boldsymbol{u}_t(\cdot,\omega)\Vert_{L^2_{\mu_t^\omega}(\mathbb{R}^d,\mathbb{R}^d)}$. This amounts to \eqref{E-fibered-ac} with 
$$\widehat{\delta}(t,\omega):=\Vert \boldsymbol{u}_t(\cdot,\omega)\Vert_{L^2_{\mu_t^\omega}(\mathbb{R}^d,\mathbb{R}^d)}.$$
Taking $L^2_\nu(\mathbb{R}^d)$ norms in \eqref{E-fibered-ac} and using Minkowski's integral inequality entails
\begin{equation}\label{E-fibered-ac-pre}
W_{2,\nu}(\mu_t,\mu_s)\leq \left[\int_{\mathbb{R}^d}\left(\int_s^t \widehat{\delta}(\tau,\omega)\,d\tau\right)^2\,d\nu(\omega)\right]^{1/2}\leq \int_s^t \Vert \widehat{\delta}(\tau,\cdot)\Vert_{L^2_\nu(\mathbb{R}^d)}\,d\tau=\int_s^t \Vert \boldsymbol{u}_\tau\Vert_{L^2_{\mu_\tau}(\mathbb{R}^{2d},\mathbb{R}^d)}\,d\tau,
\end{equation}
for every $0\leq s\leq t\leq T$. Since $\int_0^T\Vert \boldsymbol{u}_t\Vert_{L^2_{\mu_t}(\mathbb{R}^{2d},\mathbb{R}^d)}^p\,dt<\infty$, we have $\boldsymbol{\mu}\in AC^p(0,T;\mathcal{P}_{2,\nu}(\mathbb{R}^{2d})).$

\medskip

$\circ$ {\sc Step 2.2}: Parameter $p\in (2,\infty]$.\\
Notice that since $(0,T)$ is a finite interval, then our assumptions on the vector field imply that $\int_0^T \Vert \boldsymbol{u}_t\Vert_{L^2_{\mu_t}(\mathbb{R}^{2d},\mathbb{R}^d)}^2\,dt<\infty$. Then, by applying the above step we recover the inequality \eqref{E-fibered-ac-pre}. However, notice that we indeed have $\int_0^T\Vert \boldsymbol{u}_t\Vert_{L^2_{\mu_t}(\mathbb{R}^{2d},\mathbb{R}^d)}^p\,dt<\infty$ and this ends the proof.
\end{proof}

\medskip

In the following result, we derive the generic differentiability of $W_{2,\nu}$ along absolutely continuous curves, which extend the classical result in \cite[Theorem 8.4.7]{AGS-08}. 

\begin{pro}\label{T-dif}
Let $\boldsymbol{\mu}\in AC(0,T;\mathcal{P}_{2,\nu}(\mathbb{R}^{2d}))$, $\sigma\in \mathcal{P}_{2,\nu}(\mathbb{R}^{2d})$ and $\boldsymbol{u}_t\in {\rm Tan}_{\mu_t}(\mathcal{P}_{2,\nu}(\mathbb{R}^{2d}))$ be the vector field tangent to $\boldsymbol{\mu}$ characterized by Proposition \ref{equiv}. Then
\begin{equation}\label{dif}
\frac{d}{dt}\frac{1}{2}W_{2,\nu}^2(\mu_t,\sigma)=\int_{\mathbb{R}^{4d}}(x-x')\cdot \boldsymbol{u}_t(x,\omega)\, d\gamma_t(x,x',\omega,\omega'),
\end{equation}
for a.e. $t\in [0,T]$ and any $\gamma_t\in \Gamma_{o,\nu}(\mu_t,\sigma)$ optimal $\nu$-admissible transference plan for $W_{2,\nu}$.
\end{pro}

\begin{proof}
Let us define the following function
$$F(t,\omega):=\frac{1}{2}W_2^2(\mu_t^\omega,\sigma^\omega),\quad (t,\omega)\in [0,T]\times \mathbb{R}^d.$$
It is defined for a.e. $t\in [0,T]$ and $\nu$-a.e. $\omega\in \mathbb{R}^d$. Since $\boldsymbol{\mu}\in AC(0,T;\mathcal{P}_{2,\nu}(\mathbb{R}^{2d}))$ by assumption, then Proposition \ref{equiv} guarantees that $\boldsymbol{\mu}^\omega\in AC(0,T;\mathcal{P}_2(\mathbb{R}^d))$ for $\nu$-a.e. $\omega\in \mathbb{R}^d$. In particular,
\begin{equation}\label{E-ac-t}
F(\cdot,\omega)\in AC(0,T),\quad \mbox{for }\nu\mbox{-a.e. }\omega\in \mathbb{R}^d.
\end{equation}
In addition, by definition of $W_2$ we readily infer the following estimate
$$F(t,\omega)\leq \int_{\mathbb{R}^d}\vert x\vert^2\,d\mu_t^\omega(x)+\int_{\mathbb{R}^d}\vert x\vert^2\,d\sigma^\omega(x),$$
for a.e. $t\in (0,T)$ and $\nu$-a.e. $\omega\in \mathbb{R}^d$. Since $\mu_t,\sigma\in \mathcal{P}_{2,\nu}(\mathbb{R}^{2d})$ we conclude that
\begin{equation}\label{E-int-omega}
F(t,\cdot)\in L^1_\nu(\mathbb{R}^d),\quad \mbox{for a.e. }t\in [0,T].
\end{equation}
For a.e. $t\in [0,T]$, let us set any optimal $\nu$-admissible plan $\gamma_t\in \Gamma_{o,\nu}(\mu_t,\sigma)$ according to Definition \ref{D-admissible-plans-fibered} and  consider the velocity field $\boldsymbol{u}_t$ in Proposition \ref{equiv}. By virtue of the definition of ${\rm Tan}_{\mu_t}(\mathcal{P}_{2,\nu}(\mathbb{R}^{2d}))$ in \eqref{tan}, we notice that $\boldsymbol{u}_t(\cdot,\omega)\in {\rm Tan}_{\mu_t^\omega}(\mathcal{P}_2(\mathbb{R}^d))$. Again, recall that $\boldsymbol{\mu}^\omega\in AC(0,T;\mathcal{P}_2(\mathbb{R}^d))$ for $\nu$-a.e. $\omega\in \mathbb{R}^d$. Hence, we can readily apply the classical general differentiability property of $W_2$ to obtain
$$\frac{\partial F}{\partial t}(t,\omega)=\int_{\mathbb{R}^{2d}}(x-x')\cdot \boldsymbol{u}_t(x,\omega)\,d\gamma_t^\omega(x,x'),$$
for a.e. $t\in [0,T]$ and $\nu$-a.e. $\omega\in \mathbb{R}^d$, see \cite[Theorem 8.4.7]{AGS-08}. Integrating with respect to $t\in [0,T]$ and $\omega\in \mathbb{R}^d$ and using the Cauchy-Schwarz inequality we infer
$$\int_0^T\int_{\mathbb{R}^d}\left\vert\frac{\partial F}{\partial t}(t,\omega)\right\vert\,d\nu(\omega)\,dt\leq \sup_{0\leq t\leq T}W_{2,\nu}(\mu_t,\sigma)\,\int_0^T \Vert \boldsymbol{u}_t\Vert_{L^2(\mu_t)}\,dt.$$
By assumptions and Proposition \ref{equiv} we conclude that the right hand side in the above inequality is finite, {\it i.e.},
\begin{equation}\label{E-int-derivative-t-omega}
\frac{\partial F}{\partial t}\in L^1_{dt\otimes \nu}((0,T)\times \mathbb{R}^d).    
\end{equation}
Since hypothesis \eqref{E-ac-t}, \eqref{E-int-omega} and \eqref{E-int-derivative-t-omega} hold true, then the version for absolutely continuous functions of Leibniz' rule applies and we conclude that $\int_{\mathbb{R}^d} F(\cdot,\omega)\,d\nu(\omega)\in AC(0,T)$ and
$$\frac{d}{dt}\int_{\mathbb{R}^d}F(t,\omega)\,d\nu(\omega)=\int_{\mathbb{R}^d}\frac{\partial F}{\partial t}(t,\omega)\,d\nu(\omega),$$
for a.e. $t\in (0,T)$. This ends the proof.
\end{proof}

Note that in \eqref{dif} the above tangent vector $\boldsymbol{u}_t$ can actually be replaced by a generic Borel family of vector fields $\boldsymbol{u}_t\in L^2_{\mu_t}(\mathbb{R}^{2d},\mathbb{R}^d)$ with $\int_0^t\Vert \boldsymbol{u}_t\Vert_{L^2_{\mu_t}(\mathbb{R}^{2d},\mathbb{R}^d)}^2\,dt<\infty$ realizing the continuity equation \eqref{E-continuity-equation} as given in Proposition \ref{equiv}. Indeed, note that $\boldsymbol{w}_t:=\boldsymbol{u}_t-\Pi_{\mu_t}[\mu_t]$ lies in $\mathcal{X}_{\mu_t}$. Then, a straightforward extension of \cite[Proposition 8.5.4]{AGS-08} yields the following property.

\begin{pro}\label{P-tangent-barycentric-projections}
Consider any $\nu\in \mathcal{P}(\mathbb{R}^d)$ and $\mu_1,\mu_2\in \mathcal{P}_{2,\nu}(\mathbb{R}^{2d})$. Then,
$$\int_{\mathbb{R}^{4d}}(x-x')\cdot \boldsymbol{w}(x,\omega)\,d\gamma(x,x',\omega,\omega)=0,$$
for any $\boldsymbol{w}\in \mathcal{X}_{\mu_1}$ and each $\gamma\in \Gamma_{o,\nu}(\mu_1,\mu_2)$.
\end{pro}

We end this section by showing that the tangent vector $\boldsymbol{u}_t\in {\rm Tan}_{\mu_t}(\mathcal{P}_{2,\nu}(\mathbb{R}^{2d}))$ to  any $\boldsymbol{\mu}\in AC(0,T;\mathcal{P}_{2,\nu}(\mathbb{R}^{2d}))$ ({\it cf.} Proposition \ref{equiv}) can be recovered through the infinitesimal behaviour of $\nu$-admissible optimal transference plans along the
curve. This extends \cite[Proposition 8.4.6]{AGS-08}.

\begin{pro}\label{P-infinitesimal-vector-from-plans}
Consider any $\nu\in \mathcal{P}(\mathbb{R}^d)$ and $\boldsymbol{\mu}\in AC(0,T;\mathcal{P}_{2,\nu}(\mathbb{R}^{2d}))$. Let $\boldsymbol{u}_t\in {\rm Tan}_{\mu_t}(\mathcal{P}_{2,\nu}(\mathbb{R}^{2d}))$ be the associated vector fulfilling the continuity equation \eqref{E-continuity-equation} according to Proposition \ref{equiv}. Set $\gamma_{t,h}\in \Gamma_{o,\nu}(\mu_t,\mu_{t+h})$ for any $t\in (0,T)$ and any $h\in (-t,T-t)\setminus\{0\}$. Then,
$$\lim_{h\rightarrow 0}\big(\pi_x,\frac{1}{h}(\pi_{x'}-\pi_x),\pi_\omega,\pi_{\omega'}\big)_{\#}\gamma_{t,h}=((\pi_x,\boldsymbol{u}_t,\pi_\omega)_{\#}\mu_t)\otimes\delta_\omega(\omega')\quad \mbox{narrowly},$$
for a.e. $t\in (0,T)$.
\end{pro}

\begin{proof}
Let $\widetilde{\nu}(\omega,\omega'):=\nu(\omega)\otimes \delta_{\omega}(\omega')\in \mathcal{P}(\mathbb{R}^{2d})$ and note that we can set the associated space of fibered probability measures $\mathcal{P}_{\widetilde{\nu}}(\mathbb{R}^{4d})$ in Definition \ref{D-fibered-measures} endowed with the narrow topology. Similarly, we can define the fibered Wasserstein space $(\mathcal{P}_{2,\widetilde{\nu}}(\mathbb{R}^{4d}),W_{2,\widetilde{\nu}})$ according to Definition \ref{D-fibered-Wasserstein}. Let us define the subset $F\subseteq (0,T)$ consisting of $t\in (0,T)$ so that:
\begin{align}
&\lim_{h\rightarrow 0}\frac{W_{2,\nu}(\mu_{t+h},\mu_t)}{\vert h\vert}=\vert \boldsymbol{\mu}'\vert_{W_{2,\nu}}(t),\label{E-full-measure-1}\\
&\left.\frac{d}{ds}\right\vert_{s=t}\int_{\mathbb{R}^{2d}}\varphi\,d\mu_s=\int_{\mathbb{R}^{2d}}\nabla_x\varphi\cdot \boldsymbol{u}_t\,d\mu_t\quad \forall\varphi\in C^\infty_c(\mathbb{R}^{2d}),\label{E-full-measure-2}
\end{align}
assuming, in particular, that the limits on the left-hand side above exists.
By Proposition \ref{md}, Equation \eqref{E-continuity-equation} and the integrability property $\int_0^T\Vert \boldsymbol{u}_t\Vert_{L^2_{\mu_t}(\mathbb{R}^{2d},\mathbb{R}^d)}\,dt<\infty$ we guarantee that $F$ has full measure. In other words, the properties \eqref{E-full-measure-1} and \eqref{E-full-measure-2} hold for a.e. $t\in (0,T)$. Let us fix an  arbitrary $t\in F$ and denote for simplicity
\begin{equation}\label{E-19}
\widetilde{\gamma}_{t,h}:=\big(\pi_x,\frac{1}{h}(\pi_{x'}-\pi_x),\pi_\omega,\pi_{\omega'}\big)_{\#}\gamma_{t,h},\quad h\in (-t,T-t).
\end{equation}

$\diamond$ {\sc Step 1}: Narrow relative compactness.\\
By Markov's inequality and the definition \eqref{E-19} of $\widetilde{\gamma}_{t,h}$ we easily infer
$$
(\pi_{(x,x')\#}\widetilde{\gamma}_{t,h})(\mathbb{R}^{2d}\setminus B_R\times B_R)\leq\frac{1}{R^2}\left(W_{2,\nu}^2(\mu_t,\delta_0\otimes \nu)+\frac{W_{2,\nu}^2(\mu_{t+h},\mu_t)}{h^2}\right),
$$
for any $h\in (-t,T-t)$ and $R\in \mathbb{R}_+^*$. Thanks to \eqref{E-full-measure-1}, the right hand side converges to zero uniformly in $h$ as $R\rightarrow\infty$. Since $\pi_{(\omega,\omega')\#}\widetilde{\gamma}_{t,h}=\widetilde{\nu}$ is fixed and independent of $h$, then the full $\{\widetilde{\gamma}_{t,h}\}_{h\in (-t,T-t)}$ is uniformly tight. Hence, Prokhorov's theorem implies that $\{\widetilde{\gamma}_{t,h}\}_{h\in (-t,T-t)}$ is relatively compact in the narrow topology. In particular, there exist $h_n\rightarrow 0$ and $\widetilde{\gamma}_t\in \mathcal{P}_{\widetilde{\nu}}(\mathbb{R}^{4d})$ such that $\widetilde{\gamma}_{t,h_n}\rightarrow \widetilde{\gamma}_t$ as $n\rightarrow \infty$ narrowly.

\medskip

$\diamond$ {\sc Step 2}: Identification of the limit point $\widetilde{\gamma}_t$.\\
By the disintegration Theorem \ref{dis} let us write $\widetilde{\gamma}_t(x,x',\omega,\omega')=\mu_t(x,\omega)\otimes \widetilde{\gamma}_t^{x,\omega}(x')\otimes \delta_\omega(\omega')$ and define the associated barycentric projection $\widetilde{\boldsymbol{u}}_t\in L^2_{\mu_t}(\mathbb{R}^{2d},\mathbb{R}^d)$ ({\it cf.} Appendix \ref{Appendix-slope-subdifferential}) given by
\begin{equation}\label{E-18}
\widetilde{\boldsymbol{u}}_t(x,\omega):=\int_{\mathbb{R}^d}x'\,d\widetilde{\gamma}_t^{x,\omega}(x'),\quad (x,\omega)\in \mathbb{R}^{2d}.
\end{equation}
Fix any $\varphi\in C^\infty_c(\mathbb{R}^{2d})$ and use \eqref{E-full-measure-2} and the definition \eqref{E-19} of $\widetilde{\gamma}_{t,h_n}$ to obtain
\begin{align*}
\int_{\mathbb{R}^{2d}}\nabla_x\varphi\cdot \boldsymbol{u}_t\,d\mu_t
&=\lim_{n\rightarrow \infty}\frac{1}{h_n}\int_{\mathbb{R}^{4d}}(\varphi(x+h_nx',\omega)-\varphi(x,\omega))\,d\widetilde{\gamma}_{t,h_n}(x,x',\omega,\omega')\\
&=\lim_{n\rightarrow \infty}\left(\int_{\mathbb{R}^{4d}}\nabla_x\varphi(x,\omega)\cdot x'\,d\widetilde{\gamma}_{t,h_n}(x,x',\omega,\omega')+I_n\right)=\int_{\mathbb{R}^{2d}}\nabla_x\varphi\cdot\widetilde{\boldsymbol{u}}_t\,d\mu_t,
\end{align*}
where in the second identity we have used Taylor's formula with a residual term $I_n$ such that
$$\vert I_n\vert\leq h_n\Vert \varphi\Vert_{C^2(\mathbb{R}^{2d})}\int_{\mathbb{R}^{4d}}\vert x'\vert^2\,d\widetilde{\gamma}_{t,h_n}(x,x',\omega,\omega')=h_n\Vert \varphi\Vert_{C^2(\mathbb{R}^d)}\frac{W_{2,\nu}^2(\mu_t,\mu_{t+h_n})}{h_n^2},$$
(thus vanishing as $n\rightarrow \infty$ by \eqref{E-full-measure-1}) and in the last identity we have used the narrow convergence of $\widetilde{\gamma}_{t,h_n}$ and the definition \eqref{E-18} of $\widetilde{\boldsymbol{u}}_t$. Therefore, $\boldsymbol{u}_t-\widetilde{\boldsymbol{u}}_t\in \mathcal{X}_{\mu_t}$. In addition,
\begin{multline*}\Vert \widetilde{\boldsymbol{u}}_t\Vert_{L^2_{\mu_t}(\mathbb{R}^{2d},\mathbb{R}^d)}^2\leq \int_{\mathbb{R}^{4d}}\vert x'\vert^2\,d\widetilde{\gamma}_t\leq \liminf_{n\rightarrow\infty}\int_{\mathbb{R}^{4d}}\vert x'\vert^2\,d\widetilde{\gamma}_{t,h_n}\\
=\lim_{n\rightarrow\infty}\frac{W_{2,\nu}^2(\mu_t,\mu_{t+h_n})}{h_n^2}=\vert \boldsymbol{\mu}'\vert^2_{W_{2,\nu}}(t)=\Vert \boldsymbol{u}_t\Vert_{L^2_{\mu_t}(\mathbb{R}^{2d},\mathbb{R}^d)}^2,
\end{multline*}
where we have used Jensen's inequality, the Portmanteu Theorem \ref{T-Portmanteau}, property \eqref{E-full-measure-1} and Proposition \ref{equiv}. By Lemma \ref{L-projection-tangent-space} we conclude that $\widetilde{\boldsymbol{u}}_t=\boldsymbol{u}_t$. In particular, this implies that the first of the above inequalities must be an identity, that is,
\begin{align*}
0&=\int_{\mathbb{R}^{2d}}\int_{\mathbb{R}^d}\vert x'\vert^2\,d\widetilde{\gamma}_t^{x,\omega}(x')\,d\mu_t(x,\omega)-\int_{\mathbb{R}^{2d}}\left\vert\int_{\mathbb{R}^d}x'\,d\widetilde{\gamma}_t^{x,\omega}(x')\right\vert^2\,d\mu_t(x,\omega)\\
&=\int_{\mathbb{R}^{2d}}\int_{\mathbb{R}^d}\vert x'-\widetilde{\boldsymbol{u}}_t(x,\omega)\vert^2\,d\widetilde{\gamma}_t^{x,\omega}(x')\,d\mu_t(x,\omega).
\end{align*}
This of course implies that $\widetilde{\gamma}_t^{x,\omega}(x')=\delta_{\boldsymbol{u}_t(x,\omega)}(x')$ for $\mu_t$-a.e. $(x,\omega)\in \mathbb{R}^{2d}$, which amounts to saying that $\widetilde{\gamma}_t=((\pi_x,\boldsymbol{u}_t,\pi_\omega)_{\#}\mu_t)\otimes \delta_\omega(\omega')$. Since the above can be repeated for any subsequence of $\{\widetilde{\gamma}_{t,h}\}_{h\in (-t,T-t)}$, then the full sequence is convergent.
\end{proof}

\subsection{Subdifferential calculus and gradient flows}\label{sec:sub_calc_grad_flows}

In this part, we show that the classical theory of gradient flows can be extended to the new fibered space $(\mathcal{P}_{2,\nu}(\mathbb{R}^{2d}), W_{2,\nu})$. Indeed, for special energy functionals, it is consistent with the notion of weak measure-valued solutions to the corresponding evolution PDE. Recall that in \cite[Part I]{AGS-08}, the theory of gradient flows is derived for generic metric spaces supported by purely metric concepts. In particular, one can apply it directly to our Polish space $(\mathcal{P}_{2,\nu}(\mathbb{R}^{2d}),W_{2,\nu})$. However, when the base space has further structure ({\it e.g.} Banach spaces, Riemannian manifolds, etc), convexity allows representing gradient flows in a robust way as shown in \cite[Part II]{AGS-08} for the Wasserstein space $(\mathcal{P}_2(\mathbb{R}^d),W_2)$. Our goal here is to revisit this theory for the formal Riemannian manifold $(\mathcal{P}_{2,\nu}(\mathbb{R}^{2d}),W_{2,\nu})$. For consistency, we review some key parts of the proofs inspired by the treatment in \cite{AGS-08}.

\begin{defi}[Metric slope]\label{D-metric-slope}
Let $(\mathbb{X},d_{\mathbb{X}})$ be any complete metric space, consider any functional $\mathcal{E}:\mathbb{X}\longrightarrow (-\infty,+\infty]$ and $x\in D(\mathcal{E})$. We define the metric (local) slope of $\mathcal{E}$ at $x$ by
$$\vert \partial \mathcal{E}\vert_{d_\mathbb{X}}[x]:=\limsup_{y\rightarrow x}\frac{(\mathcal{E}[x]-\mathcal{E}[y])^+}{d_\mathbb{X}(x,y)}.$$
When the distance $d_\mathbb{X}$ is clear from the context, we shall simply write $\vert \partial \mathcal{E}\vert[x]$.
\end{defi}

Notice that the above definition is purely metric, and in particular, it can be considered for functionals over the fibered quadratic Wasserstein space $(\mathcal{P}_{2,\nu}(\mathbb{R}^{2d}),W_{2,\nu})$ in Definition \ref{D-fibered-Wasserstein}. However, as depicted in last section, $(\mathcal{P}_{2,\nu}(\mathbb{R}^{2d}),W_{2,\nu})$ has further structure, namely, it is formally a Riemannian manifold whose tangent space is given by \eqref{tan}. We emphasize that several notions of subdifferential calculus can be defined according to such a differential structure. For instance, a special definition for functionals over regular probability measures $\mathcal{P}_2^r(\mathbb{R}^d)$ was studied in \cite{AGS-08}. Later, a different notion was presented in \cite{AG-08} for generic functionals over $\mathcal{P}_2(\mathbb{R}^d)$. Recently, an equivalent reformulation was derived in \cite{GTu-19} (see also \cite{CDFLS-11}). Indeed, for functionals over $\mathcal{P}_2^r(\mathbb{R}^d)$ all the three notions agree. In the following Definition \ref{D-Frechet-subdifferential} we have chosen to follow the approach in  \cite{GTu-19}.

\begin{defi}[Fibered Fr\'echet subdifferential]\label{D-Frechet-subdifferential}
Consider $\nu\in \mathcal{P}(\mathbb{R}^d)$ and any functional $\mathcal{E}:\mathcal{P}_{2,\nu}(\mathbb{R}^{2d})\longrightarrow (-\infty,+\infty]$. Set any $\mu\in D(\mathcal{E})$ and any $\boldsymbol{u}\in L^2_\mu(\mathbb{R}^{2d},\mathbb{R}^d)$. We say that $\boldsymbol{u}$ belongs to the (fibered) Fr\'echet subdifferential of $\mathcal{E}$ at $\mu$, and we write $\boldsymbol{u}\in\partial_{W_{2,\nu}}\mathcal{E}[\mu]$, when the following inequality holds
\begin{equation}\label{E-fibered-subdifferential}
\mathcal{E}[\sigma]-\mathcal{E}[\mu]\geq \inf_{\gamma\in \Gamma_{o,\nu}(\mu,\sigma)}\int_{\mathbb{R}^{4d}}\boldsymbol{u}(x,\omega)\cdot(x'-x)\,d\gamma(x,x',\omega,\omega')+o(W_{2,\nu}(\mu,\sigma)),
\end{equation}
for every $\sigma\in D(\mathcal{E})$. 
By ${\partial}^\circ_{W_{2,\nu}}\mathcal{E}[\mu]$ we denote the subset of ${\partial}_{W_{2,\nu}}\mathcal{E}[\mu]$ with minimal $L^2_\mu(\mathbb{R}^{2d},\mathbb{R}^d)$-norm and we refer to it as the minimal (fibered) Fr\'echet subdifferential of $\mathcal{E}$ at $\mu$.
\end{defi}

The term ``fibered'' will be used to distinguish the above notion of subdifferential over the fibered space $(\mathcal{P}_{2,\nu}(\mathbb{R}^{2d}),W_{2,\nu})$ and the classical Fr\'echet subdifferential on $(\mathcal{P}_2(\mathbb{R}^{2d}),W_2)$. When it is clear from the context, we shall rather say ``Fr\'echet subdifferential'' or simply ``subdifferential''. Note that by Proposition \ref{P-tangent-barycentric-projections}, for any vector field $\boldsymbol{u}\in L^2_\mu(\mathbb{R}^{2d},\mathbb{R}^d)$, we have
$$\boldsymbol{u}\in \partial_{W_{2,\nu}}\mathcal{E}[\mu] \quad \Longleftrightarrow \quad \Pi_\mu[\boldsymbol{u}]\in \partial_{W_{2,\nu}}\mathcal{E}[\mu],$$
where $\Pi_\mu:L^2_\mu(\mathbb{R}^{2d},\mathbb{R}^d)\longrightarrow {\rm Tan}_\mu (\mathcal{P}_{2,\nu}(\mathbb{R}^{2d}))$ is the projection operator in Lemma \ref{L-projection-tangent-space}. For this reason, we will often restrict to elements in the fibered Fr\'echet subdifferential which are tangent vectors. In the sequel, we show that the preceding Definition \ref{D-Frechet-subdifferential} can be simplified under certain convexity assumptions on $\mathcal{E}$.

\begin{defi}($\lambda$-convexity along generalized geodesics)\label{D-gen-geodesic-convexity}
Consider $\nu\in \mathcal{P}(\mathbb{R}^d)$ and any functional $\mathcal{E}:\mathcal{P}_{2,\nu}(\mathbb{R}^{2d})\longrightarrow (-\infty,+\infty]$. We say that $\mathcal{E}$ is $\lambda$-convex along generalized geodesics with respect to $W_{2,\nu}$ if for any $\mu_*,\mu_0, \mu_1\in D(\mathcal{E})$ there is a $\nu$-admissible plan $\gamma\in \Gamma_{\nu}(\mu_*,\mu_0,\mu_1)$ with $\pi_{(x_*,x_0,\omega_*,\omega_0)\#}\gamma\in\Gamma_{o,\nu}(\mu_*,\mu_0)$, $\pi_{(x_*,x_1,\omega_*,\omega_1)\#}\gamma\in\Gamma_{o,\nu}(\mu_*,\mu_1)$ such that
\begin{align}\label{e-gen}
\mathcal{E}[\mu_\theta^{0\rightarrow 1}]\leq (1-\theta)\mathcal{E}[\mu_0]+\theta\mathcal{E}[\mu_1]-\frac{\lambda}{2}\theta(1-\theta)W_{\gamma}^2(\mu_0,\mu_1),    
\end{align}
for all $\theta\in [0,1]$. Here, the curve $\theta\in [0,1]\mapsto \mu_\theta^{0\rightarrow1}(x,\omega):=\mu_\theta^{\omega\,0\rightarrow 1}(x)\otimes \nu(\omega)$ represents the associated generalized geodesic from $\mu_0$ to $\mu_1$ with base $\mu_*$ and it is defined fiberwise by
$$\mu_\theta^{\omega\,0\rightarrow 1} := ((1-\theta)\pi_{x_0} + \theta\pi_{x_1})_\#\gamma^\omega,$$
for all $\theta\in [0,1]$ and $\nu$-a.e. $\omega\in \mathbb{R}^d$. In addition, in \eqref{e-gen} we denote:
$$W^2_\gamma(\mu_0,\mu_1):=\int_{\RR^{6d}}|x_0-x_1|^2d\gamma \geq W_{2,\nu}^2(\mu_0,\mu_1).$$
\end{defi}

Definition \ref{D-gen-geodesic-convexity} is the fibered version of the notion of convexity along generalized geodesics for functionals over the classical quadratic Wasserstein space $(\mathcal{P}_2(\mathbb{R}^d),W_2)$, see Definitions 9.2.2 and 9.2.4 in \cite{AGS-08}.

\begin{rem}[$\lambda$-geodesic convexity]\label{R-geodesic-convexity}
If in Definition \ref{D-gen-geodesic-convexity} we restrict to $\mu_*=\mu_1$, then we recover the fibered version of the classical notion of $\lambda$-geodesic convexity in \cite[Definition 9.1.1]{AGS-08}. Namely, we say that $\mathcal{E}$ is $\lambda$-geodesically convex with respect to $W_{2,\nu}$ when for any $\mu_0,\mu_1\in D(\mathcal{E})$ there exists a $\nu$-admissible optimal plan $\gamma\in \Gamma_{o,\nu}(\mu_0,\mu_1)$ such that
\begin{equation}\label{E-geodesic-convexity}
\mathcal{E}[\mu_\theta^{0\rightarrow 1}]\leq (1-\theta)\mathcal{E}[\mu_0]+\theta\mathcal{E}[\mu_1]-\frac{\lambda}{2}\theta(1-\theta)W_{2,\nu}^2(\mu_0,\mu_1),
\end{equation}
for all $\theta\in [0,1]$. Here, $\theta\in [0,1]\mapsto \mu_\theta^{0\rightarrow1}(x,\omega):=\mu_\theta^{\omega\,0\rightarrow 1}(x)\otimes \nu(\omega)$ represents the associated geodesic joining $\mu_0$ to $\mu_1$ and is fiberwise defined by
\begin{equation}\label{E-displacement-interpolation}
\mu_\theta^{\omega\,0\rightarrow 1}:=((1-\theta)\pi_x+\theta\pi_{x'})_{\#}\gamma^\omega,
\end{equation}
for $\theta\in [0,1]$ and $\nu$-a.e. $\omega\in \mathbb{R}^d$.
\end{rem}

When the functional $\mathcal{E}$ is $\lambda$-geodesically convex with respect to $W_{2,\nu}$, the Fr\'echet subdifferential can be simplified like in the classical theory. We emphasize that our notion of $\lambda$-geodesic convexity in Definition \ref{D-gen-geodesic-convexity} and Remark \ref{R-geodesic-convexity} agrees with the one in \cite{AGS-08}, which is slightly weaker than the one in \cite{AG-08} since the inequalities \eqref{e-gen} and \eqref{E-geodesic-convexity} only need to hold for a special plan $\gamma$. However, following the lines of \cite[Proposition 4.2]{AG-08} we recover the following characterization.

\begin{pro}[Variational inequality]\label{P-Frechet subdifferential-convex}
Consider $\nu\in \mathcal{P}(\mathbb{R}^d)$ and a $\lambda$-geodesically convex functional $\mathcal{E}:\mathcal{P}_{2,\nu}(\mathbb{R}^{2d})\longrightarrow (-\infty,+\infty]$ with respect to $W_{2,\nu}$. Set $\mu\in D(\mathcal{E})$ and any vector $\boldsymbol{u}\in L^2_\mu(\mathbb{R}^{2d},\mathbb{R}^d)$. Then, $\boldsymbol{u}$ belongs to $\partial_{W_{2,\nu}}\mathcal{E}[\mu]$ if, and only if,
\begin{equation}\label{E-fibered-subdifferential-convex}
\mathcal{E}[\sigma]-\mathcal{E}[\mu]\geq \int_{\mathbb{R}^{4d}}\boldsymbol{u}(x,\omega)\cdot (x'-x)\,d\gamma(x,x',\omega,\omega')+\frac{\lambda}{2}W_{2,\nu}^2(\mu,\sigma),
\end{equation}
for every $\sigma\in D(\mathcal{E})$ and each $\gamma\in \Gamma_{o,\nu}(\mu,\sigma)$ so that $\mathcal{E}$ is $\lambda$-convex along its associated geodesic $\theta\in [0,1]\mapsto \mu_\theta^{0\rightarrow 1}$ in \eqref{E-displacement-interpolation} joining $\mu$ to $\sigma$.
\end{pro}

\begin{proof}
Since the necessary condition is clear, we only prove the sufficient one. Namely, we show that \eqref{E-fibered-subdifferential} implies \eqref{E-fibered-subdifferential-convex}. Assume that $\boldsymbol{u}\in \partial_{W_{2,\nu}}\mathcal{E}[\mu]$, fix any $\sigma\in \mathcal{P}_{2,\nu}(\mathbb{R}^{2d})$ and denote $\mu_0:=\mu$ and $\mu_1:=\sigma$ for simplicity of notation. Let $\gamma\in\Gamma_{o,\nu}(\mu,\sigma)$ be the $\nu$-admissible optimal plan realizing the convexity condition \eqref{E-geodesic-convexity} in Remark \ref{R-geodesic-convexity} for its associated geodesic $\theta\in [0,1]\mapsto \mu_\theta^{0\rightarrow 1}$ given by \eqref{E-displacement-interpolation} joining $\mu_0$ to $\mu_1$. For any $\theta\in (0,1)$ Lemma 7.2.1 in \cite{AGS-08} guarantees that the only $\gamma_\theta\in \Gamma_{o,\nu}(\mu,\mu_\theta^{0\rightarrow 1})$ is given by $\gamma_\theta^\omega=(\pi_x,(1-\theta)\pi_x+\theta\pi_{x'})_{\#}\gamma^\omega$, for $\nu$-a.e. $\omega\in \mathbb{R}^d$. On the one hand, the subdifferential condition \eqref{E-fibered-subdifferential} with $\mu_\theta^{0\rightarrow 1}$ playing the role of $\sigma$ yields
\begin{align}\label{E-16}
\begin{aligned}
\mathcal{E}[\mu_\theta^{0\rightarrow 1}]-\mathcal{E}[\mu]&\geq \int_{\mathbb{R}^{4d}}\boldsymbol{u}(x,\omega)\cdot (x'-x)\,d\gamma_\theta(x,x',\omega,\omega')+o(W_{2,\nu}(\mu_0,\mu_\theta^{0\rightarrow 1}))\\
&=\theta\int_{\mathbb{R}^{4d}}\boldsymbol{u}(x,\omega)\cdot (x'-x)\,d\gamma(x,x',\omega,\omega')+o(\theta)
\end{aligned}
\end{align}
as $\theta \rightarrow 0$, where in the last line we have used that $W_{2,\nu}(\mu_0,\mu_\theta^{0\rightarrow 1})=\theta W_{2,\nu}(\mu,\sigma)$ and the explicit form of $\gamma_\theta$. On the other hand, the convexity property \eqref{E-geodesic-convexity} implies
\begin{equation}\label{E-17}
\mathcal{E}[\mu_\theta^{0\rightarrow 1}]-\mathcal{E}[\mu]\leq \theta (\mathcal{E}[\sigma]-\mathcal{E}[\mu])-\frac{\lambda}{2}\theta(1-\theta)W_{2,\nu}^2(\mu,\sigma),
\end{equation}
for any $\theta\in (0,1)$. Combining \eqref{E-16} and \eqref{E-17}, dividing by $\theta$ and passing to the limit as $\theta\rightarrow 0$ implies \eqref{E-fibered-subdifferential-convex}.
\end{proof}

Since it will be required throughout this section, we collect here some minimal assumptions on the energy functional $\mathcal{E}$, that are reminiscent of those in \cite{AGS-08} for functionals defined over the classical quadratic Wasserstein space $(\mathcal{P}_2(\mathbb{R}^d),W_2)$.

\medskip

\begin{defi}[Framework $\framework$]\label{D-framework-F}
Consider any $\nu\in \mathcal{P}(\mathbb{R}^d)$ and $\mathcal{E}:\mathcal{P}_{2,\nu}(\mathbb{R}^{2d})\longrightarrow (-\infty,+\infty]$. We shall say that $\mathcal{E}$ satisfies framework $\framework$ if the following assumptions are fulfilled:
\begin{enumerate}
\item[(${\mathcal F}_1$)] $\mathcal{E}$ is a proper functional, {\it i.e.},
$$
D(\mathcal{E})\neq \emptyset.
$$
\item[(${\mathcal F}_2$)] $\mathcal{E}$ is coercive, {\it i.e.}, there exists $\sigma\in {\mathcal P}_{2,\nu}(\mathbb{R}^{2d})$ and $r>0$ such that
$$
\inf\Big\{{\mathcal E}[\mu] : \ \mu\in {\mathcal P}_{2,\nu}(\RR^{2d}),\ W_{2,\nu}(\sigma,\mu)\leq r\Big\}>-\infty
$$
\item[(${\mathcal F}_3$)] $\mathcal{E}$ is lower semicontinuous, {\it i.e.},
$$
\lim_{n\rightarrow \infty}W_{2,\nu}(\mu_n,\mu)=0\quad \Longrightarrow\quad \mathcal{E}[\mu]\leq \liminf_{n\rightarrow\infty}\mathcal{E}[\mu_n].
$$
\item[(${\mathcal F}_4$)] $\mathcal{E}$ is $\lambda$-convex along generalized geodesics ({\it cf.} Definition \ref{D-gen-geodesic-convexity}).
\end{enumerate}
\end{defi}

A fundamental observation is that the convexity property along generalized geodesic in the above item $({\mathcal F}_4)$ implies that the penalized energy functional below is also convex. More specifically, we have the following result.

\begin{lem}[Convexity of the penalized energy]\label{L-convexity-penalized-energy}
Consider $\nu\in \mathcal{P}(\mathbb{R}^d)$ and any functional $\mathcal{E}:\mathcal{P}_{2,\nu}(\mathbb{R}^{2d})\longrightarrow (-\infty,+\infty]$ that is $\lambda$-convex along generalized geodesics. Define the associated penalized energy functional as follows
\begin{equation}\label{E-penalized-energy}
\Phi(\tau,\mu_*;\mu):=\mathcal{E}[\mu]+\frac{1}{2\tau}W_{2,\nu}^2(\mu,\mu_*),\quad \mu,\,\mu_*\in \mathcal{P}_{2,\nu}(\mathbb{R}^{2d}),\quad\tau>0.
\end{equation}
Then, for any $\mu_*,\mu_0,\mu_1\in D(\mathcal{E})$ and any $0<\tau<\frac{1}{\lambda^-}$ there is a curve $\theta\in [0,1]\mapsto \mu_\theta\in \mathcal{P}_{2,\nu}(\mathbb{R}^{2d})$ interpolating between $\mu_0$ and $\mu_1$ such that the following convexity property holds
\begin{equation}\label{E-convexity-penalized-energy}
\Phi(\tau,\mu_*;\mu_\theta)\leq (1-\theta) \Phi(\tau,\mu_*;\mu_0) + \theta\Phi(\tau,\mu_*;\mu_1) - \frac{1+\lambda\tau}{2\tau}\theta(1-\theta) W^2_{2,\nu}(\mu_0,\mu_1),
\end{equation}
for any $\theta\in [0,1]$.
\end{lem}

\begin{proof} 
Let us set $\mu_*,\mu_0,\mu_1$ and $\tau$ as in the statement and let $\theta\in [0,1]\mapsto \mu_\theta^{0\rightarrow 1}\in \mathcal{P}_{2,\nu}(\mathbb{R}^{2d})$ be the generalized geodesic joining $\mu_0$ to $\mu_1$ with base $\mu_*$ ({\it cf.} Definition \ref{D-gen-geodesic-convexity}) so that $\mathcal{E}$ satisfies the convexity property \eqref{e-gen}. We divide the reminder of the proof into two steps.

\medskip

$\diamond$ {\sc Step 1}: $1$-convexity of $\frac{1}{2}W_{2,\nu}^2(\mu_*,\cdot)$ along $\mu_\theta^{0\rightarrow1}$.\\
By Definition \ref{D-gen-geodesic-convexity} we know that $\mu_\theta^{0\rightarrow 1}(x,\omega)=\mu_\theta^{\omega 0\rightarrow 1}(x)\otimes \nu(\omega)$ and $\theta\in [0,1]\mapsto \mu_\theta^{\omega 0\rightarrow 1}$ is a generalized geodesic joining $\mu_0^\omega$ to $\mu_1^\omega$ with base $\mu_*^\omega$ for $\nu$-a.e. $\omega\in \mathbb{R}^d$. In addition, by \cite[Lemma 9.2.1]{AGS-08}, the squared classical Wasserstein distance is $1$-convex along generalized geodesics so that we have the inequality
\begin{align*}
    \frac{1}{2}W_2^2(\mu_*^\omega,\mu_\theta^{\omega 0\to 1})\leq (1-\theta)\frac{1}{2}W_2^2(\mu_*^\omega,\mu_{0}^{\omega}) + \theta \frac{1}{2}W_2^2(\mu_*^\omega,\mu_{1}^{\omega}) - \frac{1}{2}\theta(1-\theta)W_{\gamma^\omega}^2(\mu_{0}^{\omega},\mu_{1}^{\omega}),
\end{align*}
for $\nu$-a.e. $\omega\in \mathbb{R}^d$ and each $\theta\in [0,1]$. Here, $\gamma\in \Gamma_\nu(\mu_*,\mu_0,\mu_1)$ is given in Definition \ref{D-gen-geodesic-convexity} and
$$W_{\gamma^\omega}^2(\mu_{0}^{\omega},\mu_{1}^{\omega}) := \int_{\RR^{3d}}|x_0-x_1|^2d\gamma^\omega\geq W_2(\mu_0^\omega,\mu_1^\omega),$$
for $\nu$-a.e. $\omega\in \mathbb{R}^{2d}$. Then integrating with respect to $d\nu(\omega)$ (note that all of the involved integrands are Borel measurable in the sense of Proposition \ref{P-Borel-family-vs-measurability}) we obtain
\begin{equation}\label{E-convexity-W2}
\frac{1}{2}W_{2,\nu}^2(\mu_*,\mu_\theta^{0\to 1})\leq (1-\theta)\frac{1}{2}W_{2,\nu}^2(\mu_*,\mu_{0}) + \theta \frac{1}{2}W_{2,\nu}^2(\mu_*,\mu_{1}) - \frac{1}{2}\theta(1-\theta)W_{\gamma}^2(\mu_{0},\mu_{1}),
\end{equation}
for every $\theta\in [0,1]$.

\medskip

$\diamond$ {\sc Step 2}: Proof of \eqref{E-convexity-penalized-energy}.\\
Since $\mathcal{E}$ is $\lambda$-convex along $\mu_\theta^{0\rightarrow1}$, then adding \eqref{e-gen} and \eqref{E-convexity-W2} yields
$$
\Phi(\tau,\mu_*;\mu_\theta^{0\rightarrow1})\leq (1-\theta) \Phi(\tau,\mu_*;\mu_0) + \theta\Phi(\tau,\mu_*;\mu_1) - \frac{1+\lambda\tau}{2\tau}\theta(1-\theta) W^2_{\gamma}(\mu_0,\mu_1),
$$
for every $\theta\in [0,1]$. Since $W^2_\gamma(\mu_0,\mu_1)\geq W_{2,\nu}(\mu_0,\mu_1)$ we conclude \eqref{E-convexity-penalized-energy} provided that $1+\lambda\tau>0$, which is of course true as long as $\tau\in(0,\frac{1}{\lambda^-})$.
\end{proof}

Similarly to the classical subdifferential calculus in Hilbert spaces, the metric slope and the subdifferential with respect to $W_{2,\nu}$ of functionals over $(\mathcal{P}_{2,\nu}(\mathbb{R}^{2d}),W_{2,\nu})$ are related as follows (the proof can be found in Appendix \ref{Appendix-slope-subdifferential}).

\begin{pro}[Metric slope vs minimal subdifferential]\label{P-slope-minimal-subdifferential}
Consider $\nu\in \mathcal{P}(\mathbb{R}^d)$ and any functional $\mathcal{E}$ satisfying framework $\framework$. If $\mu \in D(\vert \partial \mathcal{E}\vert_{W_{2,\nu}})$, then $\partial_{W_{2,\nu}} \mathcal{E}[\mu]\neq \emptyset$ and
\begin{equation}\label{E-slope-minimal-subdifferential}
\vert \partial\mathcal{E}\vert_{W_{2,\nu}}[\mu]=\min\left\{\Vert \boldsymbol{u}\Vert_{L^2_\mu(\mathbb{R}^{2d},\mathbb{R}^d)}:\,\boldsymbol{u}\in \partial_{W_{2,\nu}}\mathcal{E}[\mu]\right\}.
\end{equation}
\end{pro}

In the sequel, we recall several notions of gradient flows associated with an energy functional $\mathcal{E}:\mathcal{P}_{2,\nu}(\mathbb{R}^{2d})\longrightarrow (-\infty,+\infty]$. We refer to Definitions 1.3.2 and 11.1.1 along with Theorem 11.1.4 in \cite{AGS-08} for gradient flows over the classical Wassertein space.

\begin{defi}\label{D-solutions-1st-order-gradient-flow}
Consider $\nu\in \mathcal{P}(\mathbb{R}^d)$, any functional $\mathcal{E}:\mathcal{P}_{2,\nu}(\mathbb{R}^{2d})\longrightarrow (-\infty,+\infty]$ and $\boldsymbol{\mu}$ in $AC^2(0,T;\mathcal{P}_{2,\nu}(\mathbb{R}^{2d}))$. Then we have the following definitions for $\boldsymbol{\mu}$:
\begin{enumerate}[label=(\roman*)]
\item {\it (Curve of maximal slope)} There exists a non-increasing function $E$ so that $E(t)=\mathcal{E}[\mu_t]$ for a.e. $t\in [0,T]$ and we have
\begin{equation}\label{E-D-maximal-slope-1st-order}
\frac{d}{dt}E(t)\leq -\frac{1}{2}\vert \boldsymbol{\mu}'\vert_{W_{2,\nu}}^2(t)-\frac{1}{2}\vert \partial \mathcal{E}\vert_{W_{2,\nu}}^2[\mu_t],
\end{equation}
for a.e. $t\in [0,T]$.
\item {\it (Gradient flow)} The tangent vector $\boldsymbol{u}_t\in {\rm Tan}_{\mu_t}(\mathcal{P}_{2,\nu}(\mathbb{R}^{2d}))$ verifies
\begin{equation}\label{E-D-gradient-flow-sol-1st-order}
\boldsymbol{u}_t\in -\partial_{W_{2,\nu}}\mathcal{E}[\mu_t],\\
\end{equation}
for a.e. $t\in [0,T]$.
\item {\it (Solution to the E.V.I with $\lambda\in \RR$)} For each $\sigma\in D(\mathcal{E})$ we have
\begin{equation}\label{E-D-EVI-1st-order}
\frac{1}{2}\frac{d}{dt}W_{2,\nu}^2(\mu_t,\sigma)+\frac{\lambda}{2}W_{2,\nu}^2(\mu_t,\sigma)\leq \mathcal{E}[\sigma]-\mathcal{E}[\mu_t],
\end{equation}
for a.e. $t\in [0,T]$.
\end{enumerate}
\end{defi}

We emphasize that the definitions of curve of maximal slope \eqref{E-D-maximal-slope-1st-order} and E.V.I. \eqref{E-D-EVI-1st-order} are purely metric and can be considered over general Polish spaces. However, the definition \eqref{E-D-gradient-flow-sol-1st-order} of gradient flow requires a richer structure of the base space that allows computing Fr\'echet subdifferentials. The typical setting is a Hilbert space, but it has been extended to $\mathcal{P}_2(\mathbb{R}^d)$ based on the underlying formal Riemannian structure. For the same reason, we can extend to the fibered space $\mathcal{P}_{2,\nu}(\mathbb{R}^{2d})$. As for the classical quadratic Wasserstein space $\mathcal{P}_2(\mathbb{R}^d)$, we shall show that all the above definitions are equivalent in $\mathcal{P}_{2,\nu}(\mathbb{R}^{2d})$, provided that the energy functional $\mathcal{E}$ satisfies framework $\framework$. 

\begin{theo}\label{T-equivalence-definitions-solutions-1st-order}
Consider $\nu\in \mathcal{P}(\mathbb{R}^d)$ and any functional $\mathcal{E}$ satisfying framework $\framework$. Set $\boldsymbol{\mu}$ in $AC^2(0,T;\mathcal{P}_{2,\nu}(\mathbb{R}^{2d}))$. Then, \eqref{E-D-maximal-slope-1st-order}, \eqref{E-D-gradient-flow-sol-1st-order} and \eqref{E-D-EVI-1st-order} are mutually equivalent. In that case, the tangent vector $\boldsymbol{u}_t\in {\rm Tan}_{\mu_t}(\mathcal{P}_{2,\nu}(\mathbb{R}^{2d}))$ to $\boldsymbol{\mu}$ satisfies the minimal selection principle
\begin{equation}\label{E-D-gradient-flow-sol-1st-order-slow}
\boldsymbol{u}_t\in -\partial^\circ_{W_{2,\nu}}\mathcal{E}[\mu_t],
\end{equation}
for a.e. $t\in [0,T]$. Moreover, $\boldsymbol{\mu}\in \Lip(0,T;\mathcal{P}_{2,\nu}(\mathbb{R}^{2d}))$ and $t\mapsto \mathcal{E}[\mu_t]$ is locally Lipschitz.
\end{theo}

\begin{proof}
First, we observe that if $\boldsymbol{\mu}$ verifies E.V.I. \eqref{E-D-EVI-1st-order}, then $\boldsymbol{\mu}$ is automatically a curve of maximal slope in the sense \eqref{E-D-maximal-slope-1st-order}. In fact, such a result is well-known for lower semicontinuous functionals in a much more general setting, where $({\mathcal P}_{2,\nu}(\RR^{2d}),W_{2,\nu})$ is replaced by any Polish space $(\mathbb{X},d_{\mathbb{X}})$, see \cite[Proposition 4.6]{AG-13} for instance. Thus, we shall focus on the remaining implications, that require the abstract Riemannian structure of $\mathcal{P}_{2,\nu}(\mathbb{R}^{2d})$. That is, we will prove that \eqref{E-D-maximal-slope-1st-order} implies \eqref{E-D-gradient-flow-sol-1st-order} and that \eqref{E-D-gradient-flow-sol-1st-order} implies \eqref{E-D-EVI-1st-order}. 

\medskip

$\diamond$ {\sc Step 0}: Preparatory result.\\
Fix any $t\in (0,T)$ and $\widetilde{\boldsymbol{u}}_t\in {\rm Tan}_{\mu_t}(\mathcal{P}_{2,\nu}(\mathbb{R}^{2d}))$, and define the map $\varphi:\mathbb{R}^{4d}\longrightarrow\mathbb{R}$ with 
$$\varphi(x,x',\omega,\omega'):=\widetilde{\boldsymbol{u}}_t(x,\omega)\cdot x'.$$
Fix $\gamma_{t,h}\in \Gamma_{o,\nu}(\mu_t,\mu_{t+h})$ for any $h\in (0,T-t)$, and define $\widetilde{\gamma}_{t,h}:=(\pi_x,\frac{1}{h}(\pi_{x'}-\pi_x),\pi_\omega,\pi_{\omega'})_\#\gamma_{t,h}$ as in Proposition \ref{P-infinitesimal-vector-from-plans}. Then, $\varphi$ is uniformly integrable with respect to $\{\widetilde{\gamma}_{t,h}\}_{h\in (0,T-t)}$. To prove such a claim, note that we can split the integral as follows
\begin{align*}
\int_{\{\vert\varphi\vert\geq R\}}\vert \varphi\vert\,d\widetilde{\gamma}_{t,h}&\leq \int_{\{\vert \widetilde{\boldsymbol{u}}_t\vert\geq S\}}\vert \varphi\vert\,d\widetilde{\gamma}_{t,h}+\int_{\{\vert \widetilde{\boldsymbol{u}}_t\vert<S,\,\vert x'\vert\geq R/S\}}\vert \varphi\vert\,d\widetilde{\gamma}_{t,h}\\
&\leq M_{t,h}^{1/2}\left(\int_{\{|\widetilde{\boldsymbol{u}}_t|\geq S\}}|\widetilde{\boldsymbol{u}}_t|^2\,d\mu_t\right)^{1/2}+\frac{S^2 M_{t,h}}{R},
\end{align*}
for any $R,S>0$, where we denote
$$M_{t,h}:=\int_{\mathbb{R}^{4d}}\vert x'\vert^2\,d\widetilde{\gamma}_{t,h}=\frac{1}{h^2}\int_{\mathbb{R}^{4d}}\vert x-x'\vert^2\,d\gamma_{t,h}=\frac{W_{2,\nu}^2(\mu_t,\mu_{t+h})}{h^2}.$$
Note that $M_t:=\sup_{h\in (0,T-t)}M_{t,h}<\infty$ for {\it a.e.} $t\in (0,T)$ because $\boldsymbol{\mu}\in AC(0,T;\mathcal{P}_{2,\nu}(\mathbb{R}^{2d}))$ and therefore $M_{t,h} $ converges to the metric derivative $\vert \boldsymbol{\mu}'\vert_{W_{2,\nu}}^2(t)$ by Proposition \ref{md}. Then,
$$
\sup_{h\in (-t,T-t)}\int_{\{\vert\varphi\vert\geq R\}}\vert \varphi\vert\,d\widetilde{\gamma}_{t,h}\leq M_t^{1/2}\left(\int_{\{|\widetilde{\boldsymbol{u}}_t|\geq S\}}|\widetilde{\boldsymbol{u}}_t|^2\,d\mu_t\right)^{1/2}+\frac{S^2 M_t}{R}.
$$
Taking limits as $R\rightarrow \infty$ we have
$$\limsup_{R\rightarrow\infty} \sup_{h\in (-t,T-t)}\int_{\{\vert\varphi\vert\geq R\}}\vert \varphi\vert\,d\widetilde{\gamma}_{t,h}\leq M_t^{1/2}\left(\int_{\{|\widetilde{\boldsymbol{u}}_t|\geq S\}}|\widetilde{\boldsymbol{u}}_t|^2\,d\mu_t\right)^{1/2},$$
for any $S>0$. Finally, taking limits $S\rightarrow \infty$ and using that $\widetilde{\boldsymbol{u}}_t\in L^2_{\mu_t}(\mathbb{R}^{2d},\mathbb{R}^d)$ we conclude
$$\lim_{R\rightarrow\infty} \sup_{h\in (-t,T-t)}\int_{\{\vert\varphi\vert\geq R\}}\vert \varphi\vert\,d\widetilde{\gamma}_{t,h}=0.$$

\medskip

$\diamond$ {\sc Step 1}: Proof that \eqref{E-D-maximal-slope-1st-order} implies \eqref{E-D-gradient-flow-sol-1st-order}.\\
Assume that $\boldsymbol{\mu}$ is a curve of maximal slope of $\mathcal{E}$ in the sense of \eqref{E-D-maximal-slope-1st-order}. First, let us show that $t\mapsto \mathcal{E}[\mu_t]$ is indeed locally absolutely continuous and we can take $E(t)=\mathcal{E}[\mu_t]$ for all $t\in [0,T]$ in \eqref{E-D-maximal-slope-1st-order}. On the one hand, by integrating \eqref{E-D-maximal-slope-1st-order} and using Young's inequality along with Lebesgue's theorem for the integral of non-decreasing functions we obtain 
\begin{equation}\label{E-int-slope-speed}
\int_s^t \vert \boldsymbol{\mu}'\vert_{W_{2,\nu}}(\tau)\vert \partial \mathcal{E}[\mu_\tau]\vert\,d\tau\leq \frac{1}{2}\int_s^t\vert \boldsymbol{\mu}'\vert_{W_{2,\nu}}(\tau)+\frac{1}{2}\int_s^t\vert \partial \mathcal{E}\vert[\mu_\tau]^2\,d\tau\leq E(s)-E(t),
\end{equation}
for each $0<s<t<T$.
On the other hand, since $\mathcal{E}$ is lower semicontinuous and $\lambda$-geodesically convex we can apply \cite[Proposition 4.19]{AG-13} and obtain the inequality
\begin{equation}\label{E-strog-upper-gradient}
\vert \mathcal{E}[\mu_t]-\mathcal{E}[\mu_s]\vert\leq \int_s^t\vert \boldsymbol{\mu}'\vert_{W_{2,\nu}}(\tau)\vert \partial \mathcal{E}\vert_{W_{2,\nu}}[\mu_\tau]\,d\tau,
\end{equation}
for each $0<s\leq t<T$ ({\it i.e.}, $\vert \partial \mathcal{E}\vert_{W_{2,\nu}}$ is a strong upper gradient). By \eqref{E-int-slope-speed} the integrand in \eqref{E-strog-upper-gradient} belongs to $L^1_{\text{loc}}(0,T)$ and this ends the claim. Second, we prove that $\boldsymbol{\mu}$ verifies \eqref{E-D-gradient-flow-sol-1st-order}. Using the definitions of metric derivatives and slope with respect to $W_{2,\nu}$ it is clear that
\begin{align*}
    \frac{d}{dt}E(t)=\lim_{s\downarrow t}\frac{\mathcal{E}[\mu_s]-\mathcal{E}[\mu_t]}{s-t} &=\lim_{s\downarrow t}\frac{\mathcal{E}[\mu_s]-\mathcal{E}[\mu_t]}{W_{2,\nu}(\mu_s,\mu_t)}\frac{W_{2,\nu}(\mu_s,\mu_t)}{s-t}\\
    &\geq -\vert \partial\mathcal{E}\vert_{W_{2,\nu}}\vert\boldsymbol{\mu}'\vert_{W_{2,\nu}}\geq -\frac{1}{2}\vert \partial\mathcal{E}\vert_{W_{2,\nu}}^2 - \frac{1}{2}\vert\boldsymbol{\mu}'\vert_{W_{2,\nu}}^2,
\end{align*}
at a.e. $t\in [0,T]$ at which $E$ is strictly decreasing (for $\frac{d}{dt}E(t)=0$ the result is obvious). Then, we achieve the equality in \eqref{E-D-maximal-slope-1st-order}, {\it i.e.},
\begin{equation}\label{E-D-maximal-slope-1st-order-equality}
\frac{1}{2}\vert \boldsymbol{\mu}'\vert_{W_{2,\nu}}^2(t)+\frac{1}{2}\vert \partial \mathcal{E}\vert_{W_{2,\nu}}^2[\mu_t]+\frac{d}{dt}E(t)=0,
\end{equation}
for a.e. $t\geq 0$. On the other hand, take $\widetilde{\boldsymbol{u}}_t\in \partial^\circ_{W_{2,\nu}}\mathcal{E}[\mu_t]$, that is non-empty by Proposition \ref{P-slope-minimal-subdifferential}. Using the definition of the fibered subdifferential we obtain
\begin{multline*}
\mathcal{E}[\mu_{t+h}]-\mathcal{E}[\mu_t]\geq \inf_{\gamma_{t,h}\in \Gamma_{o,\nu}(\mu_t,\mu_{t+h})}\int_{\mathbb{R}^{4d}}\widetilde{\boldsymbol{u}}_t(x,\omega)\cdot (x'-x)\,d\gamma_{t,h}(x,x',\omega,\omega')+o(W_{2,\nu}(\mu_{t+h},\mu_t)),
\end{multline*}
as $h\rightarrow 0$ for every $t\in (0,T)$. 
Divide the above by $h$, approximate the infimum modulo an arbitrarily small error $\varepsilon>0$ by some optimal plan $\gamma_{t,h}^\varepsilon\in \Gamma_{o,\nu}(\mu_t,\mu_{t+h})$, and define the modified plans $\widetilde{\gamma}_{t,h}^\varepsilon:=(\pi_x,\frac{1}{h}(\pi_{x'}-\pi_x),\pi_\omega,\pi_{\omega'})_{\#}\gamma_{t,h}^\varepsilon$ as above to obtain
\begin{equation}\label{E-5}
\frac{\mathcal{E}[\mu_{t+h}]-\mathcal{E}[\mu_t]}{h}\\
\geq \int_{\mathbb{R}^{4d}}\widetilde{\boldsymbol{u}}_t(x,\omega)\cdot x'\,d\widetilde{\gamma}_{t,h}^\varepsilon(x,x',\omega,\omega')-\varepsilon+o(1),
\end{equation}
as $h\rightarrow 0$. Since $\widetilde{\gamma}^\varepsilon_{t,h}\rightarrow \widetilde{\gamma}_t$ narrowly as $h\rightarrow 0$ with $\widetilde{\gamma}_t:=((\pi_x,\boldsymbol{u}_t,\pi_\omega)_{\#}\mu_t)\otimes \delta_\omega(\omega')$ by Proposition \ref{P-infinitesimal-vector-from-plans}, we have
\begin{equation}\label{E-5bis}
\lim_{h\rightarrow 0}\int_{\mathbb{R}^{4d}}\widetilde{\boldsymbol{u}}_t(x,\omega)\cdot x'\,d\widetilde{\gamma}_{t,h}^\varepsilon(x,x',\omega,\omega')= \int_{\mathbb{R}^{2d}}\widetilde{\boldsymbol{u}}_t(x,\omega)\cdot \boldsymbol{u}_t(x,\omega)\,d\mu_t(x,\omega),
\end{equation}
where $\boldsymbol{u}_t\in {\rm Tan}_{\mu_t}({\mathcal P}_{2,\nu}(\RR^{2d}))$ is the tangent vector in Proposition \ref{equiv}. Above we have used \cite[Lemma 5.1.7]{AGS-08} along with the uniform integrability property in {\sc Step 0}. Therefore, passing with $h\rightarrow 0$ in \eqref{E-5} and using \eqref{E-5bis} and the arbitrariness of $\varepsilon>0$ leads to
\begin{equation}\label{E-derivative-E}
\frac{d}{dt}E(t)\geq \int_{\mathbb{R}^{2d}}\widetilde{\boldsymbol{u}}_t(x,\omega)\cdot \boldsymbol{u}_t(x,\omega)\,d\mu_t(x,\omega),
\end{equation}
for a.e. $t\in [0,T]$. Putting \eqref{E-D-maximal-slope-1st-order-equality} and \eqref{E-derivative-E} together, we obtain that
\begin{multline*}
\frac{1}{2}\int_{\mathbb{R}^{2d}}\vert \boldsymbol{u}_t(x,\omega)+\widetilde{\boldsymbol{u}}_t(x,\omega)\vert^2\,d\mu_t(x,\omega)\\
=\frac{1}{2}\Vert \boldsymbol{u}_t\Vert_{L^2_{\mu_t}(\mathbb{R}^{2d},\mathbb{R}^d)}^2+\frac{1}{2}\Vert\widetilde{\boldsymbol{u}}_t\Vert_{L^2_{\mu_t}(\mathbb{R}^{2d},\mathbb{R}^d)}^2+\int_{\mathbb{R}^{2d}}\widetilde{\boldsymbol{u}}_t(x,\omega)\cdot \boldsymbol{u}_t(x,\omega)\,d\mu_t(x,\omega)\\
\leq\frac{1}{2}\vert \boldsymbol{\mu}'\vert^2_{W_{2,\nu}}(t)+\frac{1}{2}\vert \partial \mathcal{E}\vert^2_{W_{2,\nu}}[\mu_t]+\frac{d}{dt}E(t)=0,
\end{multline*}
for a.e. $t\in [0,T]$, were we have used \eqref{E-slope-minimal-subdifferential} and Proposition \ref{equiv} in the last line. Consequently, $\boldsymbol{u}_t=-\widetilde{\boldsymbol{u}}_t\in -\partial^\circ_{W_{2,\nu}}\mathcal{E}[\mu_t]$ for a.e. $t\geq 0$. Thus {\sc Step 1} of the proof is finished and $\boldsymbol{\mu}$ satisfies the minimal selection principle.

\medskip

$\diamond$ {\sc Step 2}: Proof that \eqref{E-D-gradient-flow-sol-1st-order} implies \eqref{E-D-EVI-1st-order}.\\
Assume that $\boldsymbol{\mu}$ is a gradient flow solution in the sense of \eqref{E-D-gradient-flow-sol-1st-order}. On the one hand, since $\boldsymbol{\mu}$ is a fibered gradient flow of $\mathcal{E}$ then, by Proposition \ref{P-Frechet subdifferential-convex}, the tangent vector field $\boldsymbol{u}_t\in {\rm Tan}_{\mu_t}(\mathcal{P}_{2,\nu}(\mathbb{R}^{2d}))$ associated with the curve satisfies the following inequality
\begin{equation}\label{E-8}
\mathcal{E}[\sigma]-\mathcal{E}[\mu_t]\geq -\int_{\mathbb{R}^{4d}}\boldsymbol{u}(x,\omega)\cdot (x'-x)\,d\gamma_t(x,x',\omega,\omega')+\frac{\lambda}{2}W_{2,\nu}^2(\mu_t,\sigma),
\end{equation}
for a.e. $t\in [0,T]$ and any $\gamma_t\in \Gamma_{o,\nu}(\mu_t,\sigma)$ so that $\mathcal{E}$ is $\lambda$-convex along the associated geodesic joining $\mu_t$ to $\sigma$. By Proposition \ref{T-dif}, we can identify the first term in the right hand side of \eqref{E-8} as the time derivative of $\frac{1}{2}W_{2,\nu}^2(\mu_t,\sigma)$. Then, the E.V.I. \eqref{E-D-EVI-1st-order} follows.

\medskip

$\diamond$ {\sc Step 3}: Regularizing effect.\\
For $h>0$ small enough, set the curves $\mu^1_t=\mu_t$ and $\mu^2_t=\mu_{t+h}$. Since both $\boldsymbol{\mu}^1$ and $\boldsymbol{\mu}^2$ solve the E.V.I. \eqref{E-D-EVI-1st-order}, then using the chain rule and Gr\"{o}nwall's lemma yields the stability estimate
$$W_{2,\nu}(\mu_{t+h},\mu_t)\leq e^{-\lambda(t-s)}W_{2,\nu}(\mu_{s+h},\mu_s),$$
for any $0\leq s\leq t\leq T-h$. Dividing by $h$ and taking limits as $h\rightarrow 0$ at any couple of points $0<t\leq s<T$ where the metric derivative exists, we obtain
$$\vert \boldsymbol{\mu}'\vert_{W_{2,\nu}}(t)\leq e^{-\lambda(t-s)}\vert \boldsymbol{\mu}'\vert_{W_{2,\nu}}(s).$$
Then, there exists $L>0$ with $\vert \boldsymbol{\mu}'\vert_{W_{2,\nu}}(t)\leq L$ for a.e. $t\in (0,T)$. Thus, we obtain $\boldsymbol{\mu}\in \Lip(0,T;\mathcal{P}_{2,\nu}(\mathbb{R}^{2d}))$. In particular, for each $\sigma \in D(\mathcal{E})$
$$\left\vert\frac{d}{dt}\frac{1}{2}W_{2,\nu}^2(\mu_t,\sigma)\right\vert\leq \Vert \boldsymbol{u}_t\Vert_{L^2(\mu_t)}W_{2,\nu}(\mu_t,\sigma)\leq L W_{2,\nu}(\mu_t,\sigma),$$
for a.e. $t\in (0,T)$. Plugging it in the E.V.I. \eqref{E-D-EVI-1st-order} we obtain
$$-LW_{2,\nu}(\mu_t,\sigma)+\frac{\lambda}{2}W_{2,\nu}^2(\mu_t,\sigma)\leq \mathcal{E}[\sigma]-\mathcal{E}[\mu_t],$$
for a.e. $t\in (0,T)$. By lower semicontinuity of $t\mapsto \mathcal{E}[\mu_t]$, the above condition actually holds for each $t\in (0,T)$. Thus, we can take $\sigma=\mu_s$ for any $s\in (0,T)$ and we conclude that
$$\vert \mathcal{E}[\mu_t]-\mathcal{E}[\mu_s]\vert\leq L \vert t-s\vert+\frac{\lambda_-}{2}L^2(t-s)^2,$$
for any $t,s\in (0,T)$. Therefore, $t\mapsto \mathcal{E}[\mu_t]$ is locally Lipschitz.
\end{proof}

\medskip

We end this section by stating Theorem A on the existence of fibered gradient flows in the sense of Definition \ref{D-solutions-1st-order-gradient-flow} provided that $\mathcal{E}$ satisfies framework $\framework$. Our goal is to use the classical theory of gradient flows on metric spaces. Particularly, we aim to apply \cite[Theorem 4.0.4]{AGS-08}. The rigorous statement of Theorem A then reads as follows.

\begin{theo}[Theorem A: existence of fibered gradient flows]\label{t-A-rig}
Let $\mathcal{E}:{\mathcal P}_{2,\nu}(\RR^{2d})\rightarrow(-\infty,+\infty]$ satisfy framework $\framework$. Then for all $\mu_0\in D({\mathcal E})$ the following assertions hold:
\begin{enumerate}[label=(\roman*)]
    \item ({\it Existence and uniqueness}) There exists a unique gradient flow $\boldsymbol{\mu}\in AC^2_{{\rm loc}}(0,\infty;\mathcal{P}_{2,\nu}(\mathbb{R}^{2d}))$ of ${\mathcal E}$ in the sense of Definition \ref{D-solutions-1st-order-gradient-flow} with $\lim_{t\to 0^+}W_{2,\nu}(\mu_t,\mu_0)=0$.
    \item ({\it Regularity}) We have $\boldsymbol{\mu}\in \Lip_{{\rm loc}}(0,\infty;\mathcal{P}_{2,\nu}(\mathbb{R}^{2d}))$ with $\mu_t\in D(|\partial \mathcal{E}|)\subset D(\mathcal{E})$ for every $t>0$ and, in addition,
\begin{align*}
\begin{aligned}
    \mathcal{E}[\mu_t] &\leq \mathcal{E}[\mu_0]\leq \mathcal{E}[\sigma] + \frac{1}{2t}W^2_{2,\nu}(\sigma,\mu_0) & &\forall\, \sigma\in D(\mathcal{E}),\\
    |\partial\mathcal{E}|_{W_{2,\nu}}^2[\mu_t] &\leq |\partial\mathcal{E}|_{W_{2,\nu}}^2[\sigma] + \frac{1}{t^2}W^2_{2,\nu}(\sigma,\mu_0) & &\forall\, \sigma\in D(|\partial\mathcal{E}|).
\end{aligned}
\end{align*}

\item {\it (Stability)} Let $\boldsymbol{\mu}^1$ and $\boldsymbol{\mu}^2$ be two gradient flows with $\mu_0^1,\,\mu_0^2\in D(\mathcal{E})$. Then, we have
\begin{align*}
W_{2,\nu}(\mu_t^1,\mu_t^2)\leq e^{-\lambda t}W_{2,\nu}(\mu^1_0,\mu^2_0),
\end{align*}
for all $t>0$.
\end{enumerate}
\end{theo}

The proof follows from the standard theory of gradient flows for metric-compatible $\lambda$-convex functionals. It suffices to apply \cite[Theorem 4.0.4]{AGS-08} with the ambient Polish space $({\mathcal P}_{2,\nu}(\mathbb{R}^{2d}), W_{2,\nu})$, the energy functional $\mathcal{E}$, and the initial datum $\mu_0$. All of the assumptions in \cite[Theorem 4.0.4]{AGS-08} follow directly from the hypothesis in framework $\framework$, except for the convexity property of the associated penalized energy functional, which holds by Lemma \ref{L-convexity-penalized-energy} thanks to the assumed $\lambda$-convexity of $\mathcal{E}$ along generalized geodesic.

\begin{rem}\label{R-why-lambda-convexity}
As mentioned in Section \ref{maingoal2}, the $\lambda$-convexity is actually not necessary to prove the existence of gradient flows. Indeed, independently on any convexity assumption, the lower-semicontinuity of the energy functional and its metric slope, along with the condition that the latter be an upper gradient may still hold for smoother functionals. As discussed in \cite{S-17}, these three conditions provide all the necessary control in order to pass to the limit in the time-discrete JKO scheme and find a time-continuous curve of maximal slope. Of course, the last two are immediately guaranteed under $\lambda$-convexity, which is essentially the approach followed in \cite{AGS-08}. However, such a restriction is far from immaterial, as, for example, the energy functional of the Patlak-Keller-Segel equation is typically not $\lambda$-convex and yet a weak solution can be found by taking limits in the JKO scheme, see \cite{BCC-08}. In cases like those, the gradient flow formalism may be too heavy and it could be more instructive to forget about it and simply identify a PDE for the limiting curve of the JKO scheme. In this paper, we work within framework $\framework$ for the sake of simplicity in order to keep our setting consistent for the existence of gradient flows in Theorem \ref{t-A-rig}, the uniqueness and the equivalence of all the notions in Theorem \ref{T-equivalence-definitions-solutions-1st-order}.
\end{rem}

\section{Examples of gradient flows in $(\mathcal{P}_{2,\nu}(\mathbb{R}^{2d}),W_{2,\nu})$}\label{sec:exist}

Our goal in this section is to analyse three classical examples of energy functionals namely, {\it internal, external} and {\it interaction} energy functionals, with particular emphasis on applications to the examples presented in Section \ref{maingoal2}. First, we provide the heuristic interpretation of fibered gradient flows as solutions to the PDE \eqref{kurak-variation}, where the velocity field is determined through the Euler first variation of the energy functional. This requires computing formally the Fr\'echet fibered subdifferential of a number of energy functionals of integral form. Second, we realize all the examples in Section \ref{maingoal2} as fibered gradent flows, many of them associated with non-convex but smooth energy functionals. Finally, we provide sufficient conditions for the main three types of functionals to satisfy framework $\framework$, so that our theory of gradient flows can be applied.

\subsection{Examples of fibered gradient flows}\label{subsec:examples-applications}

In this section we focus on a sufficiently large class of functionals containing the examples in Section \ref{maingoal2}. We shall stick to formal arguments for clarity of the presentation, but rigorous statements can be found in the next section. These functionals can be regarded as natural fibered analogues of the three classical types of energy functionals introduced in \eqref{E-intro-energies}, and they take the form $\mathcal{E}:\mathcal{P}_{2,\nu}(\mathbb{R}^{2d})\longrightarrow (-\infty,+\infty]$ with
\begin{align}\label{E-energies-fibered}
\begin{aligned}
\mathcal{E}[\mu]&=\int_{\mathbb{R}^{2d}} U(\omega,\rho(x,\omega))\,dx\,d\nu(\omega)+\int_{\mathbb{R}^{2d}}V(\omega,x)\,d\mu(x,\omega)\\
&+\frac{1}{2}\iint_{\mathbb{R}^{2d}\times \mathbb{R}^{2d}}W(\omega,\omega',x-x')\,d\mu(x,\omega)\,d\mu(x',\omega'),
\end{aligned}
\end{align}
if $\mu=\rho(x,\omega)\,dx\otimes d\nu(\omega)\in \mathcal{P}_{2,\nu}(\mathbb{R}^{2d})$ and $\mathcal{E}[\mu]=+\infty$ otherwise. We note that all these functionals are particular instances of more general integral functionals:
\begin{align}
\mathcal{E}_1[\mu]&:=\int_{\mathbb{R}^{2d}}F_1(x,\omega,\rho(x,\omega))\,dx\,d\nu(\omega),\label{E-functional-type-I}\\
\mathcal{E}_2[\mu]&:=\int_{\mathbb{R}^{4d}}F_2(x,x',\omega,\omega',\rho(x,\omega),\rho(x',\omega'))\,dx\,dx'\,d\nu(\omega)\,d\nu(\omega'),\label{E-functional-type-II}
\end{align}
whenever $\mu=\rho(x,\omega)\,dx\otimes \nu(\omega)\in \mathcal{P}_{2,\nu}(\mathbb{R}^{2d})$ for some $\rho\in L^1_{dx\otimes \nu}(\mathbb{R}^{2d})$, and $\mathcal{E}_1[\mu]=\mathcal{E}_2[\mu]=+\infty$ otherwise, for regular enough integrands $F_1:\mathbb{R}^{2d}\times \mathbb{R}_+\longrightarrow \mathbb{R}_+$ and $F_2:\mathbb{R}^{4d}\times \mathbb{R}_+^2\longrightarrow \mathbb{R}_+$. Then, we obtain the following characterization of the fibered Fr\'echet subdifferential of $\mathcal{E}_1$ and $\mathcal{E}_2$ in terms of their associated Euler first variations.

\begin{lem}[Variational integral I]\label{L-variational-integral-I}
For any integrand $F_1:\mathbb{R}^{2d}\times \mathbb{R}_+\longrightarrow \mathbb{R}_+$ of class $C^2$ with $F(x,\omega,z)=0$ at $z=0$ let us define $\mathcal{E}_1:\mathcal{P}_{2,\nu}(\mathbb{R}^{2d})\longrightarrow (-\infty,+\infty]$ by \eqref{E-functional-type-I}. Then,
$$\partial_{W_{2,\nu}}\mathcal{E}_1[\mu]=\left\{\nabla_x \frac{\delta\mathcal{E}_1}{\delta\rho}\right\},$$
for all $\mu\in D(\mathcal{E}_1)$, where $\frac{\delta\mathcal{E}_1}{\delta\rho}$ is the first variation of $\mathcal{E}_1$ as functional of $\rho$, {\it i.e.},
$$\frac{\delta\mathcal{E}_1}{\delta\rho}(x,\omega)=\partial_z F_1(x,\omega,\rho(x,\omega)).$$
\end{lem}

\begin{lem}[Variational integral II]\label{L-variational-integral-II}
For any integrand $F_2:\mathbb{R}^{4d}\times \mathbb{R}_+^2\longrightarrow \mathbb{R}_+$ of class $C^2$ with $F(x,x',\omega,\omega',z,z')=0$ at $zz'=0$ let us define  $\mathcal{E}_2:\mathcal{P}_{2,\nu}(\mathbb{R}^{2d})\longrightarrow (-\infty,+\infty]$ by \eqref{E-functional-type-II}. Then,
$$\partial_{W_{2,\nu}}\mathcal{E}_2[\mu]=\left\{\nabla_x \frac{\delta\mathcal{E}_2}{\delta\rho}\right\},$$
for all $\mu\in D(\mathcal{E}_2)$, where $\frac{\delta\mathcal{E}_2}{\delta\rho}$ is the first variation of $\mathcal{E}_2$ as functional of $\rho$, {\it i.e.},
\begin{align*}
\frac{\delta\mathcal{E}_2}{\delta\rho}(x,\omega)&=\int_{\mathbb{R}^{2d}}\partial_z F_2(x,x',\omega,\omega',\rho(x,\omega),\rho(x',\omega'))\,dx'\,d\nu(\omega')\\
&+ \int_{\mathbb{R}^{2d}}\partial_{z'} F_2(x,x',\omega',\omega,\rho(x',\omega'),\rho(x,\omega))\,dx'\,d\nu(\omega').
\end{align*}
\end{lem}

The proofs can be regarded as fibered extensions of the results in \cite[Section 10.4]{AGS-08} for the classical Wasserstein space. For simplicity, we provide a proof of the second result. 

\medskip

\begin{proof}[Proof of Lemma \ref{L-variational-integral-II}]
Let us set $\mu\in \mathcal{P}_{2,\nu}(\mathbb{R}^{2d})$ of the form $\mu(x,\omega)=\rho(x,\omega)\,dx\otimes \nu(\omega)$ for some $\rho\in L^1_{dx\otimes \nu}(\mathbb{R}^{2d})$, and consider any $\boldsymbol{u}\in \partial_{W_{2,\nu}}\mathcal{E}_2[\mu]$. By the smoothness of $\mathcal{E}_2$, we infer that $\boldsymbol{u}$ reduces to the Fr\'echet gradient, which we shall identify in the sequel. Set any smooth tangent vector $\boldsymbol{\xi}=\nabla_x\varphi\in {\rm Tan}_\mu \mathcal{P}_{2,\nu}(\mathbb{R}^{2d})$ for $\varphi\in C^\infty_c(\mathbb{R}^{2d})$ and define the perturbations
$$\mu_\varepsilon(x,\omega):=((I+\varepsilon\,\boldsymbol{\xi}(\cdot,\omega))_{\#} \mu^\omega)\otimes\nu(\omega)=\rho_\varepsilon(x,\omega)\,dx\otimes \nu(\omega),$$
for small enough $\varepsilon\in \mathbb{R}_+$. By the change of variable theorem we indeed have 
$$\rho_\varepsilon(x,\omega):=\left(\frac{\rho(\cdot,\omega)}{\det \nabla(I+\varepsilon\,\boldsymbol{\xi}(\cdot,\omega))}\circ (I+\varepsilon\,\boldsymbol{\xi}(\cdot,\omega))^{-1}\right)(x),$$
which is well defined for small enough $\varepsilon\in \mathbb{R}_+$ thanks to the compact support of $\boldsymbol{\xi}$. Therefore,
\begin{align*}
\int_{\mathbb{R}^{2d}}\boldsymbol{u}\cdot \boldsymbol{\xi}\,d\mu&=\left.\frac{d}{d\varepsilon}\right\vert_{\varepsilon=0}\mathcal{E}_2[\mu_\varepsilon]=\left.\frac{d}{d\varepsilon}\right\vert_{\varepsilon=0}\int_{\mathbb{R}^{4d}} F_2(x,x',\omega,\omega',\rho_\varepsilon(x,\omega),\rho_\varepsilon(x',\omega'))\,dx\,dx'\,d\nu(\omega)\,d\nu(\omega')\\
&=\int_{\mathbb{R}^{4d}}\partial_z F_2(x,x',\omega,\omega',\rho(x,\omega),\rho(x',\omega'))\left.\frac{d}{d\varepsilon}\right\vert_{\varepsilon=0}\rho_\varepsilon(x,\omega)\,dx\,dx'\,d\nu(\omega)\,d\nu(\omega')\\
&+\int_{\mathbb{R}^{4d}}\partial_{z'} F_2(x,x',\omega,\omega',\rho(x,\omega),\rho(x',\omega'))\left.\frac{d}{d\varepsilon}\right\vert_{\varepsilon=0}\rho_\varepsilon(x',\omega')\,dx\,dx'\,d\nu(\omega)\,d\nu(\omega')\\
&=-\int_{\mathbb{R}^{2d}}\frac{\delta\mathcal{E}_2}{\delta\rho}(x,\omega)\divop_x(\boldsymbol{\xi}(x,\omega)\rho(x,\omega))\,dx\,d\nu(\omega)=\int_{\mathbb{R}^{2d}}\nabla_x \frac{\delta\mathcal{E}_2}{\delta\rho}\cdot \boldsymbol{\xi}\,d\mu.
\end{align*}
where above we have used the straightforward relation 
$$\left.\frac{d}{d\varepsilon}\right\vert_{\varepsilon=0}\rho_\varepsilon(x,\omega)=-\divop_x(\boldsymbol{\xi}(x,\omega)\rho(x,\omega)),$$
and we have integrated by parts in the last step. Since $\boldsymbol{\xi}=\nabla_x\varphi$ span all ${\rm Tan}_\mu\mathcal{P}_{2,\nu}(\mathbb{R}^{2d})$ by density according to Definition \ref{D-tan}, then we conclude that $\boldsymbol{u}=\nabla_x\frac{\delta\mathcal{E}_2}{\delta\rho}$.
\end{proof}

Of course, Lemmas \ref{L-variational-integral-I} and \ref{L-variational-integral-II} guarantee the PDE reformulation \eqref{kurak-variation} of the fibered gradient flows associated with functionals $\mathcal{E}:\mathcal{P}_{2,\nu}(\mathbb{R}^{2d})\longrightarrow (-\infty,+\infty]$ having the form $\mathcal{E}=\mathcal{E}_1+\mathcal{E}_2$, at least when these are smooth enough. As a consequence, all the examples in Section \ref{maingoal2} are fibered gradient flows associated to appropriate energy functionals.

\begin{rem}[Examples of fibered gradient flows]\label{R-examples-from-intro}

~\smallskip

\noindent $(i)$ {\bf The Kuramoto-Sakaguchi equation}. Consider $\nu\in \mathcal{P}(\mathbb{R})$,
\begin{align*}
\begin{aligned}
&F_1=-\theta\,\omega\,z, && (\theta,\omega,z)\in \mathbb{R}\times \mathbb{R}\times \mathbb{R}_+,\\
&F_2=-\frac{K}{2}\cos(\theta-\theta')\,z\,z', && (\theta,\theta',\omega,\omega',z,z')\in \mathbb{R}^2\times \mathbb{R}^2\times\mathbb{R}_+,
\end{aligned}
\end{align*}
and set the functional $\mathcal{E}=\mathcal{E}_1+\mathcal{E}_2$ as in \eqref{E-functional-type-I}-\eqref{E-functional-type-II}. Then we have
$$\partial_\theta\frac{\delta\mathcal{E}}{\delta\rho}=-\omega-K\int_{\mathbb{R}^2}\sin(\theta'-\theta)\,d\mu(\theta',\omega'),$$
for any $\mu(\theta,\omega)=\rho(\theta,\omega)\,dx\otimes \nu(\omega)\in \mathcal{P}_{2,\nu}(\mathbb{R}^{2d})$. Therefore, the Kuramoto-Sakaguchi equation \eqref{E-Kuramoto-Sakaguchi} agrees with the fibered gradient flow of such $\mathcal{E}$.

\medskip

\noindent $(ii)$ {\bf The kinetic Lohe matrix model}. Consider $\nu\in \mathcal{P}([0,1])$, 
\begin{align*}
\begin{aligned}
&F_1=-d\,\theta\,\omega\,z, & & (\theta,V,\omega,z)\in \mathbb{R}\times \boldsymbol{SU}(d)\times \mathbb{R}\times \mathbb{R}_+,\\
&F_2=-\frac{K}{2}{\rm tr}\left(e^{i(\theta-\theta')}V\,{V'}^\dagger\right)z\,z', && (\theta,\theta',\omega,\omega',V,V',z,z')\in \mathbb{R}^2\times \boldsymbol{SU}(d)^2\times \mathbb{R}^2\times \mathbb{R}_+^2,
\end{aligned}
\end{align*}
and set the functional $\mathcal{E}=\mathcal{E}_1+\mathcal{E}_2$ as in \eqref{E-functional-type-I}-\eqref{E-functional-type-II}. Then, we have
\begin{align*}
&\partial_\theta\frac{\delta\mathcal{E}}{\delta\rho}=-\omega-K\int_{\mathbb{R}\times \boldsymbol{SU}(d)\times \mathbb{R}} {\rm Im}\, {\rm tr}\left(V'\,V^\dagger e^{i(\theta'-\theta)}\right)\,d\mu(\theta',V',\omega'),\\
&\nabla_V\frac{\delta\mathcal{E}}{\delta\rho}=-\bigg[\frac{K}{2}\int_{\mathbb{R}\times \boldsymbol{SU}(d)\times \mathbb{R}} \bigg(V'\,V^\dagger e^{i(\theta'-\theta)}-V\,{V'}^\dagger e^{i(\theta-\theta')}\\
&\hspace{5.2cm}-\frac{1}{d}{\rm tr}(V'\,V^\dagger e^{i(\theta'-\theta)}-V\,{V'}^\dagger e^{i(\theta-\theta')})I_d\bigg)\,d\mu(\theta',V',\omega')\bigg]V,
\end{align*}
for any $\mu(\theta,V,\omega)=\rho(\theta,V,\omega)\,d\theta\otimes dV\otimes \nu(\omega)\in \mathcal{P}_{2,\nu}(\mathbb{R}\times \boldsymbol{SU}(d)\times \mathbb{R})$, see \cite[Proposition 5.4]{HKR-18}. Therefore, the kinetic Lohe matrix model \eqref{E-kinetic-Lohe} agrees with the fibered gradient flow of such $\mathcal{E}$.

\medskip

\noindent $(iii)$ {\bf Non-exchangeable/multispecies systems}. Consider $\nu\in \mathcal{P}([0,1])$,
\begin{align*}
\begin{aligned}
&F_1=\sigma\,z\log z, && (x,\omega,z)\in \mathbb{R}^d\times \mathbb{R}^d\times \mathbb{R}_+,\\
&F_2=\frac{1}{2}\alpha(\omega,\omega')\,W(x-x')\,z\,z', && (x,x',\omega,\omega',z,z')\in \mathbb{R}^{2d}\times \mathbb{R}^{2d}\times \mathbb{R}_+^2,
\end{aligned}
\end{align*}
and set the funtional $\mathcal{E}=\mathcal{E}_1+\mathcal{E}_2$ as in \eqref{E-functional-type-I}-\eqref{E-functional-type-II}. Then we have
$$\nabla_x\frac{\delta\mathcal{E}}{\delta\rho}=\int_{\mathbb{R}^d}\int_0^1\alpha(\omega,\omega')\,\nabla W(x-x')\,d\mu(x',\omega')+\sigma\nabla_x\log \rho(x,\omega),$$
for any $\mu(x,\omega)=\rho(x,\omega)\,dx\otimes \nu(\omega)\in \mathcal{P}_{2,\nu}(\mathbb{R}^d\times [0,1])$. Therefore, the kinetic non-exchangeable or multispecies system \eqref{E-kinetic-nonexchangeable} agrees with the fibered gradient flow of such $\mathcal{E}$.

\medskip

\noindent $(iv)$ {\bf Kuramoto-like alignment dynamics}. Consider $\nu\in \mathcal{P}(\mathbb{R}^d)$,
\begin{align*}
\begin{aligned}
&F_1=-x\cdot\omega\,z, && (x,\omega,z)\in \mathbb{R}^d\times \mathbb{R}^d\times \mathbb{R}_+,\\
&F_2=W(x-x')\,z\,z', && (x,x',\omega,\omega',z,z')\in \mathbb{R}^{2d}\times \mathbb{R}^{2d}\times \mathbb{R}_+^2,
\end{aligned}
\end{align*}
and set the funtional $\mathcal{E}=\mathcal{E}_1+\mathcal{E}_2$ as in \eqref{E-functional-type-I}-\eqref{E-functional-type-II}. Then we have
$$\nabla_x\frac{\delta\mathcal{E}}{\delta\rho}=-\omega+K\int_{\mathbb{R}^{2d}}\nabla W(x-x')\,d\mu(x',\omega'),$$
for any $\mu(x,\omega)=\rho(x,\omega)\otimes dx\otimes \nu(\omega)\in \mathcal{P}_{2,\nu}(\mathbb{R}^{2d})$. Therefore, the Kuramoto-type alignment dynamics \eqref{kurakk} in Section \ref{maingoal3} agrees with the gradient flow of $\mathcal{E}$.
\end{rem}

\begin{rem}
Whilst examples \eqref{E-Kuramoto-Sakaguchi}, \eqref{E-kinetic-Lohe}, \eqref{E-kinetic-nonexchangeable}, \eqref{E-multispecies-PKS} in Section \ref{maingoal2}, and also \eqref{kurakk} in Section \ref{maingoal3} can all be regarded formally as fibered gradient flows associated with the energy functionals above, we remark that not all of them lie in the framework $\framework$. On the one hand,  the Kuramoto-Sakaguchi equation \eqref{E-Kuramoto-Sakaguchi} and the kinetic Lohe matrix model \eqref{E-kinetic-Lohe} do not satisfy any convexity property. However, we note that existence and uniqueness is not an issue in those cases as discussed in Remark \ref{R-why-lambda-convexity} since these energy functionals are actually smooth enough (analytic indeed). It can be an issue though for less regular energies as it is the case of the multispecies Patlak-Keller-Segel \eqref{E-multispecies-PKS}. On the other hand, examples like the kinetic non-exchangeable system \eqref{E-kinetic-nonexchangeable} or the Kuramoto-type kinetic system \eqref{kurakk} with lower semicontinuous $\lambda$-convex interaction potential $W$ perfectly lie in our framework $\framework$ as it becomes clear in the next section.
\end{rem}

\subsection{Functionals verifying framework $\framework$}

Our next goal is to propose some criteria to guarantee that the basic examples \eqref{E-energies-fibered} satisfy framework $\framework$, and thus Theorem \ref{t-A-rig} applies to them. Essentially, this will be a consequence of the following main lemma.

\begin{lem}\label{L-fibered-functional-I}
Consider $\nu\in{\mathcal P}_2(\RR^d)$ and define the functional $\mathcal{E}:{\mathcal P}_{2,\nu}(\RR^{2d})\longrightarrow (-\infty,+\infty]$ by
$$
\mathcal{E}[\mu]:= \int_{\RR^d}\mathscr{E} (\omega;\mu^\omega)\,d\nu(\omega),
$$
for each $\mu \in\mathcal{P}_{2,\nu}(\mathbb{R}^{2d})$, where $\mathscr{E}:\RR^{d}\times {\mathcal P}_2(\RR^d)\longrightarrow (-\infty,+\infty]$ is Borel-measurable with $\mathcal{P}_2(\mathbb{R}^d)$ endowed with the $W_2$ metric, and it satisfies the following assumptions:
    
\begin{enumerate}[label=(\roman*)]
    \item (Coercivity) There exists $C>0$ such that
    \begin{align*}
        \mathscr{E}(\omega;\sigma)\geq -C(|\omega|^2 + M_2(\sigma) +1),
    \end{align*}
    for all $\sigma\in {\mathcal P}_{2}(\RR^d)$ and $\nu$-a.e. $\omega\in \mathbb{R}^d$, with $M_2(\sigma):=\int_{\mathbb{R}^d}|x|^2\,d\sigma(x)$.
    \item (Proper) There exist $C>0$ and $\sigma_*\in D:=\cap_{\omega\in\RR^{d}}D(\mathscr{E}(\omega;\cdot))$ such that
    \begin{align*}
        \mathscr{E}(\omega;\sigma_*)\leq C(|\omega|^2 +1),
    \end{align*}
    for $\nu$-a.e. $\omega\in \mathbb{R}^d$.
    \item (Lower semicontinuity) The functional $\sigma\in \mathcal{P}_2(\mathbb{R}^d)\mapsto \mathscr{E}(\omega;\sigma)$ is lower semicontinuous in $({\mathcal P}_2(\RR^d),W_2)$ for $\nu$-a.e. $\omega\in \mathbb{R}^d$. 
    \item (Convexity) There is $\lambda\in\RR$ such that $\sigma\in \mathcal{P}_2(\mathbb{R}^d)\mapsto \mathscr{E}(\omega;\sigma)$ is $\lambda$-convex along every generalized geodesics in the classical $({\mathcal P}_2(\RR^d),W_2)$ sense for $\nu$-a.e. $\omega\in \mathbb{R}^d$.
\end{enumerate}
Then, $\mathcal{E}$ satisfies framework $\framework$, and thus Theorem \ref{t-A-rig} applies to it.
\end{lem}

\begin{proof}

$\diamond$ {\sc Step 1}: Good definition.\\
To show that $\mathcal{E}$ is well defined it suffices to prove that for any $\mu\in{\mathcal P}_{2,\nu}(\RR^{2d})$ the function $\omega\in \mathbb{R}^d\mapsto \mathscr{E}(\omega;\mu^\omega)$ is Borel-measurable and its negative part belongs to $L^1_\nu(\mathbb{R}^d)$, so that the integral is well defined and takes (possibly infinite) values in $(-\infty,+\infty]$. On the one hand, note that by Proposition \ref{P-Borel-family-vs-measurability}, for any $\mu\in {\mathcal P}_{2,\nu}(\RR^{2d})$ the function $\mathfrak{X}_\mu:\omega\in \mathbb{R}^d\mapsto \mu^\omega$ is Borel-measurable as a function from $\RR^d$ to $({\mathcal P}_2(\RR^d),W_2)$. Consequently, the function $\omega\in \mathbb{R}^d\mapsto \mathscr{E}(\omega;\mu^\omega)$ is Borel-measurable as a composition of Borel-measurable mappings. On the other hand, by assumption $(i)$ the integral of the negative part of $\mathscr{E}(\omega;\mu^\omega)$ satisfies
$$
    \int_{\RR^{d}} \mathscr{E}(\omega;\mu^\omega)^-\,d\nu(\omega) \lesssim \int_{\mathbb{R}^d}|\omega|^2\,d\nu(\omega) + \int_{\mathbb{R}^{2d}}|x|^2\,d\mu(x,\omega) +1<+\infty,
$$
which ensures that $\mathcal{E}$ is well defined and takes values in $(-\infty,+\infty]$.

\medskip

$\diamond$ {\sc Step 2}: The functional $\mathcal{E}$ is proper and coercive.\\
In this step we prove that $\mathcal{E}$ satisfies assumptions ${\mathcal F}_1$ and ${\mathcal F}_2$ from Definition \ref{D-framework-F}. Assumption $(ii)$ directly ensures the existence of $\sigma_*\in D$ such that
\begin{align*}
    \mathcal{E}[\sigma_*\otimes\nu] = \int_{\RR^d}\mathscr{E}(\omega;\sigma_*)\,d\nu(\omega) \lesssim M_2(\nu)+1<+\infty,
\end{align*}
which implies that $\sigma_*\otimes\nu\in D(\mathcal{E})$, and consequently $\mathcal{E}$ is proper. Regarding coercivity, set any $\sigma\in {\mathcal P}_{2,\nu}(\RR^{2d})$, $r>0$, and consider any $\mu\in \mathcal{P}_{2,\nu}(\mathbb{R}^{2d})$ such that $W_{2,\nu}(\sigma,\mu)\leq r$. Then, by assumption $(i)$ and by the triangle inequality we have
\begin{align*}
    \int_{\RR^d} \mathscr{E}(\omega,\mu^\omega)\,d\nu(\omega) &\gtrsim -M_2(\nu) - W_{2,\nu}^2(\mu,\delta_0\otimes \nu) -1\\
    & \gtrsim  - M_2(\nu) -r^2- W_{2,\nu}^2(\sigma,\delta_0\otimes\nu)-1>-\infty,
\end{align*}
which implies the coercivity of $\mathcal{E}$.

\medskip

$\diamond$ {\sc Step 3}: The functional $\mathcal{E}$ is lower semicontinuous.\\
We prove that $\mathcal{E}$ satisfies ${\mathcal F}_3$ from Definition \ref{D-framework-F}, namely lower semicontinuity. Consider any $\mu_n\rightarrow \mu$ in $W_{2,\nu}$. By definition of $W_{2,\nu}$, there exists a subsequence $\{n_k\}_{k\in \mathbb{N}}$ such that $\mu^\omega_{n_k}\to\mu^\omega$ in $({\mathcal P}_2(\RR^d),W_2)$ for $\nu$-a.e. $\omega\in \mathbb{R}^d$. Thus, by the Fatou lemma and assumption $(iii)$ we have
\begin{align*}
    \mathcal{E}[\mu]=\int_{\RR^d}\mathscr{E}(\omega;\mu^\omega)d\nu\leq\int_{\RR^d}\liminf_{k\to\infty}\mathscr{E}(\omega;\mu^\omega_{n_k})d\nu\leq \liminf_{k\to\infty}\int_{\RR^d}\mathscr{E}(\omega;\mu^\omega_{n_k})d\nu = \liminf_{k\to\infty}\mathcal{E}[\mu_{n_k}].
\end{align*}
Since the above can be repeated on every subsequence of $\{\mu_n\}_{n\in \mathbb{N}}$, this concludes the proof of the lower semicontinuity.

\medskip

$\diamond$ {\sc Step 4}: Convexity along fibered generalized geodesics.\\
 Finally, we prove that $\mathcal{E}$ satisfies ${\mathcal F}_4$ from Definition \ref{D-framework-F}, that is the $\lambda$-convexity along fibered generalized geodesics. We shall actually prove that $\mathcal{E}$ is $\lambda$-convex along every generalized geodesic, which is stronger than Definition \ref{D-gen-geodesic-convexity}, and we naturally exploit assumption $(iv)$.  Let us take a triple $\mu_*$, $\mu_0$ and $\mu_1$ in $\mathcal{P}_{2,\nu}(\mathbb{R}^{2d})$ as well as any curve $\mu_\theta^{0\rightarrow1}$ as in Definition \ref{D-gen-geodesic-convexity}. Then, by construction we know that the family $\mu_\theta^{\omega\,0\rightarrow1}$ is a classical generalized geodesic connecting $\mu_0^\omega$ and $\mu_1^\omega$ with base in $\mu_*^\omega$ for $\nu$-a.e. $\omega$. Therefore for every $\theta\in[0,1]$ we have
\begin{align*}
    \mathcal{E}[\mu_\theta^{0\rightarrow1}] &= \int_{\RR^d}\mathscr{E}(\omega;\mu_\theta^{\omega \,0\rightarrow1})d\nu(\omega)\\
    &\leq  \int_{\RR^d}\left((1-\theta) \mathscr{E}(\omega;\mu_0^\omega) + \theta \mathscr{E}(\omega;\mu_1^\omega) - \frac{\lambda}{2}\theta(1-\theta)\int_{\mathbb{R}^{3d}} |x_0-x_1|^2d\gamma^\omega(x_*,x_0,x_1) \right)\,d\nu(\omega)\\
    &= (1-\theta)\mathcal{E}[\mu_0] + \theta\mathcal{E}[\mu_1] - \frac{\lambda}{2}\theta(1-\theta) W_\gamma^2(\mu_0,\mu_1),
\end{align*}
which finishes the proof of $\lambda$-convexity and the proof of the lemma.
\end{proof}

Using Lemma \ref{L-fibered-functional-I} and a similar $2$-fiber version on the three types of energy functionals \eqref{E-energies-fibered}, and applying the classical result from \cite{AGS-08} on each fiber, we obtain the following result.

\begin{theo}\label{T-the-three-energies}
Consider $\nu\in \mathcal{P}_2(\mathbb{R}^d)$ and define the internal, external and interaction energy functionals $\mathcal{U},\mathcal{V},\mathcal{W}:\mathcal{P}_{2,\nu}(\mathbb{R}^{2d})\longrightarrow (-\infty,+\infty]$ by
\begin{align}\label{E-the-three-energies}
\begin{aligned}
\mathcal{U}[\mu]&:=\left\{\begin{array}{ll}\int_{\mathbb{R}^{2d}} U(\omega,\rho(x,\omega))\,dx\,d\nu(\omega), &\mbox{if }\mu=\rho(x,\omega)\,dx\otimes \nu(\omega),\\ +\infty, & \mbox{otherwise},\end{array}\right.\\
\mathcal{V}[\mu]&:=\int_{\mathbb{R}^{2d}} V(\omega,x)\,d\mu(x,\omega),\\
\mathcal{W}[\mu]&:=\frac{1}{2}\iint_{\mathbb{R}^{2d}\times \mathbb{R}^{2d}}W(\omega,\omega',x-x')\,d\mu(x,\omega)\,d\mu(x',\omega').
\end{aligned}
\end{align}
Assume that $U:\mathbb{R}^d\times \mathbb{R}_+\longrightarrow \mathbb{R}$, $V:\mathbb{R}^{2d}\longrightarrow \mathbb{R}$ and $W:\mathbb{R}^{2d}\times \mathbb{R}^d\longrightarrow\mathbb{R}$ are Borel-measurable mappings which satisfy the following assumptions:
\begin{enumerate}[label=(\roman*)]
\item (Internal energy) For $\nu$-a.e. $\omega\in \mathbb{R}^d$ the $\omega$-section $z\in \mathbb{R}_+\mapsto U(\omega,z)$ is a convex and lower semicontinuous function with superlinear growth such that 
\begin{equation}\label{E-hypothesis-internal-energy}
U(\omega,0)=0,\quad \liminf_{z\searrow 0} \inf_{\omega\in \mathbb{R}^d}\frac{U(\omega,z)}{z^\alpha(1+|\omega|^2)^{1-\alpha}}>-\infty,\quad \sup_{\omega\in \mathbb{R}^d}\frac{U(\omega,z_0)}{1+|\omega|^2}<\infty
\end{equation}
for some $\frac{d}{2+d}<\alpha< 1$ and some $z_0>0$, and in addition $z\in\mathbb{R}_+\mapsto z^d\, U(\omega,z^{-d})$ is convex and non-increasing.
\item (External energy) For $\nu$-a.e. $\omega\in \mathbb{R}^d$ the $\omega$-section $x\in \mathbb{R}^d\mapsto V(\omega,x)$ is proper, lower semicontinuous, $\lambda$-convex and bounded by
    \begin{align}\label{potential-eq1}
        |V(\omega,x)|\leq C(|\omega|^2 +|x|^2 +1).
    \end{align}
\item (Interaction energy) For $\nu\otimes\nu$-a.e. $\omega,\omega'\in \mathbb{R}^d$ the $(\omega,\omega')$-section $x\in \mathbb{R}^d\mapsto W(\omega,\omega',x)$ is proper, even, lower semicontinuous, $\lambda$-convex and bounded by
    \begin{align*}
        |W(\omega,\omega',x)|\leq C(|\omega|^2 +|\omega'|^2 +|x|^2 +1).
    \end{align*}
\end{enumerate}
Then $\mathcal{U}$, $\mathcal{V}$ and $\mathcal{W}$ satisfy framework $\framework$, and thus Theorem \ref{t-A-rig} applies to all of them. Moreover, functional $\mathcal{U}$ is convex, $\mathcal{V}$ is $\lambda$-convex and $\mathcal{W}$ is $\min\{0,\lambda\}$-convex along generalized geodesics. In addition, when restricted to probability measures in $\mathcal{P}_{2,\nu}(\mathbb{R}^{2d})$ with fixed center of mass in the variable $x$, functional $\mathcal{W}$ is actually $\lambda$-convex along generalized geodesics.
\end{theo}

\begin{proof}
Note that the three functionals above can be restated as follows
\begin{align*}
&\mathcal{U}[\mu]=\int_{\mathbb{R}^d}\mathscr{U}(\omega;\mu^\omega)\,d\nu(\omega), \quad\mathcal{V}[\mu]=\int_{\mathbb{R}^d}\mathscr{V}(\omega;\mu^\omega)\,d\nu(\omega),\\
&\mathcal{W}[\mu]=\iint_{\mathbb{R}^d\times \mathbb{R}^d}\mathscr{W}(\omega,\omega';\mu^\omega,\mu^{\omega'})\,d\nu(\omega)\,d\nu(\omega'),
\end{align*}
for the functionals
\begin{align*}
&\mathscr{U}(\omega;\sigma):=\left\{\begin{array}{ll}\int_{\mathbb{R}^d} U(\omega,\rho(x))\,dx, & \sigma(x)=\rho(x)\,dx,\\
+\infty, & \mbox{otherwise},\end{array}\right.\quad \mathscr{V}(\omega;\sigma):=\int_{\mathbb{R}^d} V(\omega,x)\,d\sigma(x),\\
&\mathscr{W}(\omega,\omega';\sigma,\sigma'):=\frac{1}{2}\int_{\mathbb{R}^d\times \mathbb{R}^d}W(\omega,\omega',x-x')\,d\sigma(x)\,d\sigma'(x'),
\end{align*}
with $\omega,\omega'\in\mathbb{R}^d$ and $\sigma,\sigma'\in \mathcal{P}_2(\mathbb{R}^d)$. Therefore, the thesis of this result follows by applying Lemma \ref{L-fibered-functional-I} to $\mathcal{U}$, $\mathcal{V}$ and $\mathcal{W}$. This of course requires showing that the fibered functionals $\mathscr{U}$, $\mathscr{V}$ and $\mathscr{W}$ on the satisfy the appropriate hypothesis therein. This will be a clear consequence of the classical results for the internal, external and interaction energies over the classical Wasserstein space $(\mathcal{P}_2(\mathbb{R}^d),W_2)$ under the above-mentioned hypothesis on $U$, $V$ and $W$ (see \cite[Section 9.3]{AGS-08}).

\medskip

$\diamond$ {\sc Step 1}: Properties of $\mathscr{U}$.\\
We prove that $\mathscr{U}(\omega;\cdot)$ satisfies all the hypothesis $(i)$, $(ii)$, $(iii)$ and $(iv)$ in Lemma \ref{L-fibered-functional-I} by virtue of the classical theory. Let us consider any $\mu=\rho(x,\omega)\,dx\otimes \nu(\omega)\in \mathcal{P}_{2,\nu}(\mathbb{R}^{2d})$ for some $\rho\in L^1_{dx\otimes \nu}(\mathbb{R}^{2d})$ and note first that the hypothesis $\eqref{E-hypothesis-internal-energy}_2$ allows bounding $U(\omega,z)^-$ as follows
\begin{equation}\label{E-coercivity-internal-energy}
U(\omega,z)^-\leq C(1+|\omega|^2)^{1-\alpha}(z+z^\alpha),
\end{equation}
for some constant $C>0$, every $z\in \mathbb{R}_+$ and $\omega\in \mathbb{R}^d$.  Indeed, note first that by $\eqref{E-hypothesis-internal-energy}_2$ there exist constants $C_1>0$ and $\delta>0$ such that we obtain the local estimate
\begin{equation}\label{E-coercivity-internal-energy-1}
U(\omega,z)\geq -C_1 (1+|\omega|^2)^{1-\alpha}z^\alpha,
\end{equation}
for any $z\in [0,\delta]$ and each $\omega\in \mathbb{R}^d$. Second, the convexity of $U(\omega,\cdot)$ and $\eqref{E-hypothesis-internal-energy}_1$ yields
$$U(\omega,z)=\frac{U(\omega,z)}{z}z\geq \frac{U(\omega,\delta)}{\delta}z,$$
for any $z\in [\delta,+\infty)$ and $\omega\in \mathbb{R}^d$. Using \eqref{E-coercivity-internal-energy-1} with $z=\delta$ above implies
\begin{equation}\label{E-coercivity-internal-energy-2}
U(\omega,z)\geq -C_2(1+|\omega|^2)^{1-\alpha} z,
\end{equation}
for any $z\in [\delta,+\infty)$ and $\omega\in \mathbb{R}^d$, where $C_2:=\frac{C_1}{\delta^{1-\alpha}}$. Hence, putting \eqref{E-coercivity-internal-energy-1}-\eqref{E-coercivity-internal-energy-2} together implies the claimed global bound \eqref{E-coercivity-internal-energy} with $C:=\max\{C_1,C_2\}$. Therefore we have
$$\mathscr{U}(\omega;\mu^\omega)\gtrsim -C(1+|\omega|^2)^{1-\alpha}\left(1+\int_{\mathbb{R}^d}\rho(x,\omega)^\alpha\,dx\right),$$
which we can control by H\"{o}lder's inequality as follows
$$\int_{\mathbb{R}^d}\rho(x,\omega)^\alpha\,dx\leq \left(\int_{\mathbb{R}^d}(1+|x|^2)\rho(x,\omega)\,dx\right)^\alpha\left(\int_{\mathbb{R}^d}\frac{dx}{(1+|x|^2)^\frac{2\alpha}{1-\alpha}}\right)^{1-\alpha}.$$
We note that the second factor in the right-hand side is finite thanks to the assumption on $\alpha$. Hence, Young's inequality implies that $\mathscr{U}$ must satisfy the coercivity condition $(i)$. In addition, set $\sigma_*(x)=\rho_*(x)\,dx$ for the density $\rho_*(x)=z_0\chi_{B_{R_0}}(x)$ with $|B_{R_0}|=\frac{1}{z_0}$. Then, by the convexity of $U$ along with the hypothesis $\eqref{E-hypothesis-internal-energy}_1$ and $\eqref{E-hypothesis-internal-energy}_3$ we have
$$U(\omega,\rho_*(x))\leq \max\{0,U(\omega,z_0)\}\lesssim 1+|\omega|^2,$$
for all $x\in \mathbb{R}^d$ and therefore we obtain
$$\mathscr{U}(\omega,\sigma_*)=\int_{B_{R_0}}U(\omega,\rho_*(x))\,dx\lesssim 1+|\omega|^2,$$
thus proving $(ii)$. Finally, the lower semicontinuity $(iii)$ and the convexity along generalized geodesics $(iv)$ follows by the classical theory from the hypothesis on the superlinear growth and the convexity properties of $U(\omega,\cdot)$, see Remark 9.3.8 and Proposition 9.3.9 in \cite{AGS-08}.

\medskip

$\diamond$ {\sc Step 2}: Properties of $\mathscr{V}$.\\
Again, by the classical theory we shall prove that $\mathscr{V}(\omega;\cdot)$ verifies all the hypothesis $(i)$, $(ii)$, $(iii)$ and $(iv)$ in Lemma \ref{L-fibered-functional-I}. Indeed, we have
$$\vert \mathscr{V}(\omega;\sigma)\vert\lesssim\int_{\mathbb{R}^d} (|\omega|^2+|x|^2+1)\,d\sigma(x)=|\omega|^2+M_2(\sigma)+1,$$
thus guaranteeing the coercivity assumption $(i)$ and properness $(ii)$. In addition, since $V(\omega,\cdot)$ is lower semicontinuous and its negative part $V(\omega,\cdot)^-$ has subquadratic growth at infinity, we obtain that $\mathscr{V}(\omega;\cdot)$ is lower semicontinuous in $(\mathcal{P}_2(\mathbb{R}^d),W_2)$ and thus verifies $(iii)$. Finally, $(iv)$ follows from the $\lambda$-convexity of $V(\omega,\cdot)$ by the classical theory, see \cite[Proposition 9.3.2 ]{AGS-08}.

\medskip

$\diamond$ {\sc Step 3}: Properties of $\mathscr{W}$.\\
Let us note that framework $\framework$ cannot be directly inferred for the functional $\mathcal{W}$ through the above $1$-fiber result in Lemma \ref{L-fibered-functional-I}. Instead, we actually need to use a similar version for $2$-fiber functionals like $\mathscr{W}$, which we have not stated for the sake of simplicity. Indeed, following the same arguments as above, $\mathscr{W}$ satisfies the $2$-fiber analogous conditions to those in $(i)$, $(ii)$, $(iii)$ and $(iv)$ in Lemma \ref{L-fibered-functional-I} thanks to the classical theory. In particular, note that
\begin{align*}
\vert \mathscr{W}(\omega,\omega',\sigma,\sigma')\vert&\lesssim \iint_{\mathbb{R}^d\times \mathbb{R}^d}(|\omega|^2+|\omega'|^2+|x-x'|^2+1)\,d\sigma(x)\,d\sigma(x')\\
&\lesssim |\omega|^2+|\omega'|^2+M_2(\sigma)+M_2(\sigma')+1,
\end{align*}
thus yielding the analogous conditions to $(i)$ and $(ii)$. Moreover, $W(\omega,\omega',\cdot)$ is lower semicontinuous and $W(\omega,\omega',\cdot)^-$ has subquadratic growth at infinity, which ensures that $\mathscr{W}(\omega,\omega';\cdot)$ must be lower semicontinuous in $(\mathcal{P}_2(\mathbb{R}^d)\times \mathcal{P}_2(\mathbb{R}^d),W_2\times W_2)$ and thus verifies $(iii)$.  Arguing like in Lemma \ref{L-fibered-functional-I} this yields properties $\mathcal{F}_1$, $\mathcal{F}_2$ and $\mathcal{F}_3$ in framework $\framework$. The only delicate point is the $\min\{\lambda,0\}$-convexity of $\mathcal{W}$ along generalized geodesics, which follows an argument inspired in the classical McCann's convexity criterion \cite{M-97} as we sketch below. 

Without loss of generality we assume $\lambda\leq 0$. Otherwise, we note that since $W$ is also $0$-convex we may repeat the same argument below replacing positive $\lambda$ by $0$. Consider $\mu_*,\,\mu_0,\,\mu_1\in \mathcal{P}_{2,\nu}(\mathbb{R}^{2d})$ as well as any curve $\mu_\theta^{0\rightarrow1}$ as in Definition \ref{D-gen-geodesic-convexity}, which we disintegrate as follows
$$\mu_\theta^{0\rightarrow 1}(x,\omega)=\mu_\theta^{\omega\,0\rightarrow 1}(x)\otimes \nu(\omega),$$
for a generalized geodesic $\mu_\theta^{\omega\,0\rightarrow 1}$ joining $\mu_0^\omega$ to $\mu_1^\omega$ with base $\mu_*^\omega$. Then, by the $\lambda$-convexity of $W$ we obtain the following $2$-fiber version of $(iv)$ in Lemma \ref{L-fibered-functional-I}:
\begin{align*}
\mathscr{W}&(\omega,\omega';\mu_\theta^{\omega\,0\rightarrow 1},\mu_\theta^{\omega'\,0\rightarrow 1})\\
&=\frac{1}{2}\iint_{\mathbb{R}^{3d}\times \mathbb{R}^{3d}}W(\omega,\omega',(1-\theta)(x_0-x_0')+\theta (x_1-x_1'))\,d\gamma^\omega(x_*,x_0,x_1)\,d\gamma^{\omega'}(x_*',x_0',x_1')\\
&\leq (1-\theta)\mathscr{W}(\omega,\omega';\mu_0^\omega,\mu_0^{\omega'})+\theta\mathscr{W}(\omega,\omega';\mu_1^\omega,\mu_1^{\omega'})\\
&\quad-\frac{\lambda}{4}(1-\theta)\theta\iint_{\mathbb{R}^{3d}\times \mathbb{R}^{3d}}|(x_0-x_0')-(x_1-x_1')|^2\,d\gamma^\omega(x_*,x_0,x_1)\,d\gamma^{\omega'}(x_*,x_0',x_1')\\
&=(1-\theta)\mathscr{W}(\omega,\omega';\mu_0^\omega,\mu_0^{\omega'})+\theta\mathscr{W}(\omega,\omega';\mu_1^\omega,\mu_1^{\omega'})-\frac{\lambda}{4}(1-\theta)\theta\left(W_{\gamma^\omega}^2(\mu_0^\omega,\mu_1^\omega)+W_{\gamma^{\omega'}}^2(\mu_0^{\omega'},\mu_1^{\omega'})\right)\\
&\quad +\frac{\lambda}{2}(1-\theta)\theta\left(\int_{\mathbb{R}^d} x\,d\mu_0^\omega(x)-\int_{\mathbb{R}^d} x\,d\mu_1^\omega(x)\right)\cdot\left( \int_{\mathbb{R}^d} x\,d\mu_0^{\omega'}(x)-\int_{\mathbb{R}^d} x\,d\mu_1^{\omega'}(x)\right),
\end{align*}
for all $\theta\in [0,1]$ and $\nu$-a.e. $\omega,\omega'\in \mathbb{R}^d$. Integrating against $\nu\otimes \nu$ implies
\begin{align*}
\mathcal{W}[\mu_\theta^{0\rightarrow 1}]\leq (1-\theta)\mathcal{W}[\mu_0]+\theta\mathcal{W}[\mu_1]&-\frac{\lambda}{2}(1-\theta)\theta\,W_\gamma^2(\mu_0,\mu_1)\\
&+\frac{\lambda}{2}(1-\theta)\theta \left\vert\int_{\mathbb{R}^{2d}}x\,d\mu_0(x,\omega)-\int_{\mathbb{R}^{2d}}x\,d\mu_1(x,\omega)\right\vert^2,
\end{align*}
for all $\theta\in [0,1]$. Since $\lambda\leq 0$, then the last term is non-positive so that we readily obtain $\lambda$-convexity of $\mathcal{W}$ along the generalized geodesic $\theta\in [0,1]\mapsto \mu_\theta$. Note that for $\lambda>0$ the above argument fails because the last term becomes non-negative. In that case, we need to assume that $\mu_0$ and $\mu_1$ have the same center of mass with respect to $x$ in order to guarantee that such a non-negative term actually vanishes and obtain the claimed $\lambda$-convexity.
\end{proof}

\begin{rem}
In Lemma \ref{L-fibered-functional-I}, thus Theorem \ref{T-the-three-energies}, we can relax the assumption that $\nu\in {\mathcal P}_2(\RR^d)$ to only $\nu\in{\mathcal P}(\RR^d)$ at the cost of making all the bounds appearing in the assumptions therein $\omega$-independent.
\end{rem}

\section{Uniform contractivity of Kuramoto-type models}\label{apple}

From here on we shall focus on our third goal presented in Section \ref{maingoal3}, that is the Kuramoto-type equation \eqref{kurakk}, which we recollect bellow
\begin{align}\label{kurakk-again}
\begin{aligned}
&\partial_t\mu + \divop_x(\boldsymbol{u}[\mu]\mu) = 0,\qquad t\geq 0,\quad (x,\omega)\in \RR^{2d},\\
&\boldsymbol{u}[\mu_t](t,x,\omega):=\omega-K\int_{\RR^{2d}}\nabla W(x-x')\,d\mu_t(x',\omega').
\end{aligned}
\end{align}
We recall that its relevance stems from the perspective of second-order alignment dynamics with weakly singular influence function, as discussed in Section \ref{sec:intro}, see also \cite{PP-22-2-arxiv}. As anticipated formally in Remark \ref{R-examples-from-intro} $(iv)$, we expect distributional solutions of \eqref{kurakk-again} are equivalent to fibered gradient flow of the interaction energy functional $\mathcal{E}_W$ in \eqref{W}, namely,
\begin{align}\label{W-again}
\begin{aligned}
&{\mathcal E}_W[\mu]:= - \int_{\RR^{2d}} \omega\cdot x\,d\mu(x,\omega)+K\iint_{\RR^{2d}\times \RR^{2d}} W(x-x')\,d\mu(x,\omega)\, d\mu(x',\omega'),\\
&W(x) = \frac{1}{2-\alpha}\frac{1}{1-\alpha}|x|^{2-\alpha},\quad \nabla W(x) = \frac{1}{1-\alpha}\phi(|x|)x,\quad \phi(|x|)=\frac{1}{|x|^\alpha},
\end{aligned}
\end{align}
with exponent $\alpha \in (0,1)$. To make statements rigorous, we first recall the natural definition of distributional solutions to \eqref{kurakk-again}.

\begin{defi}[Weak measure-valued solutions]\label{D-distributional-solution-kurakk}
We say that $\boldsymbol{\mu}\in AC_{\rm loc}(0,\infty;\mathcal{P}_{2,\nu}(\mathbb{R}^{2d}))$ is a weak measure-valued solution of \eqref{kurakk-again} with $\nu\in \mathcal{P}_2(\mathbb{R}^d)$ when the velocity field $\boldsymbol{u}[\boldsymbol{\mu}]$ in $\eqref{kurakk-again}_2$ belongs to $L^1_{\rm loc}(0,\infty;L^2_{\mu_t}(\mathbb{R}^{2d},\mathbb{R}^d))$ and $\boldsymbol{\mu}$ satisfies \eqref{kurakk-again} in distributional sense, {\it i.e.},
$$\int_0^\infty \int_{\mathbb{R}^{2d}}\big(\partial_t\varphi(t,x,\omega)+\nabla_x\varphi(t,x,\omega)\cdot \boldsymbol{u}[\mu_t](x,\omega)\big)\,d\mu_t(x,\omega)\,dt=-\int_{\mathbb{R}^{2d}}\varphi(0,x,\omega)\,d\mu_0(x,\omega),$$
for all $\varphi\in C^\infty_c([0,\infty)\times \mathbb{R}^{2d})$.
\end{defi}

On the other hand, $\mathcal{E}_W$ in \eqref{W-again} can be written as the combination $\mathcal{E}_W=\mathcal{V}+\mathcal{W}$ of an external and an interaction energy functional like in \eqref{E-the-three-energies} with regular enough potentials
$$V(\omega,x)=-\omega\cdot x,\qquad W(\omega,\omega',x)=K\,W(x)=\frac{K\,|x|^{2-\alpha}}{(2-\alpha)(1-\alpha)},$$
which verify the hypothesis in Theorem \ref{T-the-three-energies} (in particular $V$ and $W$ are $0$-convex, {\it cf.} Appendix \ref{Appendix-convexity-W}). In addition, Lemmas \ref{L-variational-integral-I} and \ref{L-variational-integral-II} can be applied as in Remark \ref{R-examples-from-intro} $(iv)$ and we obtain the following lemma.

\begin{lem}[Fr\'echet subdifferential of $\mathcal{E}_W$]\label{L-subdifferential-slope-W}
Consider any $\nu\in \mathcal{P}_2(\mathbb{R}^d)$. Then, $\mathcal{E}_W$ in \eqref{W-again} satisfies framework $\framework$, it is convex along fibered generalized geodesics and
$$\partial_{W_{2,\nu}}\mathcal{E}_W[\mu]=\{-\boldsymbol{u}[\mu]\},$$
for any $\mu \in\mathcal{P}_{2,\nu}(\mathbb{R}^{2d})$, where $\boldsymbol{u}[\mu]$ is the velocity field in $\eqref{kurakk-again}_2$.
\end{lem}

It is now apparent that weak measure-valued solutions and gradient flows agree in our setting. In fact we will prove the following theorem.

\begin{theo}[Theorem B: Kuramoto-type equation]\label{t-B-rig}
Fix $\nu\in {\mathcal P}_2(\RR^d)$ and $\mu_0\in {\mathcal P}_{2,\nu}(\RR^{2d})$. Then equation \eqref{kurakk-again} issued at the initial datum $\mu_0$ has a unique weak measure-valued solution in the sense of Definition \ref{D-distributional-solution-kurakk}, and it is equivalent to a fibered gradient flow of the functional $\mathcal{E}_W$ in \eqref{W-again} in the sense of Definition \ref{D-solutions-1st-order-gradient-flow}. Moreover the stability/contractivity results presented below in Theorems \ref{T-contractivity-W2nu} and \ref{T-uniform-stability-AW2} as well as Corollaries \ref{C-contractivity-W2nu} and \ref{C-uniform-stability-AW2} hold true. 
\end{theo}

Note that by Lemma \ref{L-subdifferential-slope-W}, the first part of Theorem \ref{t-B-rig} follows directly from Theorem \ref{t-A-rig} and it provides well-posedness, regularity and finite-in-time stability of weak measure-valued solutions to \eqref{kurakk-again} in the $W_{2,\nu}$ topology. The remainder of this section is dedicated to prove the second part of Theorem \ref{t-B-rig}, which address the following two fundamental questions: 
\begin{enumerate}[label=(\roman*)]
\item Quantitative mean-field limit towards a weak measure-valued solution $\boldsymbol{\mu}$ of the kinetic equation \eqref{kurakk-again} 
\item Quantitative convergence rates as $t\rightarrow \infty$ of any weak measure-valued solution $\mu_t$ of the kinetic equation \eqref{kurakk-again} towards the equilibrium.
\end{enumerate}

We face two main difficulties as compared to classical arguments. On the one hand, our interaction force $-\nabla W$ is only one-sided Lipschitz. More specifically, $W$ is $0$-convex ({\it cf.} Appendix \ref{Appendix-convexity-W}), which is not enough to produce the uniform contractivity in Theorem \ref{t-A-rig} $(iii)$. On the other hand, the variable $\omega$ introduces an apparent heterogeneity which makes the arguments more subtle. We solve both problems $(i)$ and $(ii)$ simultaneously by showing uniform-in-time contractivity with respect to the so called {\it fibered transport pseudometric} (see Definitions \ref{D-fibered-Wasserstein-double} and \ref{D-liftings} below). This pseudometric admits comparing solutions $\mu_1$ and $\mu_2$ to \eqref{kurakk-again} with different $\omega$-marginals $\nu_1$ and $\nu_2$. In addition, when $\nu_1=\nu_2=\nu$ the pseudometric reduces to the fibered metric $W_{2,\nu}$. Therefore, the uniform-in-time contractivity with respect to the pseudometric enables us to recover both convergence to the equilibrium in $W_{2,\nu}$ metric and the mean-field limit in the so called {\it adapted Wasserstein distance} (also known as {\it nested} or {\it causal} distance) ({\it cf.} \cite{L-18,PP-12,PP-14}). In contrast with \cite{CZ-21}, our method of proof is suitable for higher dimension and the uniform-in-time mean-field limit arises as a direct consequence of the contraction in the kinetic equation, which does not exploit any stability of the associated particle system.

\subsection{A fibered transport pseudometric and the adapted Wasserstein distance}

Before stating our main contractivity result, we shall introduce a pair of transport distances that will be used throughout this section. We advance that the quadratic Wasserstein distance $W_2$ does not provide satisfying uniform-in-time contractivity results and we need to refine it appropriately. On the one hand, we define the spaces $\mathcal{P}_{\hat{\nu}}(\mathbb{R}^{3d})$ and $(\mathcal{P}_{2,\hat{\nu}}(\mathbb{R}^{3d}),W_{2,\hat{\nu}})$ with fiber distribution $\hat{\nu}\in \mathcal{P}(\mathbb{R}^{2d})$ similarly to Definitions \ref{D-fibered-measures} and \ref{D-fibered-Wasserstein}.

\begin{defi}[Fibered quadratic Wasserstein space II]\label{D-fibered-Wasserstein-double}
Let $\hat{\nu}\in \mathcal{P}(\mathbb{R}^{2d})$ be any probability measure. Then, we define the subset $\mathcal{P}_{\hat{\nu}}(\mathbb{R}^{3d})\subseteq \mathcal{P}(\mathbb{R}^{3d})$ of fibered probability measures, and the fibered quadratic Wasserstein space $(\mathcal{P}_{2,\hat{\nu}}(\mathbb{R}^{3d}),W_{2,\hat{\nu}})$ as follows
\begin{align*}
\mathcal{P}_{\hat{\nu}}(\mathbb{R}^{3d})&:=\{\bar \mu\in \mathcal{P}(\mathbb{R}^{3d}):\,\pi_{(\omega,\omega')\#}\bar\mu=\hat{\nu}\},\\
\mathcal{P}_{2,\hat{\nu}}(\mathbb{R}^{3d})&:=\left\{\bar\mu\in \mathcal{P}_{\hat{\nu}}(\mathbb{R}^{3d}):\,\int_{\mathbb{R}^{3d}}\vert x\vert^2\,d\bar\mu(x,\omega,\omega')<\infty\right\},\\
W_{2,\hat{\nu}}(\bar \mu_1,\bar \mu_2)&:=\left(\int_{\mathbb{R}^{2d}}W_2^2(\bar\mu_1^{\omega,\omega'},\bar\mu_2^{\omega,\omega'})\,d\hat{\nu}(\omega,\omega')\right)^{1/2},
\end{align*}
for any $\bar\mu_1,\bar\mu_2\in \mathcal{P}_{2,\hat{\nu}}(\mathbb{R}^{3d})$, where $\{\bar\mu_1^{\omega,\omega'}\}_{(\omega,\omega')\in \mathbb{R}^{2d}}$ and $\{\mu_2^{\omega,\omega'}\}_{(\omega,\omega')\in \mathbb{R}^{2d}}$ in $\mathcal{P}_2(\mathbb{R}^{2d})$ are the families of disintegrations with respect to the variables $(\omega,\omega')$.
\end{defi}

We note that in Definition \ref{D-fibered-Wasserstein-double} the fibers have twice the dimension of the fibers in Definitions \ref{D-fibered-measures} and \ref{D-fibered-Wasserstein}. Nevertheless, we remark that the fact that space $x\in \mathbb{R}^d$ and fibers $\omega\in \mathbb{R}^d$ have the same dimension was irrelevant in Sections \ref{subsec:random-probability-measures} and \ref{sec:fibered-wasserstein}. Indeed, all the results therein have straightforward analogues for the fibered spaces in Definition \ref{D-fibered-Wasserstein-double}. In the next section, we shall exploit some of those results ({\it e.g.} Definition \ref{D-tan} and Propositions \ref{equiv} and \ref{T-dif}). The goal of the enlarged fibered Wasserstein space $(\mathcal{P}_{2,\hat{\nu}}(\mathbb{R}^{3d}),W_{2,\hat{\nu}})$ in Definition \ref{D-fibered-Wasserstein-double} will be to serve as a covering space so that we can compare measures $\mu_1,\mu_2\in \mathcal{P}_2(\mathbb{R}^{2d})$ eventually having different distributions with respect to $\omega$. Specifically, we define the following lifting mappings.

\begin{defi}[Lifting]\label{D-liftings}
Consider any $\nu_1,\nu_2\in \mathcal{P}(\mathbb{R}^d)$ and set any $\hat{\nu}\in \Gamma(\nu_1,\nu_2)$. Then, we define $\mathcal{L}_{\hat{\nu},i}:\mathcal{P}_{\nu_i}(\mathbb{R}^{2d})\longrightarrow \mathcal{P}_{\hat{\nu}}(\mathbb{R}^d\times \mathbb{R}^{2d})$ with $i=1,2$ as follows
\begin{align*}
\mathcal{L}_{\hat{\nu},1}(\mu_1)(x,\omega,\omega')&:=\mu_1^\omega(x)\otimes \hat{\nu}(\omega,\omega'),\\
\mathcal{L}_{\hat{\nu},2}(\mu_2)(x,\omega,\omega')&:=\mu_2^{\omega'}(x)\otimes \hat{\nu}(\omega,\omega'),
\end{align*}
for any $\mu_1\in \mathcal{P}_{\nu_1}(\mathbb{R}^{2d})$ and any $\mu_2\in \mathcal{P}_{\nu_2}(\mathbb{R}^{2d})$, where $\{\mu_1^\omega\}_{\omega\in \mathbb{R}^d}$ and $\{\mu_2^{\omega'}\}_{\omega'\in \mathbb{R}^d}$ are the associated families of disintegrations.
\end{defi}

Then, we may use $W_{2,\hat{\nu}}(\mathcal{L}_{\hat{\nu},1}(\mu_1),\mathcal{L}_{\hat{\nu},2}(\mu_2))$ as a pseudometric between $\mu_1$ and $\mu_2$. Note that it is not a distance in full sense since it is degenerate. Namely, if $W_{2,\hat{\nu}}(\mathcal{L}_{\hat{\nu},1}(\mu_1),\mathcal{L}_{\hat{\nu},2}(\mu_2))=0$ then we cannot necessarily guarantee that $\mu_1$ and $\mu_2$ must agree if the marginals do not agree (that is, $\nu_1\neq \nu_2$). However, we have the following straightforward relationship with the fibered Wasserstein distance in Definition \ref{D-fibered-Wasserstein} when the marginals agree (that is, $\nu_1=\nu_2$).

\begin{rem}[Identical marginals]\label{R-fibered-Wasserstein-double-identical}
Assume that $\nu_1=\nu_2=: \nu$, and set the special transference plan $\hat{\nu}(\omega):=\nu(\omega)\otimes \delta_\omega(\omega')$. Then, we obtain
$$W_{2,\hat{\nu}}(\mathcal{L}_{\hat{\nu},1}(\mu_1),\mathcal{L}_{\hat{\nu},2}(\mu_2))=W_{2,\nu}(\mu_1,\mu_2),$$
for any $\mu_1,\mu_2\in \mathcal{P}_{2,\nu}(\mathbb{R}^{2d})$.
\end{rem}

Note that in order to turn $W_{2,\hat{\nu}}(\mathcal{L}_{\hat{\nu},1}(\mu_1),\mathcal{L}_{\hat{\nu},2}(\mu_2))$ into a full metric over $\mathcal{P}_2(\mathbb{R}^{2d})$, we need to also account for the transportation cost in the variable $\omega$. This can be achieved by defining the following distance over $\mathcal{P}_2(\mathbb{R}^{2d})$, which has recently emerged in the literature of stochastic optimization under various names, {\it e.g.}, {\it adapted}, {\it nested} or {\it causal} Wasserstein distance ({\it cf.} \cite{L-18,PP-12,PP-14} and references therein).

\begin{defi}[Adapted Wasserstein distance]\label{D-adapted-Wasserstein-distance}
For any $\mu_1,\mu_2\in \mathcal{P}_2(\mathbb{R}^{2d})$, we define
\begin{equation}\label{E-adapted-Wassersein-distance}
AW_2(\mu_1,\mu_2):=\left(\inf_{\hat{\nu}\in \Gamma(\nu_1,\nu_2)}\int_{\mathbb{R}^{2d}} \left(W_2^2(\mu_1^\omega,\mu_2^{\omega'})+\vert \omega-\omega'\vert^2\right)\,d\hat{\nu}(\omega,\omega')\right)^{1/2},
\end{equation}
where $\nu_1:=\pi_{\omega\#}\mu_1$ and $\nu_2:=\pi_{\omega\#}\mu_2$, and $\{\mu_1^\omega\}_{\omega\in \mathbb{R}^d}$ and $\{\mu_2^{\omega'}\}_{\omega'\in \mathbb{R}^d}$ are the associated families of disintegrations.
\end{defi}

By standard methods one obtains that $AW_2$ is well defined over $\mathcal{P}_2(\mathbb{R}^{2d})$ and it is indeed a real distance. See \cite{BBP-21-arxiv} for the main properties and the underlying Riemannian structure over its completion, the space of filtered processes. In addition, under the notation in Definitions \ref{D-fibered-Wasserstein-double} and \ref{D-liftings} we can reformulate the adapted Wasserstein distance $AW_2$ in \eqref{E-adapted-Wassersein-distance} in terms of the fibered distances $W_{2,\hat{\nu}}$ as follows:
\begin{equation}\label{E-adapted-Wassersein-distance-restatement}
AW_2^2(\mu_1,\mu_2)=\inf_{\hat{\nu}\in \Gamma(\nu_1,\nu_2)} \left(W_{2,\hat{\nu}}^2(\mathcal{L}_{\hat{\nu},1}(\mu_1),\mathcal{L}_{\hat{\nu},2}(\mu_2))+\int_{\mathbb{R}^{2d}}\vert \omega-\omega'\vert^2\,d\hat{\nu}(\omega,\omega')\right).
\end{equation}

In addition, we obtain the following relation of the adapted Wasserstein distance with the classical Wasserstein distance over $\mathcal{P}_2(\mathbb{R}^{2d})$. We omit the straightforward proof.

\begin{pro}[$AW_2$ is stronger than $W_2$]\label{P-AW2-stronger-than-W2}
The following relation
$$W_2(\mu_1,\mu_2)\leq AW_2(\mu_1,\mu_2),$$
holds true for any $\mu_1,\mu_2\in \mathcal{P}_2(\mathbb{R}^{2d})$.
\end{pro}

The following property will be useful in the next section.

\begin{pro}[$AW_2$ vs translations]\label{P-AW2-translations}
Consider any $\mu\in \mathcal{P}_2(\mathbb{R}^{2d})$, any $\delta_x,\delta_\omega\in \mathbb{R}^d$, and define $\tilde{\mu}:=(x+\delta_x,\omega+\delta_\omega)_{\#}\mu$. Then, we have
$$AW_2^2(\mu,\tilde{\mu})=\vert \delta_x\vert^2+\vert \delta_\omega\vert^2.$$
\end{pro}

\begin{proof}
Let us consider $\nu:=\pi_{\omega\#}\mu$ and $\tilde{\nu}:=\pi_{\omega\#}\tilde{\mu}$. Then, we have
$$\tilde{\nu}=(\omega+\delta_\omega)_{\#}\nu,\quad \tilde{\mu}^{\omega+\delta_\omega}=(x+\delta_x)_{\#}\mu^\omega,$$
for $\nu$-a.e. $\omega\in \mathbb{R}^d$, where $\{\mu^\omega\}_{\omega\in \mathbb{R}^d}$ and $\{\tilde{\mu}^\omega\}_{\omega\in \mathbb{R}^d}$ are the families of disintegrations. Set the following transference plans
\begin{align*}
\hat{\nu}(\omega,\omega')&:=\nu(\omega)\otimes \delta_{\omega+\delta_\omega}(\omega')\in \Gamma(\nu,\tilde{\nu}),\\
\gamma^\omega(x,x')&:=\mu^\omega(x)\otimes \delta_{x+\delta_x}(x')\in \Gamma(\mu^\omega,\tilde{\mu}^{\omega+\delta_\omega}),
\end{align*}
for $\nu$-a.e. $\omega\in \mathbb{R}^d$. Then, by Definition \ref{D-adapted-Wasserstein-distance} we obtain
\begin{align*}
AW_2^2(\mu,\tilde{\mu})&\leq \int_{\mathbb{R}^{2d}}\left(\vert \omega-\omega'\vert^2+W_2^2(\mu^\omega,\tilde{\mu}^{\omega'})\right)\,d\hat{\nu}(\omega,\omega')\\
&=\int_{\mathbb{R}^d}\left(\vert \delta_\omega\vert^2+W_2^2(\mu^\omega,\tilde{\mu}^{\omega+\delta_\omega})\right)\,d\nu(\omega)\\
&\leq \vert \delta_\omega\vert^2+\int_{\mathbb{R}^d}\int_{\mathbb{R}^{2d}}\vert x-x'\vert^2\,d\gamma^\omega(x,x')\,d\nu(\omega)\\
&=\vert \delta_\omega\vert^2+\int_{\mathbb{R}^d}\int_{\mathbb{R}^d}\vert \delta_x\vert^2\,d\mu^\omega(x)\,d\nu(\omega)=\vert \delta_x\vert^2+\vert \delta_\omega\vert^2.
\end{align*}
The reverse inequality is clear by Proposition \ref{P-AW2-stronger-than-W2}.
\end{proof}

\subsection{Uniform-in-time contractivity in the fibered pseudometric}

In this section we prove the contractivity of $W_{2,\hat{\nu}}(\mathcal{L}_{\hat{\nu},1}(\boldsymbol{\mu}^1),\mathcal{L}_{\hat{\nu},2}(\boldsymbol{\mu}^2))$ along any couple of weak measure-valued solutions $\boldsymbol{\mu}^1$ and $\boldsymbol{\mu}^2$ of \eqref{kurakk-again} for any fixed transference plan $\hat{\nu}\in \Gamma(\nu_1,\nu_2)$. As a key observation, we prove that for compactly supported initial data the spacial diameter of the solutions to \eqref{kurakk-again} stays uniformly bounded indefinitely. This allows restricting to solutions of \eqref{kurakk-again} which lie in a special region of $\lambda$-convexity of the funcional $\mathcal{E}_W$ (recall again that in general we have $\lambda= 0$ in Theorem \ref{t-A-rig} $(iii)$). Such a property was already obtained in \cite{ZZ-20} for the analogous Kuramoto-type particle system in the one-dimensional setting. Later, it was exploited in \cite{CZ-21} to prove the uniform-in-time stability property of the particle system leading to uniform-in-time mean-field limits. Here, we address our $d$-dimensional problem \eqref{kurakk-again} at the purely kinetic level.

\begin{lem}[Uniform bounds of spacial diameters]\label{L-control-diameter}
Let $\mu_0\in \mathcal{P}(\mathbb{R}^{2d})$ be any compactly supported initial datum and consider the associated weak measure-valued solution $\boldsymbol{\mu}$ to \eqref{kurakk-again}. Define the partial supports in the variables $x$ and $\omega$ as follows
\begin{align}\label{E-partial-supports}
\begin{aligned}
\supp_x\mu_t&:=\{x\in\mathbb{R}^d:\,\exists\,\omega\in \mathbb{R}^d\mbox{ with }(x,\omega)\in \supp \mu_t\},\\
\supp_\omega\mu_0&:=\{\omega\in \mathbb{R}^d:\,\exists\,x\in \mathbb{R}^d\mbox{ with }(x,\omega)\in \supp \mu_0\},
\end{aligned}
\end{align}
and denote their associated diameters by
\begin{equation}\label{E-partial-diameters}
\mathcal{D}_x(t):=\diam\,(\supp_x \mu_t),\quad \mathcal{D}_\omega:=\diam\,(\supp_\omega\mu_0).\end{equation}
Then, the following uniform bound is fulfilled
$$\sup_{t\geq 0} \mathcal{D}_x(t)\leq \max\left\{\mathcal{D}_x(0),\left(\frac{\mathcal{D}_\omega}{K}\right)^{\frac{1}{1-\alpha}}\right\}.$$
\end{lem}

\begin{proof}
Consider the characteristic flow $X=X(t;x,\omega)$ associated with $\mu_t$:
\begin{equation}\label{E-characteristic-system}
\frac{d}{dt}X(t;x,\omega)=\boldsymbol{u}[\mu_t](X(t;x,\omega),\omega),\quad X(0;x,\omega)=x.
\end{equation}
Note that by Corollary \ref{C-sided-Lipschitz} in Appendix \ref{Appendix-convexity-W} the characteristic flow is indeed globally and uniquely defined forwards in time. In addition, by \cite[Theorem 8.2.1]{AGS-08} and the well posedness of \eqref{E-characteristic-system} the weak measure-valued solution $\boldsymbol{\mu}$ must propagate along the flow, namely
$$\mu_t=(X(t;x,\omega),\omega)_{\#}\mu_0,$$
for any $t\in \mathbb{R}_+$. Define the set of critical points where the spacial diameter is attained, {\it i.e.},
$$C(t):=\{(z_1,z_2)\in \supp\mu_0\times \supp\mu_0:\, \vert X(t;z_1)-X(t;z_2)\vert=\mathcal{D}_x(t)\},$$
where we denote $z_i=(x_i,\omega_i)$ for simplicity. Then, we obtain
$$\frac{1}{2}\mathcal{D}_x(t)^2=\max_{(z_1,z_2)\in C(t)} \frac{1}{2}\vert X(t;z_1)-X(t;z_2)\vert^2.$$
Since the flow $(t,z)\mapsto X(t;z)$ is continuous, differentiable with respect to $t$ with continuous derivative, we obtain from the usual argument in \cite[Corollary 3.3, Chapter 1]{DR-95}
\begin{equation}\label{E-diameter-Dx-derivative-0}
\frac{d^+}{dt}\frac{1}{2}\mathcal{D}_x(t)^2=\max_{(z_1,z_2)\in C(t)} (X(t;z_1)-X(t;z_2))\cdot \left(\frac{\partial X}{\partial t}(t;z_1)-\frac{\partial X}{\partial t}(t;z_2)\right),
\end{equation}
where $d^+/dt$ denotes the right derivative. Given any $(z_1,z_2)\in C(t)$ we obtain
\begin{align}\label{E-diameter-Dx-derivative-1}
\begin{aligned}
(X&(t;z_1)-X(t;z_2))\cdot \left(\frac{\partial X}{\partial t}(t;z_1)-\frac{\partial X}{\partial t}(t;z_2)\right)=(X(t;z_1)-X(t;z_2))\cdot (\omega_1-\omega_2)\\
&-K\int_{\mathbb{R}^{2d}}\left(\nabla W(X(t;z_1)-X(t;z))-\nabla W(X(t;z_2)-X(t;z))\right)\cdot (X(t;z_1)-X(t;z_2))\,d\mu_0(z).
\end{aligned}
\end{align}
In \eqref{E-diameter-Dx-derivative-1} we can bound the first term in the right hand side by the Cauchy--Schwarz inequality, and we can estimate the second term by using the convexity properties in Lemma \ref{C-convexity}. Therefore, \eqref{E-diameter-Dx-derivative-0} can be bounded as follows
\begin{equation}\label{E-diameter-Dx-derivative}
\frac{d^+}{dt}\frac{1}{2}\mathcal{D}_x(t)^2\leq \mathcal{D}_x(t)\mathcal{D}_\omega -K\phi(\mathcal{D}_x(t))\mathcal{D}_x(t)^2,
\end{equation}
for any $t\geq 0$. We then conclude from \eqref{E-diameter-Dx-derivative} by a clear continuity argument.
\end{proof}

In addition, we obtain the following information about the center of mass, which will be useful in our contractivity result.

\begin{lem}[Control of the center of mass]\label{L-control-center-mass}
Let $\mu_0\in \mathcal{P}(\mathbb{R}^{2d})$ be any initial datum and consider the associated weak measure-valued solution $\boldsymbol{\mu}$ to \eqref{kurakk-again}. 
\begin{enumerate}[label=(\roman*)]
\item {\it (Galilean invariance)} For any $\omega_0\in \mathbb{R}^d$ let us consider the push-forward
$$\tilde{\mu}_t:=(x+\omega_0\,t,\omega+\omega_0)_{\#}\mu_t,\quad t\geq 0,$$
Then, $\tilde{\boldsymbol{\mu}}$ also solves \eqref{kurakk-again} with new initial datum $\tilde{\mu}_0=(x,\omega+\omega_0)_{\#}\mu_0$.
\item {\it (Speed of the center of mass)} The following relation is fulfilled
$$\frac{d}{dt}\int_{\mathbb{R}^{2d}}x\,d\mu_t(x,\omega)=\int_{\mathbb{R}^{2d}}\omega\,d\mu_0(x,\omega),$$
for any $t\geq 0$.
\end{enumerate}
\end{lem}

\begin{proof}
Consider any $\varphi\in C^\infty_c([0,\infty)\times \mathbb{R}^d)$ and define $\bar\varphi(t,x,\omega)=\varphi(t,x+\omega_0 t,\omega+\omega_0)$. We can use $\bar\varphi$ as a new test function in the weak formulation for $\boldsymbol{\mu}$ according to Definition \ref{D-distributional-solution-kurakk}. Specifically, note that the integrand in the left hand side reads
$$\partial_t\bar\varphi(t,x,\omega)+\nabla_x\bar\varphi(t,x,\omega)\cdot \boldsymbol{u}[\mu_t](x,\omega)=\left(\partial_t\varphi+\nabla_x\varphi\cdot \boldsymbol{u}[\tilde{\mu}_t]\right)(x+\omega_0t,\omega+\omega_0).$$
Hence, the weak formulation for $\boldsymbol{\mu}$ takes the form
$$\int_0^\infty\int_{\mathbb{R}^{2d}}\left(\partial_t\varphi+\nabla_x\varphi\cdot \boldsymbol{u}[\tilde{\mu}_t]\right)(x+\omega_0t,\omega+\omega_0)\,d\mu_t(x,\omega)\,dt=-\int_{\mathbb{R}^{2d}}\varphi(0,x,\omega+\omega_0)\,d\mu_0(x,\omega).$$
By definition of $\tilde{\boldsymbol{\mu}}$ as pushfoward of $\boldsymbol{\mu}$ and by arbitrariness of $\varphi$, we infer that $\tilde{\boldsymbol{\mu}}$ is again a weak solution of \eqref{kurakk-again}. The proof of $(ii)$ follows a similar argument taking test functions $\varphi(t,x,\omega)=\eta(t)\,x$ with $\eta\in C^\infty_c([0,\infty)])$ (a rigorous argument requires a standard cut-off for $x$). In this case, the integrand in the left hand side of the weak formulation of $\boldsymbol{\mu}$ takes the form
$$\partial_t\varphi(t,x,\omega)+\nabla_x\varphi(t,x,\omega)\cdot \boldsymbol{u}[\mu_t](x,\omega)=\dot{\eta}(t)\,x+\eta(t)\,\boldsymbol{u}[\mu_t](x,\omega).$$
Hence the weak formulation of $\boldsymbol{\mu}$ reads
$$\int_0^\infty\dot{\eta}(t)\int_{\mathbb{R}^{2d}}x\,d\mu_t(x,\omega)\,dt+\int_0^\infty\eta(t)\int_{\mathbb{R}^{2d}}\boldsymbol{u}[\mu_t](x,\omega)\,d\mu_t(x,\omega)\,dt=-\eta(0)\int_{\mathbb{R}^{2d}}x\,d\mu_0(x,\omega).$$
Therefore, we obtain $(ii)$ in weak form thanks to the antisymmetry of $\nabla W$, which implies
$$\int_{\mathbb{R}^{2d}}\boldsymbol{u}[\mu_t](x,\omega)\,d\mu_t(x,\omega)=\int_{\mathbb{R}^{2d}}\omega\,d\mu_t(x,\omega)=\int_{\mathbb{R}^{2d}}\omega\,d\mu_0(x,\omega).$$
\end{proof}

By items $(i)$ and $(ii)$ above, without loss of generality we can always assume (and we shall often do) that the distribution with respect to $\omega$ is centered at the origin for simplicity, {\it i.e.},
\begin{equation}\label{E-centered-nu}
\omega_c:=\int_{\mathbb{R}^d}\omega\,d\nu(\omega)=\int_{\mathbb{R}^{2d}}\omega\,d\mu_0(x,\omega)=0.
\end{equation}
Otherwise, we can always perform an appropriate Galilean transformation according to $(i)$ with speed given by $\omega_c$ to guarantee such a condition. We are now ready to prove the uniform-in-time contractivity of $W_{2,\hat{\nu}}$ along solutions with the same center of mass.

\begin{lem}[Uniform contractivity in $W_{2,\hat{\nu}}$]\label{L-contractivity}
Consider any compactly supported initial data $\mu^1_0,\mu^2_0\in \mathcal{P}(\mathbb{R}^{2d})$ with the same center of mass, {\it i.e.},
\begin{equation}\label{E-center-mass}
\int_{\mathbb{R}^{2d}}(x,\omega)\,d\mu^1_0(x,\omega)=\int_{\mathbb{R}^{2d}}(x,\omega)\,d\mu^2_0(x,\omega).
\end{equation}
Define the marginals $\nu^1:=\pi_{\omega\#}\mu^1_0$ and $\nu^2:=\pi_{\omega\#}\mu^2_0$, set any $\hat{\nu}\in \Gamma_o(\nu_1,\nu_2)$ and let $\boldsymbol{\mu}^1$ and $\boldsymbol{\mu}^2$ be the weak measure-valued solutions to \eqref{kurakk-again} issued at $\mu^1_0$ and $\mu^2_0$. Then, we have
\begin{multline*}
W_{2,\hat{\nu}}(\mathcal{L}_{\hat{\nu},1}(\mu_t^1),\mathcal{L}_{\hat{\nu},2}(\mu_t^2))\\
\leq e^{-2K\phi(D_0)t} \,W_{2,\hat{\nu}}(\mathcal{L}_{\hat{\nu},1}(\mu_0^1),\mathcal{L}_{\hat{\nu},2}(\mu_0^2))+\frac{1}{2K\phi(D_0)}\int_{\mathbb{R}^{2d}}\vert \omega-\omega'\vert^2\,d\hat{\nu}(\omega,\omega'),
\end{multline*}
for each $t\geq 0$, where $D_0\in \mathbb{R}_+$ is given by
\begin{equation}\label{E-diameter-Dx-uniform-bound}
D_0:=\max\left\{\mathcal{D}_x^1(0),\mathcal{D}_x^2(0),\left(\frac{\mathcal{D}_\omega^1}{K}\right)^{\frac{1}{1-\alpha}},\left(\frac{\mathcal{D}_\omega^2}{K}\right)^{\frac{1}{1-\alpha}}\right\},
\end{equation}
and $\mathcal{D}_x^i(t)$ and $\mathcal{D}_\omega^i$ with $i=1,2$ are the diameters of the support of $\mu^i$ with respect to the variable $x$ and $\omega$ according to \eqref{E-partial-diameters}.
\end{lem}

\begin{proof}
For simplicity of notation we shall denote the liftings
$$\bar\mu_t^1:=\mathcal{L}_{\hat{\nu},1}(\mu_t^1),\quad \bar\mu_t^2:=\mathcal{L}_{\hat{\nu},2}(\mu_t^2),$$
according to Definition \ref{D-liftings}. It is then clear that $\bar{\boldsymbol{\mu}}^1$ and $\bar{\boldsymbol{\mu}}^2$ are weak-measure-valued solutions to the following fibered continuity equations
$$
\partial_t \bar{\mu}_t^1+\divop_x(\bar{\boldsymbol{u}}_t^1\,\bar{\mu}_t^1)=0,\quad 
\partial_t \bar{\mu}_t^2+\divop_x(\bar{\boldsymbol{u}}_t^2\,\bar{\mu}_t^2)=0,
$$
in distributional sense, where $\bar{\boldsymbol{u}}_t^i\in L^2_{\bar{\mu}_t^i}(\mathbb{R}^{3d},\mathbb{R}^d)$ are given by
\begin{align*}
&\bar{\boldsymbol{u}}_t^1(x,\omega,\omega'):=\omega-K\int_{\mathbb{R}^{2d}}\nabla W(x-x')\,d\mu_t^1(x',\omega')=:\omega-K\boldsymbol{\xi}_t^1(x),\\
&\bar{\boldsymbol{u}}_t^2(x,\omega,\omega'):=\omega'-K\int_{\mathbb{R}^{2d}}\nabla W(x-x')\,d\mu_t^2(x',\omega')=:\omega'-K\boldsymbol{\xi}_t^2(x).
\end{align*}
Using the analogous version of Theorem \ref{equiv} on the fibered space $(\mathcal{P}_{2,\hat{\nu}}(\mathbb{R}^{3d}),W_{2,\hat{\nu}})$ in Definition \ref{D-fibered-Wasserstein-double} we obtain $\bar{\boldsymbol{\mu}}^1,\bar{\boldsymbol{\mu}}^2\in AC^2_{loc}(0,+\infty;\mathcal{P}_{2,\hat{\nu}}(\mathbb{R}^{3d}))$. Our next goal is then to apply the differentiability Theorem \ref{T-dif}. To such an end, for almost every $t\geq 0$, consider $\gamma_t\in \Gamma_{o,\hat{\nu}}(\bar{\mu}^1_t,\bar{\mu}^2_t)$ an optimal $\hat{\nu}$-admissible transference plan, that is, 
$$\gamma_t(x,x',\omega,\omega')=\gamma_t^{\omega,\omega'}(x,x')\otimes \hat{\nu} (\omega,\omega'),$$
where $\gamma_t^{\omega,\omega'}\in \Gamma_o(\mu^{1,\omega}_t,\mu^{2,\omega'}_t)$ are optimal transference plan for $\hat{\nu}$-a.e. $(\omega,\omega')\in \mathbb{R}^{2d}$. Then,
\begin{equation}\label{E1}
\frac{d}{dt}\frac{1}{2}W_{2,\hat{\nu}}^2(\bar{\mu}^1_t,\bar{\mu}^2_t)=\int_{\mathbb{R}^{4d}}(x-x')\cdot(\bar{\boldsymbol{u}}^1_t(x,\omega,\omega')-\bar{\boldsymbol{u}}^2_t(x',\omega,\omega'))\,d\gamma_t(x,x',\omega,\omega')= I_1(t)+I_2(t),
\end{equation}
for {\it a.e.} $t\geq 0$, where each term reads
\begin{align*}
I_1(t)&:=\int_{\mathbb{R}^{4d}}(x-x')\cdot(\omega-\omega')\,d\gamma_t(x,x',\omega,\omega'),\\
I_2(t)&:=-K\iint_{\mathbb{R}^{2d}\times\mathbb{R}^{2d}}(x-x')\cdot (\boldsymbol{\xi}_t^1(x)-\boldsymbol{\xi}_t^2(x'))\,d\gamma_t^{\omega,\omega'}(x,x')\,d\bar\nu(\omega,\omega').
\end{align*}
On the one hand, Cauchy--Schwarz inequality leads to
\begin{equation}\label{E2}
I_1\leq \left(\int_{\mathbb{R}^{2d}}\vert \omega-\omega'\vert^2\,d\hat{\nu}(\omega,\omega')\right)^{1/2}W_{2,\hat{\nu}}(\bar{\mu}_t^1,\bar{\mu}_t^2),
\end{equation}
On the other hand, notice that we can restate
\begin{align*}
\boldsymbol{\xi}_t^1(x_1)&=\int_{\mathbb{R}^{4d}}\nabla W(x_1-x_2)\,d\gamma_t^{\omega_2,\omega_2'}(x_2,x_2')\,d\hat{\nu}(\omega_2,\omega_2'),\\
\boldsymbol{\xi}_t^2(x_1')&=\int_{\mathbb{R}^{4d}}\nabla W(x_1'-x_2')\,d\gamma_t^{\omega_2,\omega_2'}(x_2,x_2')\,d\hat{\nu}(\omega_2,\omega_2').
\end{align*}
Consequently, using variables $(x_1,x_1',\omega_1,\omega_1')$ instead of $(x,x',\omega,\omega')$ in the above integral for $I_2(t)$ and plugging the above expression for $\boldsymbol{\xi}^1_t(x_1)-\boldsymbol{\xi}^2_t(x_1')$ we obtain
\begin{align*}
I_2&=-K\iint_{\mathbb{R}^{4d}\times\mathbb{R}^{4d}}(x_1-x_1')\cdot(\nabla W(x_1-x_2)-\nabla W(x_1'-x_2'))\\
&\hspace{8.5cm}\times d\gamma_t(x_1,x_1',\omega_1,\omega_1')\,d\gamma_t(x_2,x_2',\omega_2,\omega_2'),\\
&=-\frac{K}{2}\iint_{\mathbb{R}^{4d}\times\mathbb{R}^{4d}}((x_1-x_1')-(x_2-x_2'))\cdot(\nabla W(x_1-x_2)-\nabla W(x_1'-x_2'))\\
&\hspace{8.5cm}\times d\gamma_t(x_1,x_1',\omega_1,\omega_1')\,d\gamma_t(x_2,x_2',\omega_2,\omega_2'),
\end{align*}
where in the last part we have used a standard symmetrization of the integral which uses that $\nabla W$ is an odd function. Then, the convexity properties in Lemma \ref{C-convexity} imply that
$$I_2\leq -K\iint_{\mathbb{R}^{4d}\times\mathbb{R}^{4d}}\Lambda_1(x_1-x_2,x_1'-x_2',\alpha)\vert (x_1-x_2)-(x_1'-x_2')\vert^2\,d\gamma_t(x_1,x_1',\omega_1,\omega_1')\,d\gamma_t(x_2,x_2',\omega_2,\omega_2').$$
Consider any $(x_1,x_1',\omega_1,\omega_1')$ and $(x_2,x_2',\omega_2,\omega_2')$ in $\supp \gamma_t$. Then, by Lemma \ref{L-control-diameter} we obtain
\begin{align*}
\Lambda_1(x_1-x_2,x_1'-x_2',\alpha)&=\int_0^1 \phi(\vert (1-\tau)(x_1-x_2)+\tau (x_1'-x_2')\vert)\,d\tau\\
&\geq \int_0^1\phi((1-\tau) \mathcal{D}_x^1(t)+\tau \mathcal{D}_x^2(t))\,d\tau\geq  \phi(D_0),
\end{align*}
with $D_0$ given by \eqref{E-diameter-Dx-uniform-bound}. Then, we achieve the inequality
\begin{align}\label{E3}
\begin{aligned}
I_2&\leq -K\phi(D_0)\iint_{\mathbb{R}^{4d}\times\mathbb{R}^{4d}}\vert (x_1-x_2)-(x_1'-x_2')\vert^2\,d\gamma_t(x_1,x_1',\omega_1,\omega_1')\,d\gamma_t(x_2,x_2',\omega_2,\omega_2')\\
&=-2K\phi(D_0)W_{2,\hat{\nu}}^2(\bar{\mu}^1_t,\bar{\mu}^2_t)+ 2K\phi(D_0)\left\vert\int_{\mathbb{R}^{4d}}(x-x')\,d\gamma_t(x,x',\omega,\omega')\right\vert^2.
\end{aligned}
\end{align}
By \eqref{E-center-mass} and Lemma \ref{L-control-center-mass}, the second term in the right-hand side of \eqref{E3} vanishes because we have
$$
\int_{\mathbb{R}^{4d}}x\,d\gamma_t(x,x',\omega,\omega')=\int_{\mathbb{R}^{2d}}x\,d\mu_t^1(x,\omega)=\int_{\mathbb{R}^{2d}}x\,d\mu_t^2(x,\omega)=\int_{\mathbb{R}^{4d}}x'\,d\gamma_t(x,x',\omega,\omega'),
$$
for all $t\geq 0$. Altogether, the inequalities \eqref{E1}, \eqref{E2} and \eqref{E3} imply that
\begin{equation}\label{E4}
\frac{d}{dt}\frac{1}{2}W_{2,\hat{\nu}}^2(\bar{\mu}^1_t,\bar{\mu}^2_t)\leq -2K\phi(D_0)W_{2,\hat{\nu}}^2(\bar{\mu}^1_t,\bar{\mu}^2_t)+\left(\int_{\mathbb{R}^{2d}}\vert \omega-\omega'\vert^2\,d\hat{\nu}(\omega,\omega')\right)^{1/2}W_{2,\hat{\nu}}(\bar{\mu}_t^1,\bar{\mu_t}^2),
\end{equation}
for almost every $t\geq 0$. Then, we conclude by Gr\"{o}nwall's lemma.
\end{proof}

The above result can be regarded as the natural counterpart of \cite[Theorem 3.1]{CZ-21} to the kinetic equation \eqref{kurakk-again} in dimension $d\in \mathbb{N}$. We remark that the hypothesis \eqref{E-center-mass} has been crucially used in the above proof to kill the crossed term in \eqref{E3}. Otherwise, the distance between the centers of mass would grow linearly with time, thus leading to a linearly-in-time growing remainder that breaks the above uniform-in-time contractivity estimate.

\subsection{Convergence to equilibrium and uniform-in-time mean-field limit}

Using the uniform-in-time contractivity in Lemma \ref{L-contractivity} along with Remark \ref{R-fibered-Wasserstein-double-identical} (which relates $W_{2,\hat{\nu}}$ with $W_{2,\nu}$) we obtain the following results.

\begin{theo}[Theorem B: Uniform contractivity in $W_{2,\nu}$]\label{T-contractivity-W2nu}
Consider any $\nu\in \mathcal{P}(\mathbb{R}^d)$ and any compactly supported initial data $\mu_0^1,\mu_0^2\in \mathcal{P}_{2,\nu}(\mathbb{R}^{2d})$ with the same center of mass, {\it i.e.},
\begin{equation}\label{E-T-contractivity-W2nu-same-center-mass}
\int_{\mathbb{R}^{2d}}x\,d\mu_0^1(x,\omega)=\int_{\mathbb{R}^{2d}} x\,d\mu_0^2(x,\omega).
\end{equation}
Let $\boldsymbol{\mu}^1$ and $\boldsymbol{\mu}^2$ be the weak measure-valued solutions to \eqref{kurakk-again} issued at $\mu_0^1$ and $\mu_0^2$. Then,
$$W_{2,\nu}(\mu_t^1,\mu_t^2)\leq e^{-2K\phi(D_0)t} \,W_{2,\nu}(\mu_0^1,\mu_0^2),$$
for each $t\geq 0$, where $D_0\in \mathbb{R}_+$ is given by \eqref{E-diameter-Dx-uniform-bound}.
\end{theo}

\begin{cor}[Theorem B: Convergence to equilibrium]\label{C-contractivity-W2nu}
Consider any $\nu\in \mathcal{P}(\mathbb{R}^d)$ verifying \eqref{E-centered-nu}, any compactly supported initial datum $\mu_0\in \mathcal{P}_{2,\nu}(\mathbb{R}^{2d})$ and let $\boldsymbol{\mu}$ be the weak measure-valued solution of \eqref{kurakk-again} issued at $\mu_0$. Then, there exists a unique compactly supported equilibrium $\mu_\infty\in \mathcal{P}_{2,\nu}(\mathbb{R}^{2d})$ of \eqref{kurakk-again} such that
\begin{equation}\label{E-C-contractivity-W2nu-same-center-mass}
\int_{\mathbb{R}^{2d}}x\,d\mu_0(x,\omega)=\int_{\mathbb{R}^{2d}}x\,d\mu_\infty(x,\omega).
\end{equation}
In addition, we obtain
$$W_{2,\nu}(\mu_t,\mu_\infty)\leq e^{-2K\phi(D_0)t} \,W_{2,\nu}(\mu_0,\mu_\infty),$$
for any $t\geq 0$, where $D_0$ is the uniform bound of the spacial diameter in Lemma \ref{L-control-diameter}.
\end{cor}

\begin{proof}
Note that the uniqueness of such an equilibrium $\mu_\infty$ readily follows from Theorem \ref{T-contractivity-W2nu}. We then focus on proving its existence and the convergence of $\mu_t$ to $\mu_\infty$ as $t\rightarrow \infty$ with quantitative rates. To such an end, we shall exploit Theorem \ref{T-contractivity-W2nu} above again. First, let us admit that the following claim holds true 
\begin{equation}\label{E-C-contractivity-W2nu-claim}
\sup_{t\geq 0}W_{2,\nu}(\mu_0,\mu_t)\leq C<\infty,
\end{equation}
for some $C>0$. We will prove the claim \eqref{E-C-contractivity-W2nu-claim} below. Fix any arbitrary $T>0$ and $t_1,t_2\geq T$, and consider the couple of solutions $\boldsymbol{\mu}^1$ and $\boldsymbol{\mu}^2$ of \eqref{kurakk-again} given by
$$\mu_t^1:=\mu_{t+t_1-T},\quad \mu_t^2:=\mu_{t+t_2-T},\quad t\geq 0.$$
Recall that by our assumption \eqref{E-centered-nu} and Lemma \ref{L-control-center-mass} $\mu_0^1$ and $\mu_0^2$ verify the condition \eqref{E-T-contractivity-W2nu-same-center-mass}. Then, Theorem \ref{T-contractivity-W2nu} can be applied to $\boldsymbol{\mu}^1$ and $\boldsymbol{\mu}^2$ and we have
$$W_{2,\nu}(\mu_{t_1},\mu_{t_2})= W_{2,\nu}(\mu_T^1,\mu_T^2)\leq e^{-2K\phi(D_0)T} \,W_{2,\nu}(\mu_0^1,\mu_0^2)\leq 2C e^{-2K\phi(D_0)T} .$$
In the last step, we have used the triangle inequality and our claim \eqref{E-C-contractivity-W2nu-claim}. Taking $T$ large enough, the above implies that $(\mu_t)_{t\geq 0}$ is a Cauchy net over the metric space $(\mathcal{P}_{2,\nu}(\mathbb{R}^{2d}),W_{2,\nu})$. By completeness ({\it cf.}, Proposition \ref{P-dnu-polish}), there exists $\mu_\infty \in\mathcal{P}_{2,\nu}(\mathbb{R}^{2d})$ such that
$$\lim_{t\rightarrow \infty}W_{2,\nu}(\mu_t,\mu_\infty)=0.$$
On the one hand, note that $\mu_\infty$ must be an equilibrium of \eqref{kurakk-again}, and it is easy to infer that \eqref{E-C-contractivity-W2nu-same-center-mass} is verified. Finally, since $\mu_t\rightarrow \mu$ narrowly in $\mathcal{P}(\mathbb{R}^{2d})$ (see Remark \ref{R-characterization-convergence-W2nu}) and any point of the support of $\mu_\infty$ can be approximated as a limit of points in the supports of $\mu_t$ ({\it cf.} \cite[Proposition 5.1.8]{AGS-08}), then $\mu_\infty$ must also be compactly supported and it has spacial diameter bounded by $D_0$. Thus, applying again Theorem \ref{T-contractivity-W2nu} to $\boldsymbol{\mu}$ and the stationary solution $\mu_\infty$ we obtain the exponential relation.

To conclude our proof, we just need to prove our claim \eqref{E-C-contractivity-W2nu-claim}. Denote the center of mass by
$$x_c:=\int_{\mathbb{R}^{2d}}x\,d\mu_0(x,\omega)=\int_{\mathbb{R}^{2d}}x\,d\mu_t(x,\omega),$$
for each $t\geq 0$. Then, we arrive at
\begin{align*}
W_{2,\nu}(\mu_0,\mu_t)&\leq W_{2,\nu}(\mu_0,\delta_{x_c}\otimes \nu)+W_{2,\nu}(\mu_t,\delta_{x_c}\otimes \nu)\\
&= \left(\int_{\mathbb{R}^{2d}}\vert x-x_c\vert^2\,d\mu_0(x,\omega)\right)^{1/2}+\left(\int_{\mathbb{R}^{2d}}\vert x-x_c\vert^2\,d\mu_t(x,\omega)\right)^{1/2}\\
&\leq  \int_{\mathbb{R}^{2d}}\left(\int_{\mathbb{R}^{2d}}\vert x-x'\vert^2\,d\mu_0(x',\omega')\right)^{1/2}\,d\mu_0(x,\omega)\\
&\hspace{4cm}+\int_{\mathbb{R}^{2d}}\left(\int_{\mathbb{R}^{2d}}\vert x-x'\vert^2\,d\mu_t(x',\omega')\right)^{1/2}\,d\mu_t(x,\omega),
\end{align*}
for any $t\geq 0$, where in the first line we have used the triangle inequality and the last line we have applied Minkowski's integral inequality. Using Lemma \ref{L-control-diameter} ends the claim with $C=2D_0$.
\end{proof}

Similarly, using the uniform-in-time contractivity in Lemma \ref{L-contractivity} and taking infimum over $\hat{\nu}$ to recover the adapted Wasserstein distance in Definition \ref{D-adapted-Wasserstein-distance} imply the following theorem.

\begin{theo}[Theorem B: Uniform stability in $AW_2$]\label{T-uniform-stability-AW2}
Consider any compactly supported initial data $\mu_0^1,\mu_0^2\in \mathcal{P}_{2,\nu}(\mathbb{R}^{2d})$ with the same center of mass ({\it i.e.}, satisfying \eqref{E-center-mass}), and let $\boldsymbol{\mu}^1$ and $\boldsymbol{\mu}^2$ be the weak measure-valued solutions to \eqref{kurakk-again} issued at $\mu_0^1$ and $\mu_0^2$. Then,
$$AW_2(\mu_t^1,\mu_t^2)\leq \left(1+\frac{1}{2K\phi(D_0)}\right)\,AW_2(\mu_0^1,\mu_0^2),$$
for any $t\geq 0$, where $D_0\in \mathbb{R}_+$ is given by \eqref{E-diameter-Dx-uniform-bound}.
\end{theo}

\begin{cor}[Theorem B: Uniform mean-field limit]\label{C-uniform-stability-AW2}
Consider any $\nu\in \mathcal{P}(\mathbb{R}^d)$ verifying \eqref{E-centered-nu}, any compactly supported initial datum $\mu_0\in \mathcal{P}_{2,\nu}(\mathbb{R}^{2d})$ and let $\boldsymbol{\mu}$ be the weak measure-valued solution of \eqref{kurakk-again} issued at $\mu_0$. Take any sequence of empirical measures $\boldsymbol{\mu}^N$, {\it i.e.},
\begin{equation}\label{E-empirical-measures}
\mu_t^N:=\frac{1}{N}\sum_{i=1}^N\delta_{x_i^N(t)}(x)\otimes \delta_{\omega_i^N}(\omega),
\end{equation}
with configurations $(x_1^N(t),\omega_1^N),\ldots,(x_N^N(t),\omega_N^N)$ solving the associated particle system
\begin{equation}\label{kurap}
\dot{x}_i^N=\omega_i^N-K\nabla W(x_i^N-x_j^N),\qquad i=1,\ldots,N,
\end{equation}
and assume that the following conditions hold
\begin{align}
&\lim_{N\rightarrow\infty}AW_2(\mu_0^N,\mu_0)=0,\label{E-convergence-initial-empirical}\\
&\frac{1}{N}\sum_{i=1}^N (x_{i,0}^N,\omega_i^N)=\int_{\mathbb{R}^{2d}}(x,\omega)\,d\mu_0(x,\omega).\label{E-center-mass-empirical}
\end{align}
Then, we obtain the uniform-in-time mean-field limit
$$\lim_{N\rightarrow\infty} \sup_{t\geq 0} AW_2(\mu_t^N,\mu_t)=0.$$
\end{cor}

Since $(x_1^N(t),\omega_1^N),\ldots,(x_N^N(t),\omega_N^N)$ solve \eqref{kurap}, then $\boldsymbol{\mu}^N$ solves \eqref{kurakk-again} in the sense of distributions. Then, the result follows from the stability Theorem \ref{T-uniform-stability-AW2} applied to the pair of solutions $\mu$ and $\boldsymbol{\mu}^N$. Corollary \ref{C-uniform-stability-AW2} quantifies a uniform-in-time mean-field limit in the adapted Wasserstein distance $AW_2$. By Proposition \ref{P-AW2-stronger-than-W2} we also have
$$\lim_{N\rightarrow\infty} \sup_{t\geq 0} W_2(\mu_t^N,\mu_t)=0.$$

\begin{rem}\label{R-uniform-stability-AW2}
The following comments are in order:
\begin{enumerate}[label=(\roman*)]
\item {\bf (Assumption \eqref{E-convergence-initial-empirical})} Note that \eqref{E-convergence-initial-empirical} is stronger than the usual approximation of $\mu_0$ by empirical measures $\mu_0^N$ with respect to the quadratic Wasserstein distance $W_2$. A typical method for the latter relies on the law of large numbers in \cite{V-58}. It states that for any sequence of i.i.d. random variables $\{(X_i,\Omega_i)\}_{i\in \mathbb{N}}$ with law $\mu_0$ we have
$$\frac{1}{N}\sum_{i=1}^N\delta_{X_i}(x)\otimes\delta_{\Omega_i}(\omega)\rightarrow \mu_0\quad \mbox{in}\quad W_2,$$
almost surely as $N\rightarrow\infty$. However, in \cite{PP-16} it was proven that the $W_2$-convergence cannot generally be improved into $AW_2$-convergence in such generality. Recently, in \cite{BBBW-20-arxiv} (see also \cite[Theorem 4.8]{BL-18-arxiv}) it was  proven that for compactly supported $\mu_0$ (as in the setting of our Corollary \ref{C-uniform-stability-AW2}) we have
$$\frac{1}{N}\sum_{i=1}^N\delta_{P_x^N(X_i)}\otimes\delta_{P_\omega^N(\Omega_i)}(\omega)\rightarrow \mu_0\quad \mbox{in}\quad AW_2,$$
almost surely as $N\rightarrow\infty$, for some finite-rank mappings $P_x^N:\supp_x\mu_0\rightarrow \supp_x\mu_0$ and $P_\omega^N:\supp_\omega\mu_0\rightarrow\supp_\omega\mu_0$. This guarantees the existence of $\mu_0^N$ verifying \eqref{E-convergence-initial-empirical}.
\item {\bf (Assumption \eqref{E-center-mass-empirical})} Consider $\mu_0^N$ supported over $(x_{1,0}^N,\omega_1^N),\ldots,(x_{N,0}^N,\omega_N^N)$ as above so that it satisfies the assumption \eqref{E-convergence-initial-empirical}. Let us set the change of variables
$$\tilde{x}_{i,0}^N:=x_{i,0}^N-\frac{1}{N}\sum_{j=1}^N x_{j,0}^N+\int_{\mathbb{R}^{2d}}x\,d\mu_0(x,\omega),\quad \tilde{\omega}_i^N:=\omega_i^N-\frac{1}{N}\sum_{j=1}^N\omega_j^N,$$
for $i=1,\ldots,N$. Then, the modified empirical measures $\tilde{\mu}_0^N:=\frac{1}{N}\sum_{i=1}^N\delta_{\tilde{x}_{i,0}^N}(x)\otimes \delta_{\tilde{\omega}_{i,0}^N}(\omega)$ verify \eqref{E-center-mass-empirical} by construction. In addition, we have
$$AW_2(\tilde{\mu}_0^N,\mu_0)\leq AW_2(\tilde{\mu}_0^N,\mu_0^N)+AW_2(\mu_0^N,\mu_0).$$
By Propositions \ref{P-AW2-translations} and \ref{P-AW2-stronger-than-W2} we obtain
\begin{align*}
AW_2^2(\tilde{\mu}_0^N,\mu_0^N)&= \left\vert \frac{1}{N}\sum_{j=1}^N x_{j,0}^N-\int_{\mathbb{R}^{2d}}x\,d\mu_0(x,\omega)\right\vert^2+\left\vert\frac{1}{N}\sum_{j=1}^N \omega_j^N\right\vert^2\\
&\leq W_2^2(\mu_0^N,\mu_0)\leq AW_2(\mu_0^N,\mu_0).
\end{align*}
Altogether implies $\lim_{N\rightarrow\infty}AW_2(\tilde{\mu}_0^N,\mu_0)=0$, so that $\tilde{\mu}_0^N$ verify both \eqref{E-convergence-initial-empirical} and \eqref{E-center-mass-empirical}.
\end{enumerate}
\end{rem}


\appendix

\section{The metric-valued Lebesgue $L^2$ space}\label{Appendix-metric-valued-L2}

In this appendix, we recall the definition of the Lebesgue spaces with values in a metric. To the best of our knowledge, these spaces have not been often treated in the literature and were first introduced in \cite{KS-93}. See also \cite{GT-21} and references for more recent approaches. The metric-valued Lebesgue spaces become the metric-valued analogues of the so called Lebesgue-Bochner spaces of vector-valued functions, see \cite{DU-77}. Since its construction relies on subtle measurability properties and we shall use them along this paper, we briefly recall them here along with some of their main properties.

\begin{defi}\label{D-metric-valued-L2}
Let $(\Omega,d_\Omega,\nu)$ be a metric measure space where $\nu\in\mathcal{P}(\Omega)$ and consider $(\mathbb{X},d_\mathbb{X})$ any metric space. We define
\begin{equation}\label{E-metric-valued-L2-space}
L^2_\nu(\Omega,(\mathbb{X},d_\mathbb{X})):=\left\{f:\Omega\longrightarrow \mathbb{X}:\begin{array}{c} f\mbox{ is Borel measurable, essentially separably-valued},\\
\displaystyle\mbox{and }\int_{\Omega}d_\mathbb{X}^2(f(\omega),x_0)\,d\nu(\omega)<\infty\end{array}\right\}.
\end{equation}
for some $x_0\in \mathbb{X}$. Moreover, we also define the map $d_{L^2_\nu(\Omega,(\mathbb{X},d_\mathbb{X}))}$ by
\begin{equation}\label{E-metric-valued-L2-distance}
d_{L^2_\nu(\Omega,(\mathbb{X},d_\mathbb{X}))}(f,\widetilde{f}):=\left(\int_{\Omega}d_\mathbb{X}^2(f(\omega),\widetilde{f}(\omega))\,d\nu(\omega)\right)^{1/2},
\end{equation}
for every $f,\widetilde{f}\in L^2_\nu(\Omega,(\mathbb{X},d_\mathbb{X}))$.
\end{defi}

For simplicity of notation, we will denote $L^2(\Omega,\mathbb{X})$ and $d_{L^2(\Omega,\mathbb{X})}$ when both the probability measure $\nu$ and the distance $d_\mathbb{X}$ are clear.

\begin{rem}\label{R-metric-valued-L2}
The following comments are in order:
\begin{enumerate}[label=(\roman*)]
\item For Borel measurable maps $f,\widetilde{f}:\Omega\longrightarrow \mathbb{X}$, the function $\omega\in \Omega\longmapsto d_\mathbb{X}^2(f(\omega),\widetilde{f}(\omega))$ is Borel measurable too. Consequently, the integrals in \eqref{E-metric-valued-L2-space} and \eqref{E-metric-valued-L2-distance} are well defined.
\item The definition \eqref{E-metric-valued-L2-space} is independent on the reference point $x_0\in \mathbb{X}$. Indeed, by triangle inequality, once it holds for some $x_0$, it holds for any other $x_0$.
\item By $f:\Omega\longrightarrow \mathbb{X}$ essentially separably-valued we mean a map such that there exists a $\nu$-negligible set $\mathcal{N}\subseteq \Omega$ with $f(\Omega\setminus \mathcal{N})$ separable in $(\mathbb{X},d_\mathbb{X})$. 
\item Like for the classical Lebesgue spaces, we will identify elements in $L^2(\Omega,\mathbb{X})$ that agree $\nu$-a.e.
\end{enumerate}
\end{rem}

\begin{pro}\label{P-metric-valued-L2-properties}
Let $(L^2(\Omega,\mathbb{X}),d_{L^2(\Omega,\mathbb{X})})$ be the space in Definition \ref{D-metric-valued-L2}. Then, the following properties hold true:
\begin{enumerate}[label=(\roman*)]
\item $(L^2(\Omega,\mathbb{X}),d_{L^2(\Omega,\mathbb{X})})$ is a metric space.
\item If $(\mathbb{X},d_\mathbb{X})$ is complete then so is $(L^2(\Omega,\mathbb{X}),d_{L^2(\Omega,\mathbb{X})})$.
\item The family of simple functions is dense on $(L^2(\Omega,\mathbb{X}),d_{L^2(\Omega,\mathbb{X})})$.
\item If $(\Omega,d_\Omega)=(\mathbb{R}^d,\vert \cdot\vert)$ and $(\mathbb{X},d_\mathbb{X})$ is separable, then so is $(L^2(\mathbb{R}^d,\mathbb{X}),d_{L^2(\mathbb{R}^d,\mathbb{X})})$.
\end{enumerate}
\end{pro}

\begin{proof}

$\diamond$ {\sc Step 1}: Metric space.\\
Symmetry and triangle inequality are clear by the corresponding properties of the distance $d_\mathbb{X}$ along with the norm of the scalar Lebesgue space $L^2(\Omega,\mathbb{R})$. In addition, non-degeneracy also follows by the last item in Remark \ref{R-metric-valued-L2}.

\medskip

$\diamond$ {\sc Step 2}: Completeness.\\
Let us assume that $(\mathbb{X},d_\mathbb{X})$ is complete and take any Cauchy sequence $\{f_n\}_{n\in \mathbb{N}}\subseteq L^2(\Omega,\mathbb{X})$. Then, up to a subsequence, 
\begin{equation}\label{E-10}
\left(\int_{\mathbb{R}^d}d_\mathbb{X}^2(f_{\sigma(n+1)}(\omega),f_{\sigma(n)}(\omega))\,d\nu(\omega)\right)^{1/2}\leq \frac{1}{2^n},
\end{equation}
for every $n\in \mathbb{N}$ and some strictly increasing function $\sigma:\mathbb{N}\longrightarrow \mathbb{N}$. Let us define the sequence of non-decreasing functions
$$g_k(\omega):=\sum_{n=1}^k d_\mathbb{X}(f_{\sigma(n+1)}(\omega),f_{\sigma(n)}(\omega)),\quad \omega\in \Omega,$$
for each $k\in \mathbb{N}$ and consider the pointwise limit $g$. Notice that $\Vert g_k\Vert_{L^2(\Omega,\mathbb{R})}\leq 1$ for every $k\in \mathbb{N}$ by virtue of \eqref{E-10}. Consequently, the monotone convergence theorem guarantees that $g\in L^2(\Omega,\mathbb{R})$. By construction, if $k,m\in\mathbb{N}$ with $m\geq k\geq 2$ then
\begin{equation}\label{E-11}
d_\mathbb{X}(f_{\sigma(k)}(\omega),f_{\sigma(m)}(\omega))\leq g(\omega)-g_{k-1}(\omega),
\end{equation}
for all $\omega\in\Omega$. This implies that $\{f_{\sigma(n)}(\omega)\}_{n\in \mathbb{N}}\subseteq \mathbb{X}$ is Cauchy for $\nu$-a.e. $\omega\in \Omega$. By completeness of $(\mathbb{X},d_\mathbb{X})$ there is a limit $f(\omega)\in \mathbb{X}$. Our final goal is to show that $f\in L^2(\Omega,\mathbb{X})$ and $\{f_{\sigma(n)}\}_{n\in \mathbb{N}}$ converges to $f$ in the $d_{L^2(\Omega,\mathbb{X})}$ distance. On the one hand, it is clear that $f$ is Borel measurable and essentially separably-valued as a pointwise limit of a sequence of Borel measurable and essentially separably-valued functions. Taking limits as $m\rightarrow\infty$ in \eqref{E-11} yields
\begin{equation}\label{E-12}
d_\mathbb{X}(f_{\sigma(k)}(\omega),f(\omega))\leq g(\omega)-g_{k-1}(\omega),
\end{equation}
for any $\omega\in \Omega$. Since each of the $f_{\sigma(k)}$ belongs to $L^2(\Omega,\mathbb{X})$, the triangle inequality and \eqref{E-12} show that $f\in L^2(\Omega,\mathbb{X})$ too. Indeed, integrating \eqref{E-12} with respect to $\nu(\omega)$, we also have
$$d_{L^2(\Omega,\mathbb{X})}(f_{\sigma(k)},f)\leq \Vert g-g_{k-1}\Vert_{L^2(\Omega,\mathbb{R})},$$
for every $k\geq 2$ . By dominated convergence theorem we conclude that $\{f_{\sigma(n)}\}_{n\in \mathbb{N}}$ converges to $f$ in $d_{L^2(\Omega,\mathbb{X})}$. Since the initial sequence $\{f_n\}_{n\in \mathbb{N}}$ is Cauchy with respect to such a distance, we indeed infer that the full sequence converges towards $f$.

\medskip

$\diamond$ {\sc Step 3}: Density of simple functions.\\
Let us consider the family of simple functions
\begin{equation}\label{E-simple-functions}
\mathcal{S}:=\left\{\sum_{i=1}^k\chi_{E_i}(\omega)x_i:\, k\in \mathbb{N},\, E_1,\ldots,E_k\subseteq \Omega\mbox{ are exhaustive Borelians and }x_1,\ldots,x_k\in \mathbb{X} \right\}.
\end{equation}
Here exhaustive means that $\Omega$ can be recovered as disjoint union of all the $E_i$. It is clear that $\mathcal{S}\subseteq L^2(\Omega,\mathbb{X})$. Our goal is to show that $\mathcal{S}$ is dense in $L^2(\Omega,\mathbb{X})$ under the $d_{L^2(\Omega,\mathbb{X})}$ distance. The result will follow by adapting to this metric setting the classical Pettis' measurability theorem characterizing strong measurability for Banach-valued maps, see \cite[Theorem 1.2.]{DU-77}. Specifically, let us take any $f\in L^2(\Omega,\mathbb{X})$, that is essentially separably-bounded, and consider a $\nu$-negligible set  $\mathcal{N}\subseteq \Omega$ so that $f(\Omega\setminus \mathcal{N})$ is separable in $\mathbb{X}$. Hence, there is a dense countable subset $\{x_n\}_{n\in \mathbb{N}}\subseteq f(\Omega\setminus \mathcal{N})$ and for each $k\in \mathbb{N}$ we can define the Borelian sets
$$B_{n,k}:=\left\{\omega\in \Omega\setminus \mathcal{N}:\,d_\mathbb{X}(f(\omega),x_n)< \frac{1}{k}\right\},$$
for any $n\in \mathbb{N}$. We define the pairwise disjoint Borelians $E_{n,k}:=B_{n,k}\setminus \cup_{m<n}B_{m,k}$, along with the associated sequence of countably-valued functions
\begin{equation}\label{E-13}
f_k(\omega):=\sum_{n=1}^\infty \chi_{E_{n,k}}(\omega)x_n,\quad \omega\in \Omega\setminus \mathcal{N},
\end{equation}
for each $k\in \mathbb{N}$. By density of $\{x_n\}_{n\in \mathbb{N}}$ in $f(\Omega\setminus \mathcal{N})$ we have
$$d_\mathbb{X}(f(\omega),f_k(\omega))<\frac{1}{k},$$
for every $\omega\in \Omega\setminus \mathcal{N}$ and any $k\in \mathbb{N}$. Let us now truncate the countable range of the functions $f_k$ in the above sequence. To such an end, we define new Borelians $\widetilde{E}_{1,k}^m,\ldots,\widetilde{E}_{m,k}^m$ by the formula
$$\widetilde{E}_{n,k}^m:=\left\{\begin{array}{ll}
E_{n,k}, & \mbox{if }n<m,\\
\cup_{r\geq m} E_{r,k}, & \mbox{if }n=m,
\end{array}\right.$$
which are exhaustive, and the truncated simple functions
$$g_k^m(\omega):=\sum_{n=1}^{m-1} \chi_{\widetilde{E}_{n,k}^m}(\omega)x_k+\chi_{\widetilde{E}_{m,k}^m}(\omega) x_0,$$
for $\omega\in \Omega\setminus\mathcal{N}$ and $k,m\in \mathbb{N}$. Then, by triangle inequality
\begin{equation}\label{E-14}
d_{L^2(\Omega,\mathbb{X})}(g_k^m,f)\leq \left(\frac{1}{k^2}+\int_{\widetilde{E}_{m,k}^m} d_\mathbb{X}^2(f(\omega),x_0)\,d\nu(\omega)\right)^{1/2},
\end{equation}
for any $k\in \mathbb{N}$. We end the proof by recalling that $d_\mathbb{X}(f,x_0)\in L^2(\Omega,\mathbb{R})$ and noticing that, by construction $\lim_{m\rightarrow \infty}\nu(\widetilde{E}_{m,k}^m)\rightarrow 0$ for every $k\in \mathbb{N}$.

\medskip

$\diamond$ {\sc Step 4}: Separability.\\
Now, let us assume that $(\mathbb{X},d_\mathbb{X})$ is separable and take $(\Omega,d_\Omega)=(\mathbb{R}^d,\vert \cdot\vert)$. Set any dense countable subset $\mathcal{D}_\mathbb{X}\subseteq \mathbb{X}$ and define the subfamily $\mathcal{S}_*$ of $\mathcal{S}$ where points $x_1,\ldots,x_n$ in \eqref{E-simple-functions} are restricted to $\mathcal{D}_\mathbb{X}$. The same reasoning as above clearly shows that 
$$\overline{\mathcal{S}_*}^{L^2(\Omega,\mathbb{X})}=L^2(\Omega,\mathbb{X}).$$
To end the proof, let us take a simple function $f\in \mathcal{S}_*$ that is,
$$f(\omega)=\sum_{i=1}^k\chi_{E_i}(\omega)x_i,$$
for $\omega\in \mathbb{R}^d$, where $E_i$ are pairwise disjoint Borelian sets with $\cup_{i=1}^k E_i=\mathbb{R}^d$ and $x_1,\ldots,x_k\in \mathcal{D}_\mathbb{X}$. Our final goal is to approximate by a simple function in some universal countable family. To such an end, notice that $\nu$ is a finite Radon measure and, in particular it is outer regular. Consequently, for any fixed $\varepsilon>0$ there are open sets $E_i\subseteq O_i\subseteq \mathbb{R}^d$ so that 
\begin{equation}\label{E-15}
\nu(O_i\setminus E_i)<\frac{1}{\max_{1\leq i,j\leq k}d_\mathbb{X}^2(x_i,x_j)}\frac{\varepsilon^2}{k},
\end{equation}
for all $i=1,\ldots,k$. Define the following (countable) family of sets
$$\mathcal{Q}:=\left\{\prod_{i=1}^d(a_i,b_i):\,a_i,b_i\in\mathbb{Q}\mbox{ and }a_i<b_i\right\}.$$
Since $(\mathbb{R}^d,\vert \cdot\vert)$ is a separable metric space (thus a Lindel\"{o}f space), a classical argument allows taking countable families of cubes $\{Q_{i,n}\}_{n\in\mathbb{N}}\subseteq \mathcal{Q}$ so that $O_i=\cup_{n=1}^\infty Q_{i,n}$ for every $i=1,\ldots,k$. Then, we can define the following countable family of Borelian sets
$$\begin{array}{ll}
B_{1,1}=Q_{1,1}, & B_{1,n}=Q_{1,n}\setminus \cup_{m<n} Q_{1,m},\\
B_{2,1}=Q_{2,1}\setminus O_1, & B_{2,n}=(Q_{2,n}\setminus \cup_{m<n}Q_{2,m})\setminus O_1,\\
\hspace{0.9cm}\vdots & \hspace{0.9cm}\vdots\\
B_{k,1}=Q_{k,1}\setminus \cup_{i<k}O_i, & B_{k,n}=(Q_{k,n}\setminus \cup_{m<n}Q_{k,m})\setminus \cup_{i<k}O_i,
\end{array}$$
that can be constructed from those in $\mathcal{Q}$ and their complementary sets by taking countably-many finite intersections and countable unions. Then, we can define the simple function
$$g(\omega):=\sum_{i=1}^k\sum_{n=1}^\infty \chi_{B_{i,n}}(\omega)x_i,\quad \omega\in \mathbb{R}^d.$$
Consequently, we obtain
\begin{align*}
d_{L^2(\mathbb{R}^d,\mathbb{X})}^2(f,g)&= \sum_{i=1}^k\sum_{n=1}^\infty\int_{B_{i,n}} d_\mathbb{X}^2(f(\omega),x_i)\,d\nu(\omega)\leq \sum_{i=1}^k \int_{O_i}d_\mathbb{X}^2(f(\omega),g(\omega))\,d\nu(\omega)\\
&=\sum_{i=1}^k\int_{E_i} d_\mathbb{X}^2(f(\omega),x_i)\,d\nu(\omega)+\sum_{i=1}^k\int_{O_i\setminus E_i} d_\mathbb{X}^2(f(\omega),x_i)\,d\nu(\omega).
\end{align*}
In the last line, the first term obviously vanishes since $\left.f\right\vert_{E_i}\equiv x_i$. Then, using \eqref{E-15} in the second term leads to $d_{L^2(\mathbb{R}^d,\mathbb{X})}(f,g)\leq \varepsilon$ and this ends the proof.
\end{proof}

\section{Extended fibered subdifferential and local metric slope}\label{Appendix-slope-subdifferential}

The main goal of this section is to prove the relation \eqref{E-slope-minimal-subdifferential} in Proposition  \ref{P-slope-minimal-subdifferential} linking the local slope with the Fr\'echet subdifferential of functionals over the fibered Wasserstein space $(\mathcal{P}_{2,\nu}(\mathbb{R}^{2d}),W_{2,\nu})$. As mentioned in Section \ref{sec:fibered-wasserstein}, fibered subdifferentials $\partial_{W_{2,\nu}}\mathcal{E}[\mu]$, consisting of vectors $\boldsymbol{u}\in L^2_\mu(\mathbb{R}^{2d},\mathbb{R}^d)$, are often not a good representation when the measure $\mu\in \mathcal{P}_{2,\nu}(\mathbb{R}^{2d})$ is not regular. Then, an appropriate extension to transference plan is needed, which leads to the concept of {\it extended Fr\'echet subdifferential}. Below we adapt this construction to our novel fibered space inspired by the treatment in classical Wasserstein space.

\begin{defi}[Extended fibered Fr\'echet subdifferential]\label{D-Frechet-subdifferential-extended}
Consider $\nu\in \mathcal{P}(\mathbb{R}^d)$ and define the sets of probability measures $\boldsymbol{L}^2_{\mu}(\mathbb{R}^{2d})$ and the norm $\Vert \cdot\Vert_{\boldsymbol{L}^2_{\mu}(\mathbb{R}^{2d})}$ as follows
\begin{align}
\boldsymbol{L}^2_{\mu}(\mathbb{R}^{2d})&:=\left\{\gamma\in \Gamma_\nu(\mu,\mu'):\,\mu'\in \mathcal{P}_{2,\nu}(\mathbb{R}^{2d})\right\},\label{E-fibered-subdifferential-extended-plans}\\
\Vert \gamma\Vert_{\boldsymbol{L}^2_{\mu}(\mathbb{R}^{2d})}^2&:=\int_{\mathbb{R}^{4d}}\vert x'\vert^2\,d\gamma(x,x',\omega,\omega'),\label{E-fibered-subdifferential-extended-plans-norm}
\end{align}
for each $\mu\in \mathcal{P}_{2,\nu}(\mathbb{R}^{2d})$ and $\gamma\in \boldsymbol{L}^2_{\mu}(\mathbb{R}^{2d})$. Set any functional $\mathcal{E}:\mathcal{P}_{2,\nu}(\mathbb{R}^{2d})\longrightarrow (-\infty,+\infty]$, any $\mu\in D(\mathcal{E})$ and any $\gamma\in\boldsymbol{L}^2_{\mu}(\mathbb{R}^{2d})$. We say that $\gamma$ belongs to the extended (fibered) Fr\'echet subdifferential of $\mathcal{E}$ at $\mu$, and we write $\gamma\in\boldsymbol{\partial}_{W_{2,\nu}}\mathcal{E}[\mu]$, when the following inequality holds
\begin{equation}\label{E-fibered-subdifferential-extended}
\mathcal{E}[\sigma]-\mathcal{E}[\mu]\geq \inf_{\eta\in \Gamma_{o,\nu}(\gamma,\sigma)}\int_{\mathbb{R}^{4d}}x'\cdot(x''-x)\,d\eta(x,x',x'',\omega,\omega',\omega'')+o(W_{2,\nu}(\mu,\sigma)),
\end{equation}
for every $\sigma\in D(\mathcal{E})$. Above we denote
\begin{equation}\label{E-fibered-subdifferential-extended-optimal-plans}
\Gamma_{o,\nu}(\gamma,\sigma):=\left\{\eta\in \Gamma_\nu(\mu,\mu',\sigma):\,\mu'\in \mathcal{P}_{2,\nu}(\mathbb{R}^{2d}),\, \pi_{(x,x',\omega,\omega')\#}\eta=\gamma,\,\pi_{(x,x'',\omega,\omega'')\#}\eta\in \Gamma_{o,\nu}(\mu,\sigma)\right\},
\end{equation}
and $\Gamma_\nu(\mu,\mu',\sigma)$ is the set of $\nu$-admissible $3$-plans in Definition \ref{D-admissible-plans-fibered}. By $\boldsymbol{\partial}^\circ_{W_{2,\nu}}\mathcal{E}[\mu]$ we denote the subset of $\boldsymbol{\partial}_{W_{2,\nu}}\mathcal{E}[\mu]$ with minimal $\boldsymbol{L}^2_\mu(\mathbb{R}^{2d})$-norm and we refer to it as the minimal extended (fibered) Fr\'echet subdifferential of $\mathcal{E}$ at $\mu$.
\end{defi}

As in the classical Wasserstein space, there is a natural relation between vectors in $L^2_\mu(\mathbb{R}^{2d},\mathbb{R}^d)$ and transference plans in $\boldsymbol{L}^2_\mu(\mathbb{R}^{2d})$, which we recall bellow.

\begin{pro}[Vectors vs transference plans]\label{P-vectors-vs-plans}
Consider any $\nu\in \mathcal{P}(\mathbb{R}^d)$ and $\mu\in \mathcal{P}_{2,\nu}(\mathbb{R}^{2d})$ and define $\mathfrak{i}_\mu:L^2_\mu(\mathbb{R}^{2d},\mathbb{R}^d)\longrightarrow \boldsymbol{L}^2_\mu(\mathbb{R}^{2d})$ and $\mathfrak{b}_\mu:\boldsymbol{L}^2_\mu(\mathbb{R}^{2d})\longrightarrow L^2_\mu(\mathbb{R}^{2d},\mathbb{R}^d)$ by
\begin{align}\label{E-vectors-vs-plans}
\begin{aligned}
&\mathfrak{i}_\mu[\boldsymbol{u}](x,x',\omega,\omega'):=((Id,\boldsymbol{u}(\cdot,\omega))_{\#}\mu^\omega)(x,x')\otimes \nu(\omega)\otimes \delta_\omega(\omega'),\\
&\mathfrak{b}_\mu[\gamma](x,\omega):=\int_{\mathbb{R}^d}x'\,d\gamma^{x,\omega}(x'), 
\end{aligned}
\end{align}
for any $\boldsymbol{u}\in L^2_\mu(\mathbb{R}^{2d},\mathbb{R}^d)$ and $\gamma\in \boldsymbol{L}^2_\mu(\mathbb{R}^{2d})$. Here, the Borel families $\{\mu^\omega\}_{\omega\in\mathbb{R}^d}\subseteq \mathcal{P}_2(\mathbb{R}^d)$ and $\{\gamma^{x,\omega}\}_{(x,\omega)\in\mathbb{R}^{2d}}\subseteq \mathcal{P}_2(\mathbb{R}^d)$ denote the disintegrations of $\mu$ and $\gamma$ according to
\begin{equation}\label{E-barycentric-disintegration}
\mu(x,\omega)=\mu^\omega(x)\otimes \nu(\omega),\quad \gamma(x,x',\omega,\omega')=\gamma^{x,\omega}(x')\otimes \mu^\omega(x)\otimes \nu(\omega)\otimes \delta_\omega(\omega').
\end{equation}
Then, the following properties hold true:
\begin{enumerate}[label=(\roman*)]
\item $\mathfrak{i}_\mu$ is an isometric embedding:
$$\Vert \mathfrak{i}_\mu[\boldsymbol{u}]\Vert_{\boldsymbol{L}^2_\mu(\mathbb{R}^{2d})}=\Vert \boldsymbol{u}\Vert_{L^2_\mu(\mathbb{R}^{2d},\mathbb{R}^d)},\quad \boldsymbol{u}\in L^2_\mu(\mathbb{R}^{2d},\mathbb{R}^d).$$
\item $\mathfrak{b}_\mu$ is non-expansive:
$$\Vert \mathfrak{b}_\mu[\gamma]\Vert_{L^2_\mu(\mathbb{R}^{2d},\mathbb{R}^d)}\leq \Vert \gamma\Vert_{\boldsymbol{L}^2_\mu(\mathbb{R}^{2d})},\quad \gamma\in \boldsymbol{L}^2_\mu(\mathbb{R}^{2d}).$$
\item $\mathfrak{b}_\mu$ is a left inverse of $\mathfrak{i}_\mu$:
$$\mathfrak{b}_\mu[\mathfrak{i}_\mu[\boldsymbol{u}]]=\boldsymbol{u},\quad \boldsymbol{u}\in L^2_\mu(\mathbb{R}^{2d},\mathbb{R}^d).$$
\end{enumerate}
\end{pro}

The proof is straightforward, so we omit it. The operator $\mathfrak{b}_\mu$ above then acts as a projection and is called the {\it barycentric projection} of plans into vectors. One of the main kindness is that it preserve elements in the (extended) Fr\'echet subdifferential as stated in the following result.

\begin{pro}[Barycentric projection of extended subdifferential]\label{P-fibered-Frechet-subdifferential-barycentric}
Consider $\nu\in \mathcal{P}(\mathbb{R}^d)$, $\mu\in \mathcal{P}_{2,\nu}(\mathbb{R}^{2d})$ and $\mathcal{E}:\mathcal{P}_{2,\nu}(\mathbb{R}^{2d})\longrightarrow (-\infty,+\infty]$. Then, the following property holds true:
$$\gamma\in \boldsymbol{\partial}_{W_{2,\nu}}\mathcal{E}[\mu]\quad \Longrightarrow\quad \mathfrak{b}_\mu[\gamma]\in \partial_{W_{2,\nu}}\mathcal{E}[\mu].$$
\end{pro}

\begin{proof}
Assume that $\gamma\in \boldsymbol{\partial}_{W_{2,\nu}}\mathcal{E}[\mu]$ and let us disintegrate like in \eqref{E-barycentric-disintegration}. Also, take any $\bar\gamma_o=\bar\gamma_o(x,x'',\omega,\omega'')\in \Gamma_{o,\nu}(\mu,\sigma)$ and define $\eta\in \mathcal{P}(\mathbb{R}^{6d})$ by disintegration as follows
$$\eta(x,x',x'',\omega,\omega',\omega''):=\gamma^{x,\omega}(x')\otimes \bar{\gamma}_o^\omega(x,x'')\otimes \nu(\omega)\otimes \delta_{\omega}(\omega')\otimes \delta_{\omega}(\omega'').$$
Note that $\eta\in \Gamma_{o,\nu}(\gamma,\sigma)$ because $\pi_{(x,x',\omega,\omega')\#}\eta =\gamma$ and $\pi_{(x,x'',\omega,\omega'')\#}\eta =\bar{\gamma}_o$. Since we have $\gamma\in \boldsymbol{\partial}_{W_{2,\nu}}\mathcal{E}[\mu]$, then using \eqref{E-fibered-subdifferential-extended} for the above $\eta$, along with the definition of the barycentric projection \eqref{E-vectors-vs-plans} imply that $\mathfrak{b}_\mu[\gamma]$ satisfies \eqref{E-fibered-subdifferential} in Definition \ref{D-Frechet-subdifferential}. Therefore, $\mathfrak{b}_\mu[\gamma]\in \partial_{W_{2,\nu}}\mathcal{E}[\mu]$.
\end{proof}

The core of this part will be to prove first the following analogue of Proposition \ref{P-slope-minimal-subdifferential} using transference plans $\gamma\in \boldsymbol{L}^2_\mu(\mathbb{R}^{2d})$, which extends the approach in \cite{AGS-08} to the fibered setting.

\begin{pro}[Metric slope vs minimal extended subdifferential]\label{P-slope-minimal-subdifferential-plans}
Consider $\nu\in \mathcal{P}(\mathbb{R}^d)$ and any functional $\mathcal{E}$ satisfying framework $\framework$. If $\mu \in D(\vert \partial \mathcal{E}\vert_{W_{2,\nu}})$, then $\boldsymbol{\partial}_{W_{2,\nu}} \mathcal{E}[\mu]\neq \emptyset$ and
\begin{equation}\label{E-slope-minimal-subdifferential-extended}
\vert \partial\mathcal{E}\vert_{W_{2,\nu}}[\mu]=\min\left\{\Vert \boldsymbol{\gamma}\Vert_{\boldsymbol{L}_\mu(\mathbb{R}^{2d})}:\,\gamma \in \boldsymbol{\partial}_{W_{2,\nu}}\mathcal{E}[\mu]\right\}.
\end{equation}
\end{pro}

When the functional $\mathcal{E}$ is defined over a Banach space, the usual proof exploits a reformulation of the local slope when $\mathcal{E}$ is $\lambda$-convex as a full supremum, along with a clever use of a geometric version of the Hahn-Banach, see \cite[Proposition 1.4.4]{AGS-08}. For $(\mathcal{P}_{2,\nu}(\mathbb{R}^{2d}),W_{2,\nu})$, the Hahn-Banach theorem is not applicable due to clear reasons. Inspired by \cite[Theorem 10.3.10]{AGS-08} for the case of functionals over the classical Wasserstein space, our method will exploit a different reformulation of the local slope $\vert\partial \mathcal{E}\vert_{W_{2,\nu}}$ in terms of the the Moreau-Yosida approximations of the functional $\mathcal{E}$. This approach is suitable when dealing with general metric spaces and gives a response beyond the linear structure that can be applied to a variety of Banach manifolds.

\begin{lem}[Moreau-Yosida approximation]\label{L-Moreau-Yosida}
Consider $\nu\in \mathcal{P}(\mathbb{R}^d)$, $\mathcal{E}:\mathcal{P}_{2,\nu}(\mathbb{R}^{2d})\longrightarrow (-\infty,+\infty]$ and consider the associated penalized energy functional $\Phi$ according to \eqref{E-penalized-energy} in Section \ref{sec:fibered-wasserstein}. Assume that $\mathcal{E}$ satisfies the hypothesis in framework $\framework$ ({\it cf.} Definition \ref{D-framework-F}), set $\mu_*\in \overline{D(\mathcal{E})}$ and define $\tau_*:=\frac{1}{\lambda^-}$. Then, we have the following properties:
\begin{enumerate}[label=(\roman*)]
\item (Existence and uniqueness of minimizers) For any $0<\tau<\tau_*$ there exists a unique minimizer $\mu_\tau\in \mathcal{P}_{2,\nu}(\mathbb{R}^{2d})$ of the functional $\Phi(\tau,\mu_*;\cdot)$.
\item (Convergence of minimizers) Take the minimizer $\mu_\tau\in \mathcal{P}_{2,\nu}(\mathbb{R}^{2d})$ above for any $0<\tau<\tau_*$. Then, we obtain the convergence $W_{2,\nu}(\mu_\tau,\mu_*)\rightarrow 0$ as $\tau\rightarrow 0$.
\end{enumerate}
\end{lem}

The proof is standard, it works over general metric spaces and it follows from Lemmas 4.1.1 and 3.1.2 in \cite{AGS-08}. More specifically, those results only require that $\mathcal{E}$ is proper, coercive, lower semicontinuous and the penalized energy functional $\Phi$ verifies the convexity property \eqref{E-convexity-penalized-energy} in Lemma \ref{L-convexity-penalized-energy}. Of course they all hold under the framework $\framework$. Then, we omit the proof.

\begin{lem}[Reformulation of the local slope]\label{L-Moreau-Yosida-slope}
Under the assumptions of Lemma \ref{L-Moreau-Yosida}, set any $\mu\in \overline{D(\mathcal{E})}$ and the minimizer $\mu_\tau\in \mathcal{P}_{2,\nu}(\mathbb{R}^{2d})$ of the penalized energy $\Phi(\tau,\mu;\cdot)$ for every $0<\tau<\tau_*$. Then, there exists a sequence $\{\tau_n\}_{\in\mathbb{N}}\subseteq (0,\tau_*)$ with $\tau_n\rightarrow 0$ such that
\begin{equation}\label{E-Moreau-Yosida-slope}
\vert \partial\mathcal{E}\vert_{W_{2,\nu}}^2[\mu]=\lim_{n\rightarrow\infty}\frac{W_{2,\nu}^2(\mu_{\tau_n},\mu)}{\tau_n^2}.
\end{equation}
\end{lem}

Again, the proof is standard over metric spaces and follows from \cite[Lemma 3.1.5]{AGS-08}. More specifically, the result only requires that $\mathcal{E}$ is proper, coercive, lower semicontinuous and the penalized energy $\Phi$ admits minimizers for a positive range $\tau\in (0,\tau_*)$. Of course they all hold under the framework $\framework$. Then, we omit the proof.

\medskip

\begin{proof}[Proof of Proposition \ref{P-slope-minimal-subdifferential-plans}]

~
\medskip

$\diamond$ {\sc Step 1}: Direct inequality in \eqref{E-slope-minimal-subdifferential-extended}.\\
Consider any $\gamma\in \boldsymbol{\partial}_{W_{2,\nu}}\mathcal{E}[\mu]$ and use Definition \ref{D-Frechet-subdifferential-extended} to find the inequality
\begin{align*}
\mathcal{E}[\mu]-\mathcal{E}[\sigma]&\leq \sup_{\eta\in \Gamma_{o,\nu}(\gamma,\sigma)}\int_{\mathbb{R}^{6d}}x'\cdot (x-x'')\,d\eta+o (W_{2,\nu}(\mu,\sigma))\\
&\leq \Vert \gamma\Vert_{\boldsymbol{L}^2_\mu(\mathbb{R}^{2d})} W_{2,\nu}(\mu,\sigma)+o(W_{2,\nu}(\mu,\sigma)),
\end{align*}
for any $\sigma\in \mathcal{P}_{2,\nu}(\mathbb{R}^{2d})$, where we have used the Cauchy--Schwarz inequality, the definition \eqref{E-fibered-subdifferential-extended-plans-norm} of the $\boldsymbol{L}^2_\mu(\mathbb{R}^{2d})$-norm, and the fact that $\pi_{(x,x'',\omega,\omega'')\#}\eta\in \Gamma_{o,\nu}(\mu,\sigma)$ by definition \eqref{E-fibered-subdifferential-extended-optimal-plans} of $\Gamma_{o,\nu}(\gamma,\sigma)$. Dividing by $W_{2,\nu}(\mu,\sigma)$, taking positive part and limits as $\sigma\rightarrow \mu$ in $W_{2,\nu}$ we have
$$\vert \partial \mathcal{E}\vert_{W_{2,\nu}}[\mu]\leq \Vert \gamma\Vert_{\boldsymbol{L}^2_\mu(\mathbb{R}^{2d})},$$
which ends this step. The next steps aim at deriving the converse inequality.

\medskip

$\diamond$ {\sc Step 2}: A minimizing family of plans $\gamma_\tau\in \boldsymbol{\partial}_{W_{2,\nu}}\mathcal{E}[\mu_\tau]$.\\
Let us consider the unique minimizer $\mu_\tau\in \mathcal{P}_{2,\nu}(\mathbb{R}^{2d})$ of the penalized energy functional $\Phi(\tau,\mu;\cdot)$ for every $\tau\in (0,\tau_*)$. By the optimality condition of $\mu_\tau$ we have
\begin{equation}\label{E-Moreau-Yosida-optimality}
\mathcal{E}[\sigma]-\mathcal{E}[\mu_{\tau}]\geq -\frac{1}{\tau}\int_{\mathbb{R}^d}\left(\frac{1}{2}W_2^2(\sigma^\omega,\mu^\omega)-\frac{1}{2}W_2^2(\mu_\tau^\omega,\mu^\omega)\right)\,d\nu(\omega),
\end{equation}
for each $\sigma\in \mathcal{P}_{2,\nu}(\mathbb{R}^{2d})$ and $\tau\in (0,\tau_*)$. Associated to each $\mu_\tau$, let us set a plan $\widehat{\gamma}_\tau\in \Gamma_{o,\nu}(\mu_\tau,\mu)$, and the following associated rescaled plan $\gamma_\tau(x,x',\omega,\omega'):=\gamma_\tau^\omega(x,x')\otimes \nu(\omega)\otimes \delta_\omega(\omega')$, 
where
$$\gamma_\tau^\omega:=(\pi_x,\frac{1}{\tau}(\pi_{x'}-\pi_x))_{\#}\widehat{\gamma}_\tau^\omega,$$
for $\nu$-a.e. $\omega\in \mathbb{R}^d$. Using the super-differentiability in \cite[Theorem 10.2.2]{AGS-08} applied at $\nu$-a.e. value $\omega\in \mathbb{R}^d$ of the integral in the right hand side of \eqref{E-Moreau-Yosida-optimality} for the quadratic Wasserstein distance $W_2$, we obtain the following inequality
\begin{equation}\label{E-Moreau-Yosida-optimality-2}
\mathcal{E}[\sigma]-\mathcal{E}[\mu_\tau]\geq \int_{\mathbb{R}^{6d}}x'\cdot (x''-x)\,d\eta_\tau(x,x',x'',\omega,\omega',\omega'') -\frac{1}{\tau}W_{2,\nu}^2(\mu_\tau,\sigma),
\end{equation}
for each $\sigma\in \mathcal{P}_{2,\nu}(\mathbb{R}^{2d})$ and $\tau\in (0,\tau_*)$, where $\eta_\tau$ is any probability measure in $\Gamma_{o,\nu}(\gamma_\tau,\sigma)$. In particular, this implies that $\gamma_\tau\in \boldsymbol{\partial}_{W_{2,\nu}}\mathcal{E}[\mu_\tau]$ for each $\tau\in (0,\tau_*)$. 

\medskip

$\diamond$ {\sc Step 3}: Reformulation of $\vert \partial \mathcal{E}\vert_{W_{2,\nu}}[\mu]$ in terms of $\{\gamma_\tau\}_{\tau\in (0,\tau_*)}$.\\
By definition of $\gamma_\tau\in \boldsymbol{L}^2_{\mu_\tau}(\mathbb{R}^{2d})$ note that we have the following chain of identities
$$\Vert \gamma_\tau\Vert_{\boldsymbol{L}^2_{\mu_\tau}(\mathbb{R}^{2d})}^2=\int_{\mathbb{R}^{4d}}\vert x'\vert^2\,d\gamma_\tau(x,x',\omega,\omega')=\int_{\mathbb{R}^{4d}}\frac{\vert x-x'\vert^2}{\tau^2}\,d\widehat{\gamma}_\tau(x,x',\omega,\omega')=\frac{W_{2,\nu}^2(\mu_\tau,\mu)}{\tau^2},$$
for each $\tau\in (0,\tau_*)$. Therefore, using Lemma \ref{L-Moreau-Yosida-slope} we know that there must exist a sequence $\{\tau_n\}_{n\in \mathbb{N}}\subseteq (0,\tau_*)$ with $\tau_n\rightarrow 0$ such that
\begin{equation}\label{E-approximate-fibered-minimal-subdifferential}
\vert \partial \mathcal{E}\vert_{W_{2,\nu}}[\mu]=\lim_{n\rightarrow\infty}\Vert \gamma_{\tau_n}\Vert_{\boldsymbol{L}^2_{\mu_{\tau_n}}(\mathbb{R}^{2d})}.
\end{equation}

\medskip

$\diamond$ {\sc Step 4}: Compactness of $\{\gamma_{\tau_n}\}_{n\in \mathbb{N}}$ and $\{\eta_{\tau_n}\}_{n\in \mathbb{N}}$.\\
First, let us note that the following convergence of moments take place
\begin{align*}
&\lim_{n\rightarrow \infty}\int_{\mathbb{R}^{4d}}\vert x\vert^2\,d\gamma_{\tau_n}(x,x',\omega,\omega')=\lim_{n \rightarrow \infty}W_{2,\nu}^2(\mu_{\tau_n},\delta_0(x)\otimes \nu(\omega))=W_{2,\nu}^2(\mu,\delta_0(x)\otimes \nu(\omega))<\infty,\\
&\lim_{n\rightarrow \infty}\int_{\mathbb{R}^{4d}}\vert x'\vert^2\,d\gamma_{\tau_n}(x,x',\omega,\omega')=\lim_{n\rightarrow \infty}\Vert \gamma_{\tau_n}\Vert_{\boldsymbol{L}^2_{\mu_{\tau_n}}}^2=\vert \partial\mathcal{E}\vert_{W_{2,\nu}}^2[\mu]<\infty.
\end{align*}
The first part of the claim follows from item $(ii)$ of Lemma \ref{L-Moreau-Yosida}, which implies that $\mu_{\tau_n}\rightarrow \mu$ in $W_{2,\nu}$, whilst the second part holds by \eqref{E-approximate-fibered-minimal-subdifferential} and the assumption $\mu \in D(\vert \partial\mathcal{E}\vert_{W_{2,\nu}})$. In particular, when regarded in $\mathcal{P}_{\hat{\nu}}(\mathbb{R}^{4d})$ with $\hat{\nu}(\omega,\omega'):=\nu(\omega)\otimes \delta_\omega(\omega')$ we obtain that the sequence $\{\gamma_{\tau_n}\}_{n\in \mathbb{N}}$ is $\tilde{\nu}$ uniformly tight. Hence, by the Prokhorov Theorem \ref{T-Prokhorov} there exists a subsequence (still denoted $\{\gamma_{\tau_n}\}_{n\in \mathbb{N}}$ for simplicity) and $\gamma\in \mathcal{P}_{\hat{\nu}}(\mathbb{R}^{4d})$ such that 
\begin{equation}\label{E-compactness-gamma}
\gamma_{\tau_n}\rightarrow \gamma\quad \mbox{narrowly in }\mathcal{P}_{\hat{\nu}}(\mathbb{R}^{4d}).
\end{equation}
Indeed, it is clear that $\gamma\in \boldsymbol{L}^2_\mu(\mathbb{R}^{2d})$ by the above uniform bounds of the quadratic moments with respect to $x$ and $x'$.

Second, for $\{\eta_{\tau_n}\}_{n\in \mathbb{N}}$ we shall argue as above. Specifically, using the above uniform bounds of the quadratic moments with respect to $x$ and $x'$ and
$$\lim_{n\rightarrow\infty}\int_{\mathbb{R}^{6d}}\vert x''\vert^2\,d\eta_{\tau_n}(x,x',x'',\omega,\omega',\omega'')=W_{2,\nu}(\sigma,\delta_0(x)\otimes \nu(\omega))<\infty,$$
implies that, when regarded in $\mathcal{P}_{2,\tilde{\nu}}(\mathbb{R}^{6d})$ with $\tilde{\nu}(\omega,\omega',\omega''):=\nu(\omega)\otimes \delta_\omega(\omega')\otimes \delta_\omega(\omega'')$, the sequence $\{\eta_{\tau_n}\}_{n\in \mathbb{N}}$ is uniformly tight. Then, there exists $\eta\in \mathcal{P}_{\tilde\nu}(\mathbb{R}^{6d})$ such that 
\begin{equation}\label{E-compactness-eta}
\eta_{\tau_n}\rightarrow \eta\quad \mbox{narrowly in }\mathcal{P}_{\tilde\nu}(\mathbb{R}^{6d}),
\end{equation}
up to a subsequence. Using all the above bounds of moments with respect to $x$,  $x'$ and $x''$, together with the fact that $\mu_{\tau_n}\to\mu$ in $W_{2,\nu}$ as well as $\gamma_{\tau_n}\to\gamma$ and $\eta_{\tau_n}\to\eta$ narrowly, we argue by the 
 stability of optimality in Proposition \ref{P-stability-optimal-plans} and conclude that $\eta\in \Gamma_{o,\nu}(\gamma,\sigma)$.

\medskip

$\diamond$ {\sc Step 5}: Variational inequality of \eqref{E-Moreau-Yosida-optimality-2} under $\lambda$-convexity.\\
Notice that \eqref{E-Moreau-Yosida-optimality-2} is not suitable to take the limit as $\tau\rightarrow 0$ because the second term in the right hand sides blows up. To circumvent this issue, we shall extend the variational reformulation in Proposition \ref{P-Frechet subdifferential-convex} for the Fr\'echet subdifferential to the extended Fr\'echet subdifferential consisting of plans thanks to the $\lambda$-convexity assumption on $\mathcal{E}$. Since this argument becomes tricky when tackling transference plans, we provide a sketch of the proof for the reader's convenience. 

Let us set $\eta_\tau\in \Gamma_{o,\nu}(\gamma_\tau,\sigma)$ for every $\tau\in (0,\tau_*)$ such that $\mathcal{E}$ is $\lambda$-convex along the geodesic $\theta\in [0,1]\mapsto \mu_{\tau,\theta}\in \mathcal{P}_{2,\nu}(\mathbb{R}^{2d})$ joining $\mu_\tau$ to $\sigma$ associated to the optimal $\nu$-admissible plan $\bar\gamma_\tau:=\pi_{(x,x'',\omega,\omega'')\#}\eta_\tau\in \Gamma_{o,\nu}(\mu_\tau,\sigma)$. With symbols, this means that
$$\mu_{\tau,\theta}^\omega=((1-\theta)\pi_{x}+\theta\pi_{x''})_\# \bar{\gamma}_\tau^\omega,\quad \theta\in [0,1],$$
for $\nu$-a.e. $\omega\in \mathbb{R}^d$. Of course, $\eta_\tau$ exists and can be built as the composition of $\bar \gamma_\tau$ (which exists by the $\lambda$-convexity of $\mathcal{E}$) and $\gamma_\tau$ (which was built in the previous step) thanks to the compatibility condition $\pi_{(x,\omega)\#}\bar\gamma_\tau=\mu_\tau=\pi_{(x,\omega)\#}\gamma_\tau$. We now build the following plans $\eta_{\tau,\theta}\in \Gamma_{o,\nu}(\gamma_\tau,\mu_{\tau,\theta})$
$$\eta_{\tau,\theta}:=(\pi_x,\pi_{x'},(1-\theta)\pi_x+\theta\pi_{x''},\pi_\omega,\pi_{\omega'},\pi_{\omega''})_{\#}\eta_\tau,$$
for any $\theta\in [0,1]$. On the one hand, by applying \eqref{E-Moreau-Yosida-optimality-2} with $\mu_{\tau,\theta}$ playing the role of $\sigma$ yields
\begin{align*}
\mathcal{E}[\mu_{\tau,\theta}]-\mathcal{E}[\mu_\tau]&\geq \int_{\mathbb{R}^{6d}}x'\cdot (x''-x)\,d\eta_{\tau,\theta}(x,x',x'',\omega,\omega',\omega'') -\frac{1}{\tau}W_{2,\nu}^2(\mu_\tau,\mu_{\tau,\theta})\\
&=\theta\int_{\mathbb{R}^{6d}}x'\cdot(x''-x)\,d\eta_\tau(x,x',x'',\omega,\omega',\omega'')-\frac{\theta^2}{\tau}W_{2,\nu}^2(\mu_\tau,\sigma),
\end{align*}
for any $\theta\in [0,1]$ and $\tau\in (0,\tau_*)$. On the other hand, using the $\lambda$-convexity property of $\mathcal{E}$ implies
$$\mathcal{E}[\mu_{\tau,\theta}]-\mathcal{E}[\mu_\tau]\leq \theta (\mathcal{E}[\sigma]-\mathcal{E}[\mu])-\frac{\lambda}{2}\theta(1-\theta)W_{2,\nu}^2(\mu_\tau,\sigma).$$
Joining the above two inequalities, dividing by $\theta$ and passing to the limit as $\theta\rightarrow 0$ imply
\begin{equation}\label{E-Moreau-Yosida-optimality-3}
\mathcal{E}[\sigma]-\mathcal{E}[\mu_\tau]\geq \int_{\mathbb{R}^{6d}}x'\cdot (x''-x)\,d\eta_\tau(x,x',x'',\omega,\omega',\omega'')+\frac{\lambda}{2}W_{2,\nu}^2(\mu_\tau,\sigma),
\end{equation}
for any $\sigma\in \mathcal{P}_{2,\nu}(\mathbb{R}^{2d})$ and $\tau\in (0,\tau_*)$.

\medskip

$\diamond$ {\sc Step 6}: Conclusion of \eqref{E-slope-minimal-subdifferential-extended}.\\
First, we prove that $\gamma\in \boldsymbol{\partial}_{W_{2,\nu}}\mathcal{E}[\mu]$. We take any sequence $\tau_n\to 0^+$ and pass to the limit as $n\rightarrow \infty$ in the variational reformulation \eqref{E-Moreau-Yosida-optimality-3}. Note that since $\mu_{\tau_n}\rightarrow \mu$ in $W_{2,\nu}$ and $\mathcal{E}$ is lower semicontinuous, then it is clear that we can pass to the limit in the left hand side of \eqref{E-Moreau-Yosida-optimality-3} and also in the second term of the right hand side. The delicate point is precisely the first term in the right hand side, which also involves the limit of the $\eta_{\tau_n}$. Let us define the function $g:\mathbb{R}^{6d}\longrightarrow\mathbb{R}$ given by 
$$g(x,x',x'',\omega,\omega',\omega''):=x'\cdot (x''-x).$$
We claim that $g$ is uniformly integrable with respect to $\{\eta_{\tau_n}\}_{n\in \mathbb{N}}$. Note that under the claim, Lemma 5.1.7 in \cite{AGS-08} directly allows passing to the limit as $n\rightarrow \infty$ because $\eta_{\tau_n}\rightarrow \eta$ narrowly in $\mathcal{P}(\mathbb{R}^{6d})$ by virtue of \eqref{E-compactness-eta}. This of course shows that $\gamma\in \boldsymbol{\partial}_{W_{2,\nu}}\mathcal{E}[\mu]$ because $\eta\in \Gamma_{o,\nu}(\gamma,\sigma)$ and $\sigma\in \mathcal{P}_{2,\nu}(\mathbb{R}^{2d})$ is an arbitrary measure. 

To prove the claim, we note that $g$ can be controlled as $\vert g\vert\leq g_1+g_2$ by the functions
$$g_1(x,x',x'',\omega,\omega',\omega'')=\vert x'\vert \vert x\vert,\quad g_2(x,x',x'',\omega,\omega',\omega'')=\vert x'\vert \vert x''\vert.$$
Then, we will show that both $g_1$ and $g_2$ are uniformly integrable with respect to $\{\eta_{\tau_n}\}_{n\in \mathbb{N}}$. For any $R,\tilde R>0$ we have
\begin{align*}
\int_{g_1\geq R}g_1\,d\eta_{\tau_n}&\leq \frac{\tilde R^2}{R}\Vert\gamma_{\tau_n} \Vert^2_{\boldsymbol{L}^2_{\mu_{\tau_n}}(\mathbb{R}^{2d})}+\Vert\gamma_{\tau_n} \Vert_{\boldsymbol{L}^2_{\mu_{\tau_n}}(\mathbb{R}^{2d})}\left(\int_{\vert x\vert\geq \tilde R}\vert x\vert^2\,d\mu_{\tau_n}(x,\omega)\right)^{1/2},\\
\int_{g_2\geq R}g_2\,d\eta_{\tau_n}&\leq \frac{\tilde R^2}{R}\Vert\gamma_{\tau_n} \Vert^2_{\boldsymbol{L}^2_{\mu_{\tau_n}}(\mathbb{R}^{2d})}+\Vert\gamma_{\tau_n} \Vert_{\boldsymbol{L}^2_{\mu_{\tau_n}}(\mathbb{R}^{2d})}\left(\int_{\vert x''\vert\geq \tilde R}\vert x''\vert^2\,d\sigma(x'',\omega'')\right)^{1/2}.
\end{align*}
Since $\sup_{n\in \mathbb{N}}\Vert\gamma_{\tau_n} \Vert^2_{\boldsymbol{L}^2_{\mu_{\tau_n}}(\mathbb{R}^{2d})}<\infty$ by \eqref{E-approximate-fibered-minimal-subdifferential}, then taking $\sup$ with respect to $n$ and $\limsup$ when $R\rightarrow\infty$ in the inequalities above we have
\begin{align*}
\limsup_{R\rightarrow \infty}\sup_{n\in \mathbb{N}}\int_{g_1\geq R}g_1\,d\eta_{\tau_n}&\lesssim \sup_{n\in \mathbb{N}}\left(\int_{\vert x\vert\geq \tilde R}\vert x\vert^2\,d\mu_{\tau_n}(x,\omega)\right)^{1/2},\\
\limsup_{R\rightarrow \infty}\sup_{n\in \mathbb{N}}\int_{g_2\geq R}g_2\,d\eta_{\tau_n}&\lesssim \left(\int_{\vert x\vert\geq \tilde R}\vert x''\vert^2\,d\sigma(x'',\omega'')\right)^{1/2},
\end{align*}
for any $\tilde{R}>0$. Recalling that $W_{2,\nu}(\mu_{\tau_n},\mu)\rightarrow 0$, using Remark \ref{R-characterization-convergence-W2nu} along with Lemma 5.1.7 in \cite{AGS-08} again, we conclude that the right hand side above must vanish by taking $\limsup$ when $\tilde{R}\rightarrow\infty$ and this ends the proof of the claim.

Finally, let us show that $\gamma\in \boldsymbol{\partial}_{W_{2,\nu}}\mathcal{E}[\mu]$ is indeed a minimizer in \eqref{E-slope-minimal-subdifferential-extended}. To this end, note that $(x,x',\omega,\omega')\in \mathbb{R}^{4d}\mapsto \vert x'\vert^2$ is lower semicontinuous (indeed continuous). Since $\gamma_{\tau_n}\rightarrow \gamma$ narrowly in $\mathcal{P}(\mathbb{R}^{4d})$ by \eqref{E-compactness-gamma}, then the representation \eqref{E-approximate-fibered-minimal-subdifferential} of the slope implies
$$\Vert \gamma\Vert_{\boldsymbol{L}^2_{\mu}(\mathbb{R}^{2d})}\leq \liminf_{n\rightarrow \infty}\Vert \gamma_{\tau_n}\Vert_{\boldsymbol{L}^2_{\mu_{\tau_n}}(\mathbb{R}^{2d})}=\lim_{n\rightarrow \infty}\Vert \gamma_{\tau_n}\Vert_{\boldsymbol{L}^2_{\mu_{\tau_n}}(\mathbb{R}^{2d})}=\vert \partial \mathcal{E}\vert_{W_{2,\nu}}[\mu],$$
and this concludes the proof.
\end{proof}

Then combining Proposition \ref{E-slope-minimal-subdifferential-extended} with the stability of Fr\'echet subdifferentials under barycentric projections in Proposition \ref{P-fibered-Frechet-subdifferential-barycentric} allow proving Proposition \ref{E-slope-minimal-subdifferential}.

\begin{proof}[Proof of Proposition \ref{E-slope-minimal-subdifferential}]
Take any $\gamma\in \boldsymbol{\partial}_{W_{2,\nu}}\mathcal{E}[\mu]$ minimizing \eqref{E-slope-minimal-subdifferential-extended} as in Proposition \ref{P-slope-minimal-subdifferential-plans} and compute its barycentric projection $\boldsymbol{u}:=\mathfrak{b}_\mu[\gamma]\in L^2_\mu(\mathbb{R}^{2d},\mathbb{R}^d)$ through Proposition \ref{P-vectors-vs-plans}. By Proposition \ref{P-fibered-Frechet-subdifferential-barycentric} we infer that $\boldsymbol{u}\in \partial_{W_{2,\nu}}\mathcal{E}[\mu]$. In addition, item $(ii)$ in Proposition \ref{P-vectors-vs-plans} and the optimality condition of $\gamma$ imply
$$\Vert \boldsymbol{u}\Vert_{L^2_\mu(\mathbb{R}^{2d},\mathbb{R}^d)}\leq \Vert \gamma\Vert_{\boldsymbol{L}^2_\mu(\mathbb{R}^{2d})}= \vert \partial\mathcal{E}\vert_{W_{2,\nu}}[\mu].$$
Hence, we end the proof as the converse inequality follows from the same argument as in {\sc Step 1} of the proof of Proposition \ref{P-slope-minimal-subdifferential-plans}.
\end{proof}

\section{Convexity properties of potential $W$}\label{Appendix-convexity-W}

In this section, we summarize the main convexity properties of the potential function $W$ in \eqref{W} appearing in the Kuramoto-type equation \eqref{kurakk} of Section \ref{maingoal3}.

\begin{lem}\label{C-convexity}
The following equivalent statements hold true
\begin{enumerate}[label=(\roman*)]
    \item {\it (First-order condition)}
    $$(\nabla W(x)-\nabla W(y))\cdot (x-y)\geq \Lambda_1(x,y,\alpha)\vert x-y\vert^2,$$
    for every $x,y\in \mathbb{R}^d$.
    \item {\it (Second-order condition)}
    $$v^\top\cdot D^2W(x)\cdot v\geq \Lambda_2(x,\alpha)\vert v\vert^2,$$
    for any $x\in \mathbb{R}^d$ and $v\in \mathbb{R}^d$.
\end{enumerate}
Here, the functions $\Lambda_1$ and $\Lambda_2$ take the form
$$
\Lambda_1(x,y,\alpha):=\int_0^1 \phi(\vert (1-\tau)x+\tau y\vert)\,d\tau,\quad \Lambda_2(x,\alpha):=\phi(\vert x\vert).
$$
Since $\Lambda_1,\Lambda_2\geq 0$, the potential $W$ is in particular convex.
\end{lem}

\begin{proof}
Taking derivatives in \eqref{W} we obtain
\begin{equation}\label{E-kernel-1order-jacobian}
D^2 W(x)=\frac{1}{1-\alpha}\phi(\vert x\vert)\left(I-\alpha\frac{x}{\vert x\vert}\otimes \frac{x}{\vert x\vert}\right), \quad x\in \mathbb{R}^d\setminus\{0\}.
\end{equation}
Then, item $(ii)$ follows from the Cauchy--Schwartz inequality. Item $(i)$ follows from item $(ii)$ and a first order Taylor expansion.
\end{proof}

Consequently, using the $0$-convexity of $W$ and the Cauchy-Swartz inequality we obtain one-sided Lipschitz continuity of the velocity field associated with equation \eqref{kurakk}. 

\begin{cor}\label{C-sided-Lipschitz}
Consider any $\nu\in \mathcal{P}(\mathbb{R}^d)$, $\mu\in \mathcal{P}_{2,\nu}(\mathbb{R}^{2d})$, and let $\boldsymbol{v}[\mu]=(\boldsymbol{u}[\mu],\boldsymbol{0})$ be the velocity field of the Kuramoto-type equation \eqref{kurakk}, that is,
$$\boldsymbol{u}[\mu](x,\omega):=\omega-K\int_{\RR^{2d}}\nabla W(x-x')\,d\mu(x',\omega').$$
Then, $\boldsymbol{v}[\mu]$ verifies the following conditions:
\begin{enumerate}[label=(\roman*)]
\item (Continuous) $\boldsymbol{v}[\mu]\in C(\mathbb{R}^{2d},\mathbb{R}^d)$.
\item (One-sided Lipschitz) We have
$$(\boldsymbol{v}[\mu](x,\omega)-\boldsymbol{v}[\mu](x',\omega'))\cdot((x,\omega)-(x',\omega'))\leq \frac{1}{2}\vert (x,\omega)-(x',\omega')\vert^2,$$
for any $x,x',\omega,\omega'\in \mathbb{R}^d$.
\item (Sublinear growth) We have
$$\vert \boldsymbol{v}[\mu](x,\omega)\vert\leq \frac{K}{1-\alpha}M_2(\mu)^{\frac{1-\alpha}{2}}+\frac{K}{1-\alpha}\vert x\vert^{1-\alpha}+\vert \omega\vert,$$
for any $x,\omega\in \mathbb{R}^d$, with $M_2(\mu):=\int_{\mathbb{R}^{2d}}|x|^2\,d\mu(x,\omega).$
\end{enumerate}
\end{cor}

We remark that this is the key regularity for the well-posedness of a unique global flow of $\boldsymbol{v}[\mu]$, and allows solving \eqref{kurakk} via the method of characteristics, see \cite{PP-22-2-arxiv,P-19-arxiv}.

\medskip

\subsection*{Acknowledgments}
This work has been supported by the Polish National Science Centre grant No 2018/31/D/ST1/02313 (SONATA) (JP), by the European Union's Horizon Europe research and innovation program under the Marie Sk\l odowska-Curie grant agreement No 101064402 (DP), and partially by the European Research Council under the European Union's Horizon 2020 research and innovation program grant agreement No 865711, the State Research Agency (SRA) of the Spanish Ministry of Science and Innovation and European Regional Development Fund (ERDF), project PID2022-137228OB-I00, and by Modeling Nature Research Unit, project QUAL21-011. We are also grateful to D. Lacker, J. Morales and F. Santambrogio for the fruitful discussions.

\subsection*{Data Availability} 
Data sharing not applicable to this article as no datasets were generated or analysed during the current study

\bibliographystyle{amsplain} 
\bibliography{PP-23}

\end{document}